\documentclass[10pt]{article}

\usepackage{amsmath}
\usepackage{amsthm}
\usepackage{amssymb}
\usepackage{enumitem}
\usepackage{multicol}
\usepackage[active]{srcltx}
\usepackage{algorithm}
\usepackage{array}
\usepackage[dvipsnames,svgnames]{xcolor}
\usepackage{tikz}
\usetikzlibrary{positioning,arrows.meta}
\usepackage{graphicx}
\usepackage{subcaption}
\usepackage{pgfplots}
\usepackage{xcolor}
\usepackage{cancel}
\usepackage{hyperref}
\usepackage[top=35mm, bottom=25mm, left=20mm, right=20mm]{geometry}
\usepackage{comment}

\providecommand{\U}[1]{\protect \rule{.1in}{.1in}}
\newtheorem{theorem}{Theorem}

\newtheorem{corollary}[theorem]{Corollary}
\newtheorem{fact}[theorem]{Fact}

\newtheorem{definition}[theorem]{Definition}
\newtheorem{example}[theorem]{Example}

\newtheorem{lemma}[theorem]{Lemma}

\newtheorem{proposition}[theorem]{Proposition}
\newtheorem{remark}[theorem]{Remark}

\newtheorem{assumption}[theorem]{Assumption}

\newcommand\R{\mathbb{R}}
\newcommand\Rinf{\overline{\mathbb{R}}}
\newcommand\dom[1]{ \bs{{\rm dom}}(#1)} 
\newcommand\dist{ \bs{{\rm dist}}} 
\newcommand\gh{\psi} 

\newcommand\prox[3]{\bs{{\rm Prox}}_{#2#1}^{#3}}

\newcommand\ov[1]{\overline{#1}}
\newcommand\mb{\mathbf{B}}
\newcommand\bs[1]{\boldsymbol{#1}}
\newcommand\argmin[1]{\bs{\arg\min}_{#1}}
\newcommand\argmint[1]{\mathop{\bs{\arg\min}}\limits_{#1}}
\newcommand\Nz{\mathbb{N}_0}

\newcommand{\ReLU}{\operatorname{ReLU}}
\newcommand{\RelErr}{\operatorname{RelErr}}
\newcommand\Fix{{\bs{\rm Fix}}}
\newcommand\inft[1]{\mathop{\bs{\inf}}\limits_{#1}}
\newcommand\liminft[1]{\mathop{\bs{\liminf}}\limits_{#1}}
\newcommand\maxt[1]{\mathop{\bs{\max}}\limits_{#1}}
\let\oldliminf\liminf
\renewcommand{\liminf}{\bs{\oldliminf}}
\let\oldlimsup\limsup
\renewcommand{\limsup}{\bs{\oldlimsup}}
\let\oldmax\max
\renewcommand{\max}{\bs{\oldmax}}
\let\oldmin\min
\renewcommand{\min}{\bs{\oldmin}}
\let\oldsup\sup
\renewcommand{\sup}{\bs{\oldsup}}
\let\oldinf\inf
\renewcommand{\inf}{\bs{\oldinf}}
\let\oldlim\lim
\renewcommand{\lim}{\bs{\oldlim}}

\title{{\bf Quasar-Convex Optimization: Fundamental Properties and High-Order Proximal-Point Methods}}
\author{Masoud Ahookhosh\thanks{Department of Mathematics, University of Antwerp, Antwerp, Belgium. Email: masoud.ahookhosh@uantwerp.be}, 
\and
Jos\'e M. M. de Brito \thanks{School of Science, Great Bay University; Great Bay Institute for Advanced Study,
Dongguan 523000, Guangdong Province,  People's Republic of China; and Instituto Federal do Piau\'{\i}, S\~ao Raimundo Nonato, Piau\'{\i}, Brazil. E-mail: jose.brito@ifpi.edu.br}
\and
Alireza Kabgani\thanks{Department of Mathematics, University of Antwerp, Antwerp, Belgium. Email: alireza.kabgani@uantwerp.be}
\and  
Felipe Lara\thanks{Instituto de Alta Investigación (IAI), Universidad de Tarapacá, Arica, Chile. Email: flarao@academicos.uta.cl}
\and
Jinyun Yuan\thanks{School of Science, Great Bay University; Great Bay Institute for Advanced Study,
Dongguan 523000, Guangdong Province,  People's Republic of China. Email: yuanjy570522@126.com}
}

\begin{document}


\maketitle

\begin{abstract}
We study the optimization of (strongly) quasar-convex functions, a class that arises naturally in many machine learning and data science applications due to its favorable properties. The fundamental properties of this class are first developed, including its stability under standard calculus operations, growth conditions, and the absence of spurious critical points, which together imply a benign global geometry with no saddle points. Motivated by these properties, a class of proximal-point algorithms (HiPPA) with high-order regularization of order $p>1$ is introduced. Conditions are identified under which the iterates converge globally to minimizers, and a unified convergence analysis is provided with explicit rates and iteration complexity bounds under appropriate regularity assumptions. The results reveal a sharp transition in behavior with respect to the order $p$: for $p\in(1,2)$, the method achieves local linear convergence with complexity $\mathcal{O}(\log(\varepsilon^{-1}))$ when initialized sufficiently close to a minimizer; for $p=2$, it attains global linear convergence with the same complexity; and for $p>2$, it exhibits superlinear convergence with complexity $\mathcal{O}(\log\log(\varepsilon^{-1}))$, where $\varepsilon>0$ denotes the target accuracy. The theory is complemented with preliminary numerical experiments on selected machine learning problems, which illustrate the effectiveness of the proposed methods and are consistent with the theoretical findings.

\noindent 

{\small}

\medskip


\medskip

\end{abstract}

\section{Introduction}\label{sec:intro}
This paper addresses the optimization problem
\begin{equation}\label{eq:mainproblem}
\mathop{\bs\min}\limits_{x\in \R^n}\ h(x),
\end{equation}
subject to the following basic assumptions. 
\begin{assumption}[Basic assumptions]\label{ass:basic} In problem \eqref{eq:mainproblem}, we assume that
\begin{enumerate}[label=(\textbf{\alph*}), font=\normalfont\bfseries, leftmargin=0.7cm]
\item  $h: \R^n\to\Rinf:=\R\cup \{+\infty\}$ is a proper, lower semicontinuous (lsc), (strongly) quasar-convex (see Definition~\ref{def:quasar}), and is possibly nonsmooth;
\item the set of minimizers is nonempty, with the corresponding minimal value $h^*$.
\end{enumerate}
\end{assumption}
\noindent The class of {\it (strongly) quasar-convex functions} arises in the modeling of a wide range of problems across several application domains, including machine learning \cite{farzin2025minimisation,wang2023continuized}, empirical risk minimization \cite{lee2016optimizing}, learning linear dynamical systems \cite{Hardt2018Gradient}, phase retrieval \cite{wang2023continuized}, and beyond. Although several studies have been carried out on basic properties of (strongly) quasar-convex functions in the smooth setting (e.g., \cite{guminov2023accelerated,Hardt2018Gradient,HADR,Hinder2020Near,wang2023continuized}), relatively little attention has been given to the nonsmooth setting (e.g., \cite{deBrito2025extending}). This motivates a more detailed study of this class of functions.

As a general nonconvex optimization problem, there are several methods to address the problem \eqref{eq:mainproblem}, however, among all other methods, we are interested in the high-order proximal-point method (HiPPA) (see Algorithm~\ref{ppa:sqcx:2}) given by
\begin{equation}\label{eq:hippa}\tag{HiPPA}
x^{k+1}\in \prox{h}{\beta_k}{p} (x^k):=\argmint{y\in \R^n} \left(h(y)+\frac{1}{p\beta_k}\Vert x^k- y\Vert^p\right),
\end{equation}
for $p>1$ and $\beta_k\geq \beta_{\bs\min}>0$. We note that this algorithm reduces to the classical proximal-point algorithm (PPA) for $p=2$ as a special case; see, e.g., \cite{martinet1970breve,martinet1972determination,rockafellar1976monotone}. Due to its notable theoretical and computational properties, such as simplicity and low memory requirements, PPA has garnered considerable attention over the past few decades in both convex and nonconvex settings; see, e.g., \cite{Bauschke17,Guler92,KecisThibault15,Parikh14,Poliquin96,rockafellar1976monotone}.

While PPA converges to a global minimizer in convex settings, in the nonconvex case it can, in general, be guaranteed to converge only to a \emph{stationary point}, which may correspond to either a local minimizer or a saddle point. In recent years, there has been increasing interest in developing optimization methods capable of escaping saddle points; see, e.g., \cite{bodard2025second,Davis2022Escaping,davis2022proximal,lee2019first,vlatakis2019efficiently}. This line of research is closely related to the identification of benign optimization landscapes (i.e., all local minimizers are also global minimizers) in many structured problems that arise across diverse application domains; see, e.g., \cite{Josz2023Certifying,ma2022blessing,McRae2024Benign,Vidal_Zhu_Haeffele_2022} and references therein. As an alternative strategy, one can formulate problems using special function classes, such as convex or strongly quasiconvex functions, which inherently preclude the presence of saddle points; cf. \cite{aujol2024fista,grimmer2025some,iusem2018second,Iusem2022proximal,Iusem20204two-step,Kabgani2023,Lara2022strongly}. This motivates further investigation of generalized convexity classes that exclude saddle points and guarantee the convergence to local or global minimizers, which is our main motivation to study the class of quasar-convex functions and investigate its fundamental properties, including calculus, growth conditions, and nonconvex landscape. 


Global convergence ensures that iterates converge from arbitrary initial points, whereas the convergence rate characterizes how rapidly this convergence occurs. As such, convergence rates play a central role in the analysis of optimization methods for problems of the form \eqref{eq:mainproblem}, particularly in high- and large-scale settings. For the classical PPA (scheme \eqref{eq:hippa} with $p=2$), sublinear convergence is typical \cite{Gu2020Tight,Guler9291Convergence,Luque19849Asymptotic}, with linear or even superlinear rates under strong convexity \cite{rockafellar1976monotone}. Extending such guarantees to broader classes, including uniformly convex or strongly quasiconvex functions, is challenging, though PPA has been shown to converge to global minimizers in the latter case \cite[Theorem 10]{Lara2022strongly}. Recent work has further investigated the convergence rates and complexity of high-order proximal-point methods in both convex and nonconvex settings \cite{Ahookhosh23, Ahookhosh24, Kabgani24itsopt,Kabgani25First,Kabgani25itsdeal,Kabgani2025Moreau,nesterov2021inexact,Nesterov2023a}. Identifying function classes for which HiPPA guarantees global convergence to a global minimizer with fast rates and complexity has not been studied broadly, which further motivated the present study.

\vspace{-2mm}
\subsection{{\bf Contribution}}\label{subsec:contribution}
Our contributions can be summarized as follows:
\begin{description}
 \item[(i)] (\textbf{Fundamental properties of (strongly) quasar-convex functions}) We characterize the class of strongly quasar-convex functions and establish its associated calculus, encompassing addition, scalar multiplication, and composition. We also examine coercivity properties within this class. In particular, we study the landscape of this class of functions and especially show that every local minimizer is global, thus excluding the presence of spurious local minimizers and saddle points. Furthermore, we derive first-order characterizations and corresponding Polyak--{\L}ojasiewicz (PL) inequality, growth conditions, and error bounds for optimization problems in this framework.
    \item[(ii)] (\textbf{Convergence and complexity analysis of HiPPA}) We begin by presenting a convergence analysis of the sequence generated by HiPPA toward global minimizers of \eqref{eq:mainproblem}. Specifically, if $\gamma=0$ (i.e., quasar-convex case), then, for $p\in (1,2)$, HiPPA attains the convergence rate $\mathcal{O}(k^{1-p})$ and the complexity $\mathcal{O}(\varepsilon^{-1/(p-1)})$, for $p=2$, it gets the convergence rate $\mathcal{O}(k^{-1})$ and the complexity $\mathcal{O}(\varepsilon^{-1})$, and for $p>2$, it achieves the convergence rate $\mathcal{O}(k^{-p/2})$ and the complexity $\mathcal{O}(\varepsilon^{-2/p})$. More specifically, if $\gamma>0$ (i.e., strongly quasar-convex case), then, for $p\in (1,2)$, HiPPA exhibits linear convergence with complexity $\mathcal{O}(\log(\varepsilon^{-1}))$ if it starts within a ball around the minimizer, for $p=2$, it attains linear convergence with the complexity $\mathcal{O}(\log(\varepsilon^{-1}))$, and for $p>2$, it achieves superlinear convergence with the complexity $\mathcal{O}(\log(\log(\varepsilon^{-1})))$. To the best of our knowledge, this constitutes the first comprehensive analysis of convergence rates and complexity of HiPPA in the context of strongly quasar-convex optimization.
    \item[(iii)] (\textbf{Applications}) We apply the HiPPA algorithm to two practical problems: (a) generalized linear models (GLMs) and (b) robust multi-task regression (RMTR). We show, in particular, that RMTR is nonsmooth and strongly quasar-convex, yet not strongly star-convex. To the best of our knowledge, HiPPA is the first method equipped with a rigorous convergence theory for RMTR.
\end{description}

\vspace{-2mm}
\subsection{{\bf Organization}}\label{subsec:Organization}
The remainder of this paper is structured as follows. We provide essential preliminaries and notations in Section~\ref{sec:preliminaries}. In Section~\ref{sec:03}, we study the fundamental properties of the class of nonsmooth quasar-convex functions. In Section~\ref{sec:04}, we study the convergence, rates, and complexity of the sequence generated by HiPPA. In Section~\ref{sec:glm_numerical}, we report our promising numerical results on a generalized linear model (GLM) and robust multi-task regression (RMTR). Finally, Section~\ref{sec:conclusion} delivers our concluding remarks.


\section{Preliminaries and notation}\label{sec:preliminaries}
Throughout this paper, $\R^n$ denotes the $n$-dimensional \textit{Euclidean space} equipped with the \textit{inner product} $\langle \cdot,\cdot \rangle$ and the associated norm $\|\cdot\| := \sqrt{\langle \cdot,\cdot \rangle}$. 
The \textit{open ball} centered at $\ov{x}\in \R^n$ with radius $r>0$ 
is denoted by $\mb(\ov{x}, r)$.
The set of \textit{nonnegative integers} is denoted by $\Nz := \mathbb{N}\cup\{0\}$. 
Given a matrix $A\in\R^{m\times n}$, we denote its range, null space, and smallest nonzero singular value by $\mathcal{R}(A)$, $\mathcal{N}(A)$, and $\sigma_{\min}(A)$, respectively.
For any $x,y,z\in\R^n$, the following identity will be used repeatedly: 
\begin{align}
 & ~~~~~ \langle x - z, y - x \rangle= \frac{1}{2} \lVert z - y \rVert^{2} -
 \frac{1}{2} \lVert x - z \rVert^{2} - \frac{1}{2} \lVert y - x \rVert^{2}.
 \label{3:points} 
\end{align}

Let $h: \mathbb{R}^{n} \rightarrow \overline{\mathbb{R}} := \mathbb{R}\cup\{+\infty\}$ be an extended-valued function. 
The \textit{effective domain} of $h$ is defined by
$\dom{h} := \{x \in \mathbb{R}^{n} : h(x) < +\infty\}$.
The function $h$ is said to be \textit{proper} if $\dom{h} \neq \emptyset$. 
For $\lambda \in \mathbb{R}$, the \textit{sublevel set} of $h$ at level $\lambda$ is
$S_{\lambda}(h) := \{x \in \mathbb{R}^{n} : h(x) \le \lambda\}$.
We denote by $\argmin{\mathbb{R}^{n}} h$ the set of minimizers of $h$.
A function $h$ is \textit{lower semicontinuous} (lsc henceforth) at $\ov{x} \in \mathbb{R}^{n}$ if for every sequence $\{x^k\}_{k\in\Nz} \subseteq \mathbb{R}^{n}$ with $x^k \to \ov{x}$, one has
$h(\ov{x}) \leq \liminf_{k \to +\infty} h(x^k)$. 
A point $\widehat{x}\in\R^n$ is called a \textit{cluster point} of a sequence $\{x^k\}_{k\in\Nz}$ if there exists an infinite subset $J\subseteq\Nz$ such that $x^j\to \widehat{x}$ as $j\to \infty$ with $j\in J$.

Let us recall the Clarke directional derivative, the Clarke subdifferential, and the notion of Clarke critical points, which will be used in the sequel.
Let $h:\R^n \to \overline{\R}$ be proper and locally Lipschitz around $x\in\dom h$.
The Clarke directional derivative of $h$ at $x$ in direction $d \in \mathbb{R}^n$ is defined by
\[
h^\circ (x; d) :=\mathop{\limsup}\limits_{\substack{y \to x \\ t\downarrow 0}}\frac{h(y+td) - h(y)}{t},
\]
 and the Clarke subdifferential of $h$ at $x$ is given by
 \begin{equation*}
  \partial^{C} h(x) := \{v \in \mathbb R^n: \, \langle v, d \rangle \le h^\circ(x; d), ~ \forall ~ d \in \mathbb R^n\},
 \end{equation*} 
By convention, $\partial^C h(x):=\emptyset$ for $x\notin\dom h$. Moreover,
\begin{equation}\label{prop:direc}
 h^{\circ} (x; d) = \maxt{\xi \in \partial^{C} h(x)} \langle \xi, d \rangle, \qquad \forall ~ d \in \mathbb R^{n}.
\end{equation} 
 A point $x \in \dom h$ is called a Clarke critical point of $h$ if $0 \in \partial^{C} h(x)$.

A function $h: \mathbb{R}^{n} \rightarrow \overline{\mathbb{R}}$ with convex domain is said to be:
\begin{enumerate}[label=(\textbf{\alph*}), font=\normalfont\bfseries, leftmargin=0.7cm]
\item\label{def:st:conv} \textit{strongly convex} on $\dom{h}$ if there exists $\gamma > 0$ such that, for all $x,y \in \dom{h}$ and all $\lambda \in [0,1]$,
 \begin{equation}\label{strong:convex}
  h(\lambda y + (1-\lambda)x) \leq \lambda h(y) + (1-\lambda) h(x) -
  \lambda (1 - \lambda) \frac{\gamma}{2} \lVert x - y \rVert^{2},
 \end{equation}
 \item \label{def:st:qua} \textit{strongly quasiconvex} on $\dom{h}$ if there exists $\gamma > 0$ such that, for all $x,y \in \dom{h}$ and all $\lambda \in [0,1]$,
 \begin{equation}\label{strong:quasiconvex}
  h(\lambda y + (1-\lambda)x) \leq \max \{h(y), h(x)\} - \lambda(1 -
  \lambda) \frac{\gamma}{2} \lVert x - y \rVert^{2}.
 \end{equation}
\end{enumerate}
Setting $\gamma=0$ in \eqref{strong:convex} (resp. \eqref{strong:quasiconvex}) yields the notions of \textit{convexity} (resp. \textit{quasiconvexity}).
In particular, every (strongly) convex function is (strongly) quasiconvex, whereas the converse implication does not hold in general (see \cite{cambini2008generalized,hadjisavvas2006handbook,Lara2022strongly}).

The following fact collects useful properties of $x \mapsto \|x\|^q$, $q>1$, used throughout the sequel.
\begin{fact}[\textbf{Basic properties of the power norm}]\cite{Ahookhosh2025Asymptotic}\label{fact:norm:unif}
Let $x, y\in \R^n$ be arbitrary.
\begin{enumerate}[label=(\textbf{\alph*}), font=\normalfont\bfseries, leftmargin=0.7cm]
 \item\label{fact:norm:unif:a} If $q>1$, then
  \begin{equation}\label{eq:ineqp}
   \Vert x-y\Vert^q\leq \Vert x\Vert^q-q\Vert x-y\Vert^{q-2}\langle x-y, y\rangle.
  \end{equation}

  \item\label{fact:norm:unif:b} If $q\in (1, 2)$, it is shown in \cite{Kabgani2025Moreau} that, for any
 $r>0$  and $a, b\in \mb(0, r)$,
\[
\langle \Vert a\Vert^{q-2}a - \Vert b\Vert^{q-2}b, a-b\rangle\geq \kappa(q) r^{q-2}\Vert a - b\Vert^2,
\]
where
\begin{equation}\label{eq:def:kappa_p}
\kappa(t):=\left\{
   \begin{array}{ll}
     \frac{(2+\sqrt{3})(t-1)}{16} & t\in (1, \widehat{t}], \\[0.2cm]
      \frac{2+\sqrt{3}}{16}\left(1-\left(3-\sqrt{3}\right)^{1-t}\right) ~~& t\in [\widehat{t},2),
   \end{array}\right.
\end{equation}
and $\widehat{t}$ is the solution to $\frac{t(t-1)}{2} = 1 - \left[1 + \frac{(2-\sqrt{3})t}{t-1}\right]^{1-t}$ on $(1, 2]$, and is determined numerically as $\widehat{t} \approx 1.3214$. For simplicity, we write $\kappa(t)$ as $\kappa_t$. Hence, $\gh(x):=\lVert x\rVert^{q}$ is strongly convex on $\mb(0, r)$ with modulus $\phi(t)=\frac{\kappa_q r^{q-2}}{2}t^2$
\cite[Lemma~4.2.1]{Nesterov2018}.

 \item\label{fact:norm:unif:c} If $q > 2$, and $\hat{\sigma}_{q} = \left( \frac{1}{2} \right)^\frac{3q-2}{2}$, then for every $\lambda \in \, (0, 1)$, we have
 \begin{equation}\label{eq:unifcon:normp}
  \Vert \lambda x+ (1-\lambda) y\Vert^{q} \leq \lambda \Vert x \Vert^{q} + (1- \lambda) \Vert y \Vert^{q} - \lambda (1- \lambda) \hat{\sigma}_{q} \Vert x-y \Vert^{q}.
 \end{equation}  
\end{enumerate}
\end{fact}

Inspired by \cite[Lemma~10 and Remark~2]{Hinder2020Near}, we introduce the following notion of quasar-convexity. 
This generalized convexity notion has appeared in several machine learning applications, where it is used to establish fast convergence of gradient-type methods; see  \cite{guminov2023accelerated,Hardt2018Gradient,HADR,Hinder2020Near,wang2023continuized} and the references therein.

\begin{definition}[\textbf{Strong quasar-convexity}]\label{def:quasar}
 Let $\kappa \in (0,1]$, $\gamma \geq 0$, and let $h: \mathbb{R}^n \to \Rinf$ be a proper function with convex domain such that $\argmin{\mathbb{R}^n} h \neq \emptyset$. The function $h$ is said to be $(\kappa, \gamma)$-\textit{strongly quasar-convex} with respect to $\ov{x} \in \argmin{\mathbb{R}^n} h$ if, for all $x \in \dom h$ and all $\lambda \in [0, 1]$,
\begin{align}\label{eq:def:quasar}
 h(\lambda \ov{x} + (1-\lambda)x) \leq \kappa \lambda h(\ov{x}) + (1-\kappa \lambda) h(x) - \lambda \left( 1 - \frac{\lambda}{2 - \kappa} \right) \frac{\kappa \gamma}{2} \lVert x - \ov{x} \rVert^{2}, 
\end{align}  
If $\gamma = 0$, we simply say that $h$ is $\kappa$-quasar-convex with respect to $\ov{x} \in \argmin{\mathbb{R}^n} h$. 
\end{definition}
Strong convexity controls the function along every segment joining two points in its domain through a quadratic correction term. In contrast, $(\kappa,\gamma)$-strong quasar-convexity imposes such control only along segments joining an arbitrary point to a minimizer $\ov{x}$, while also weakening the right-hand side through the parameter $\kappa$. In this sense, strong quasar-convexity can be viewed as a one-sided relaxation of strong convexity, in which curvature is imposed only along directions pointing toward a minimizer.

In Definition~\ref{def:quasar}, we omit the phrase ``with respect to $\ov{x} \in \argmin{\mathbb{R}^{n}}\,h$" whenever the reference point $\ov{x}$ is clear from the context.  

\begin{remark}[\textbf{Relations between quasar-convexity, star-convexity, and convexity}]\label{rem:on_beg_prop}
 Let $\kappa \in (0,1]$, $\gamma \geq 0$, and let $h: \mathbb{R}^n \to \Rinf$ be a proper function with convex domain such that $\ov{x}\in\argmin{\mathbb{R}^n} h$.
    \begin{enumerate}[label=(\textbf{\alph*}), font=\normalfont\bfseries, leftmargin=0.7cm]
\item\label{rem:on_beg_prop:a} If $h$ is $\kappa$-quasar-convex with respect to $\ov{x}$, then the set $\argmin{\mathbb{R}^{n}}\,h$ is {\it star-convex} with star center $\ov{x}$ (see \cite[Observation 3]{Hinder2020Near} and Example~\ref{ex:star_shaped_min}); that is, for every $y \in \argmin{\mathbb{R}^{n}}\,h$ and every $\lambda \in [0, 1]$, we have $\lambda \ov{x} + (1-\lambda)y \in \argmin{\mathbb{R}^{n}}\,h$.

\item \label{rem:on_beg_prop:b} If $h$ is $(\kappa,\gamma)$-strongly quasar-convex with respect to $\ov{x}$ and $\gamma>0$, then $\argmin{\mathbb{R}^n}h$ is a singleton (see \cite[Observation 4]{Hinder2020Near} and Example~\ref{ex:strongly-quasar-nonsmooth-nonstar}).

\item \label{rem:on_beg_prop:c} 
If $\kappa = 1$ and $h$ is (strongly) quasar-convex with respect to $\ov{x}$, then $h$ is {\it (strongly) star-convex} with respect to $\ov{x}$ \cite{nesterov2006cubic}. In particular, every convex function $h$ is $1$-quasar-convex with respect to any minimizer of $h$, and every $\gamma$-strongly convex function $h$ is $(1,\gamma)$-strongly quasar-convex with respect to its unique minimizer.

 \item \label{rem:on_beg_prop:d} If $h$ is $(\kappa,\gamma)$-strongly quasar-convex and $\gamma>0$, then (see \cite[Corollary 1]{Hinder2020Near})
\begin{equation}\label{qwc:sconvex}
 h(\overline{x}) + \frac{\kappa \gamma}{2(2 - \kappa)} \lVert y - \overline{x} \rVert^{2} \leq h(y), ~ \forall ~ y \in K,
\end{equation}
where $\overline{x} \in \argmin{\mathbb{R}^n} h$, thus strongly quasar-convex functions satisfy a quadratic growth condition with modulus $\frac{\kappa \gamma}{2(2 - \kappa)} > 0$. 
 \end{enumerate}
\end{remark}



The relationships between convexity, star-convexity, and quasar-convexity are summarized in Figure~\ref{fig:convexity-relations}.
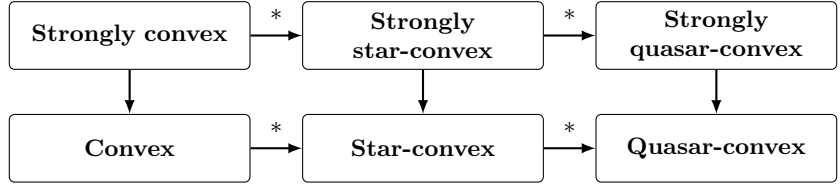
\begin{figure}[ht]
\centering





\begin{tikzpicture}[
    every node/.style={align=center, font=\small},
    box/.style={
        draw,
        rounded corners=2pt,
        minimum width=3.2cm,
        minimum height=0.9cm,
        text width=3.0cm,
        inner sep=2pt
    },
    impl/.style={->, >=latex, thick},
    vert/.style={->, >=latex, thick}
]

\node[box] (sc)  at (0,1.5) {\textbf{Strongly convex}};
\node[box] (ssc) at (3.9,1.5) {\textbf{Strongly star-convex}};
\node[box] (sqc) at (7.8,1.5) {\textbf{Strongly quasar-convex}};

\node[box] (c)   at (0,0) {\textbf{Convex}};
\node[box] (stc) at (3.9,0) {\textbf{Star-convex}};
\node[box] (qc)  at (7.8,0) {\textbf{Quasar-convex}};

\draw[impl] (sc) -- node[above=2pt] {$*$} (ssc);
\draw[impl] (ssc) -- node[above=2pt] {$*$} (sqc);
\draw[impl] (c) -- node[above=2pt] {$*$} (stc);
\draw[impl] (stc) -- node[above=2pt] {$*$} (qc);

\draw[vert] (sc) -- (c);
\draw[vert] (ssc) -- (stc);
\draw[vert] (sqc) -- (qc);

\end{tikzpicture}
\caption{Relationships between convexity, star-convexity, and quasar-convexity. 
Here, ``$*$"  indicates that $\argmin{\R^n} h\neq\emptyset$ is required, and the star-convexity and quasar-convexity properties are understood with respect to a minimizer $\ov{x}\in\argmin{\R^n} h$. None of the reverse implications holds in general (see \cite{guminov2023accelerated,HADR,Hinder2020Near,nesterov2006cubic}).
\label{fig:convexity-relations}
}
\end{figure}

\begin{remark}[\textbf{Non-comparability of (strong) quasiconvexity and (strong) quasar-convexity}]
There is, in general, no inclusion relationship between (strong) quasiconvexity and (strong) quasar-convexity. 
Indeed, the function in \cite[Example~9]{Hadjisavvas2026Heavy} is strongly quasiconvex but not quasar-convex, while the function in \cite[Example~27]{deBrito2025extending} is strongly quasar-convex but not quasiconvex. 
See also Example~\ref{ex:strongly-quasar-nonsmooth-nonstar} for a $(\kappa,\gamma)$-strongly quasar-convex function which is not star-convex, and hence not convex.
For differentiable functions, sufficient conditions ensuring that strong quasiconvexity implies quasar-convexity are provided in \cite{Hadjisavvas2026Heavy,lara2025delayed}.
\end{remark}


The following example exhibits a nonsmooth, nonconvex function that is strongly quasar-convex but not star-convex, showing that strong quasar-convexity is strictly more general than star-convexity.
\begin{example}[\textbf{Strong quasar-convexity beyond star-convexity}]\label{ex:strongly-quasar-nonsmooth-nonstar}
Define $h:\R^2\to\R$ by
\[
h(x):=
\begin{cases}
  \|x\|^{1/2}+\left(2+\sin(3\bs\arg(x))\right)\|x\|^2, & x\neq 0, \\
  0, & x=0.
\end{cases}
\]
where $\bs\arg(x) \in (-\pi,\pi]$ denotes the principal argument of $x\in\mathbb R^2 \setminus\{0\}$. Then $h$ is proper, nonsmooth, nonconvex, and
$(\kappa,\gamma)$-strongly quasar-convex with respect to $\ov x=0$ for $\kappa=\frac{1}{2}$ and $\gamma=1$.
Indeed, for $x\neq 0$, write $x=ru$ with $r=\|x\|>0$ and $\|u\|=1$, and set $q(u):=2+\sin(3\bs\arg(u))\in[1,3]$. Then,
\[
h(x)=r^{1/2}+q(u)r^2.
\]
Let $\lambda\in[0,1]$ and put $s:=1-\lambda$. We have
\[
h(sx)=s^{1/2}r^{1/2}+s^2 q(u)r^2.
\]
Since $s=1-\lambda$ and $q(u)\ge 1$, we have
\[
\left(1-\frac{\lambda}{2}\right)q(u)-s^2q(u)=\lambda\left(\frac32-\lambda\right)q(u)\ge\lambda\left(\frac{3}{2}-\lambda\right)
\ge\frac{\lambda}{4}\left(1-\frac{2\lambda}{3}\right),
\]
which yields
\[
s^2 q(u)\le \left(1-\frac{\lambda}{2}\right)q(u)-\frac{\lambda}{4}\left(1-\frac{2\lambda}{3}\right).
\]
This, together with $s^{1/2}\le 1-\frac{\lambda}{2}$, implies
\begin{align*}
  h((1-\lambda)x) &= s^{1/2}r^{1/2}+s^2 q(u)r^2\\
  &\leq \left(1-\frac{\lambda}{2}\right)r^{1/2}+\left(1-\frac{\lambda}{2}\right)q(u)r^2-\frac{\lambda}{4}\left(1-\frac{2\lambda}{3}\right)r^2\\
 &= \left(1-\frac{\lambda}{2}\right)h(x)
-\frac{\lambda}{4}\left(1-\frac{2\lambda}{3}\right)\|x\|^2.
\end{align*}
Equivalently, since $h(0)=0$, we get
\[
h(\lambda\ov x+(1-\lambda)x)
\le \frac{\lambda}{2}h(\ov x)+\left(1-\frac{\lambda}{2}\right)h(x)
-\lambda\left(1-\frac{\lambda}{2-\kappa}\right)\frac{\kappa\gamma}{2}\|x-\ov x\|^2,
\]
with $\kappa=\frac12$ and $\gamma=1$. Hence $h$ is $\bigl(\frac12,1\bigr)$-strongly quasar-convex with respect to $\ov{x}=0$.

Moreover, $h$ is not star-convex with respect to $\ov{x}=0$. Indeed, if it were, then
$h(tx)\le t\,h(x)$ for all $x\in\R^2$ and $t\in[0,1]$.
Taking $t=\frac12$, we get
\[
h\left(\frac{x}{2}\right)-\frac12 h(x)=\left(\frac{1}{\sqrt2}-\frac12\right)\|x\|^{1/2}-\frac14 q(u)\|x\|^2.
\]
Since $q(u)\le 3$, the right-hand side is positive for all sufficiently small $x\neq 0$, which is a contradiction. Thus $h$ is not star-convex, and hence it is not convex.
\end{example}

For further background on generalized convexity and nonsmooth analysis, we refer to \cite{Bauschke17,cambini2008generalized,Clarke1990,GLM-Survey,hadjisavvas2006handbook,Lara2022strongly,lara2025characterizations,Rockafellar09} and the references therein.

\subsection{Smooth (strongly) quasar-convex functions}
To provide a more comprehensive account of the properties of the class of (strongly) quasar-convex functions, we now summarize several known results in the smooth setting.

\begin{fact}[\textbf{First-order characterization}]\label{fact:char:diff}\cite[Lemma 10]{Hinder2020Near}
A differentiable function $h: \mathbb{R}^{n} \rightarrow \mathbb{R}$ is $(\kappa, \gamma)$-strongly quasar-convex with $\gamma \geq 0$ and $\kappa \in (0, 1]$ with respect to $\overline{x} \in \argmin{\mathbb{R}^{n}}\,h$ if and only if
\begin{equation}\label{diff:squasar}
 h(\overline{x}) \geq h(y) + \frac{1}{\kappa} \langle \nabla h(y), \overline{x} - y \rangle + \frac{\gamma}{2} \lVert y - \overline{x} \rVert^{2}, \qquad \forall ~ y \in \mathbb{R}^{n}.
\end{equation}
\end{fact}

By Fact~\ref{fact:char:diff}, quasi-strongly convex functions \cite{necoara2019linear} are closely connected to strongly quasar-convex functions. Additional properties for differentiable (strongly) quasar-convex functions can be found in \cite{guminov2023accelerated,Hardt2018Gradient,HADR,Hinder2020Near,wang2023continuized}.

The next result provides a sufficient condition for strongly quasar-convex functions.

\begin{fact}[{\bf Sufficient condition}] \cite[Proposition 8]{HADR}
 Let \(f:\mathbb{R}\to \mathbb{R}\) be a \((\kappa, \gamma)\)-strongly quasar-convex function with \(f(0) = 0 = h^{*}\) and \(g:S^{d-1} \to \mathbb{R}\) be a differentiable function with \(g(x)\geq 1\) for all \(x\in S^{d-1}\). Define
\[
    h(x) = f(\| x\|)g\left(\frac{x}{\| x\|}\right),\quad h(0)=0.
\]
Then, \(h\) is \((\kappa ,\gamma)\)-strongly quasar convex.
\end{fact}



\begin{fact}[{\bf Polyak--{\L}ojasiewicz (PL) inequality}]\cite[Lemma 2]{HADR}
 A PL function with modulus $\mu > 0$ satisfies an error bound with modulus $\mu > 0$. A function which is \(L\)-smooth and satisfies error bound with modulus $\theta > 0$ satisfies PL with modulus \(\mu := \frac{\theta}{L}\).
\end{fact}

Let us recall the restricted secant inequality property.

\begin{definition}[{\bf Restricted secant inequality}]
 It is said that a function \(h: \mathbb{R}^d\mapsto \mathbb{R}\) satisfies a restricted secant inequality with modulus $\nu > 0$ if
 $$\langle \nabla h(x),x - \overline{x}\rangle \geq \nu \| x - \overline{x}\| ^2,\quad \forall ~ x \in \mathbb{R}^d.$$
\end{definition}

\begin{fact}[{\bf Restricted secant inequality}] \cite[Lemma 3]{HADR}
Suppose \(h\) satisfies the uniform acute angle condition (UAAC)
\begin{equation}\label{uaac}
 \forall ~ x \in \mathbb{R}^d,\quad 1 \geq \frac{\langle \nabla h(x),x - \overline{x}\rangle}{\| \nabla h(x)\| \| x - \overline{x}\|} \geq a > 0.
\end{equation} 
Then, if \(h\) satisfies the error bound with modulus $\theta > 0$, then \(h\) satisfies the restricted secant inequality with modulus $\theta a > 0$.
\end{fact}

\begin{fact}[{\bf Restricted secant inequality}] \cite[Lemma 4]{HADR}
 Let \(h\) be a \(L\)-smooth function. Then the following assertions hold:
 \begin{enumerate}
  \item[$(a)$] If $h$ is \((\kappa ,\gamma)\)-strongly quasar-convex, then $h$ satisfies the restricted secant inequality with modulus $\frac{\kappa\gamma}{2 - \kappa} > 0$.

  \item[$(b)$] If $h$ satisfies the restricted secant inequality with modulus $\nu > 0$ and $0 < \kappa < \frac{2\nu}{L}$, then $h$ is $\left(\kappa, \frac{\nu}{\kappa} -\frac{L}{2}\right)$-strongly quasar-convex.
 \end{enumerate}
\end{fact}

The subsequent result provides a relationship between strongly quasar-convex functions and functions that satisfy the PL property.

\begin{fact}[{\bf PL inequality}] \cite[Theorem 7]{HADR}
 Let \(h\) be a \(L\)-smooth and \(\mu\)-PL function with a unique minimizer. Then there exists \(\kappa, \gamma^{\prime} > 0\) such that \(h\) is \((\kappa, \gamma^{\prime})\)-strongly quasar-convex if and only if for some \(a \in (0,1]\), \(h\) satisfies relation \eqref{uaac}.
 In particular, if (UAAC) holds, then as long as $\kappa < \frac{2\mu a}{L}$, $h$ is $(\kappa, \frac{\mu a}{\kappa} -\frac{L}{2})$-strongly quasar-convex.
\end{fact}

\begin{fact}[{\bf PL inequality}] \cite[Proposition 10]{HADR}
 Let $h$ be \((\kappa ,\gamma)\)-strongly quasar-convex function. Then $h$ satisfies the PL property with modulus $\gamma \kappa^{2} > 0$.
\end{fact}

In the case of a twice continuously differentiable function, we have the following results.

\begin{fact}[{\bf Twice continuously differentiability}] {\rm (see \cite[Proposition in page 20]{HADR})}
 Let $h$ be a twice continuously differentiable and \((\kappa, \gamma)\)-strongly quasar-convex function. Let \(x \neq \overline{x}\) and \(t > 0\). Then, 
\[
 \frac{1}{t} \int_{0}^{t} \frac{\langle \nabla^{2} h(\overline{x} + s(x - \overline{x}))(\overline{x} - x), \overline{x} - x\rangle}{\|x - \overline{x} \|^{2}} ds \geq \kappa \frac{\gamma}{2}.
\]
\end{fact}

Finally, observe that in the following sufficient condition, no assumption on (strong) quasar-convexity is needed.

\begin{fact}[{\bf $C^2$ and strong convexity}]\cite[Proposition in page 21]{HADR}
Let \(h\) be a $C^2$ function with a unique minimizer \(\overline{x}\). If $h$ satisfies a quadratic growth condition with modulus $\mu > 0$, then there exists \(\eta > 0\) such that \(h\) is strongly convex on $B(\overline{x}, \eta)$.
\end{fact}


\section{Nonsmooth quasar-convex functions: Fundamental properties}\label{sec:03}
In this section, we investigate the characterization and calculus of the class of quasar-convex functions without differentiability assumptions. For the differentiable case, we refer to \cite{guminov2023accelerated,Hardt2018Gradient,HADR,Hinder2020Near,Pun2024,wang2023continuized} and references therein.
 


Let us begin with an example showing that, in the nonsmooth setting, $\kappa$-quasar-convexity may hold either with respect to every minimizer or only with respect to some of them, depending on the geometry of the minimizer set.

\begin{example}\label{exam:noncq_for_all}
Let $C$ be a nonempty closed proper subset of $\mathbb{R}^n$ with nonempty kernel
\[
\ker(C):=\{\overline{x}\in C:[\overline{x},z]\subset C \text{ for every } z\in C\},
\]
let $\alpha\in(0,1)$, and define
$h:\mathbb{R}^n\to\mathbb{R}$ by $
h(x)=\dist(x,C)^\alpha.$
Since $\alpha\in(0,1)$, the function $h$, in general, is nonsmooth and nonconvex. Moreover,
\[
\argmin{\mathbb{R}^n} h = C.
\]
If $\overline{x}\in \ker(C)$, then for every $0<\kappa\le \alpha$, $h$ is $\kappa$-quasar-convex with respect to $\overline{x}$. Indeed, if $x\in \mathbb{R}^n$ and $\lambda\in[0,1]$, choose $p\in C$ such that
\[
\|x-p\|=\dist(x,C).
\]
Since $\overline{x}\in \ker(C)$, we have $c_\lambda:=\lambda\overline{x}+(1-\lambda)p\in C$ and
\[
\dist\bigl(\lambda \overline{x}+(1-\lambda)x,C\bigr)
\le
\|\lambda \overline{x}+(1-\lambda)x-c_\lambda\|
=
(1-\lambda)\dist(x,C).
\]
Therefore,
\[
h\bigl(\lambda \overline{x}+(1-\lambda)x\bigr)
\le
(1-\lambda)^\alpha h(x).
\]
Since $t\mapsto t^\alpha$ is concave on $[0,1]$, we have
$
(1-\lambda)^\alpha\le 1-\alpha\lambda\le 1-\kappa\lambda.
$
Hence
\[
h\bigl(\lambda \overline{x}+(1-\lambda)x\bigr)
\le
(1-\kappa\lambda)h(x)
=
\kappa\lambda h(\overline{x})+(1-\kappa\lambda)h(x),
\]
and thus $h$ is $\kappa$-quasar-convex with respect to $\overline{x}$.

On the other hand, if $\overline{x}\in C\setminus \ker(C)$, then $h$ fails to
be $\kappa$-quasar-convex with respect to $\overline{x}$ for every
$\kappa\in(0,1]$. Indeed, there exist $z\in C$ and $\lambda_0\in(0,1)$ such
that
\[
\lambda_0\overline{x}+(1-\lambda_0)z\notin C.
\]
Since $h(\overline{x})=h(z)=0$ but
$
h\bigl(\lambda_0\overline{x}+(1-\lambda_0)z\bigr)>0,
$
the quasar-convexity inequality cannot hold. Thus, $h$ is
$\kappa$-quasar-convex exactly with respect to the minimizers in $\ker(C)$.

In particular, if $C$ is convex, then $\ker(C)=C$, and $h$ is
$\kappa$-quasar-convex with respect to every minimizer. Figures~\ref{fig:disk_all_minimizers} 
illustrate  these two situations for $\alpha=\frac{1}{2}$, by considering the sets
\[
C_1=\{(x,y)\in\mathbb{R}^2:x^2+y^2\le 1\}
\quad\text{and}\quad
C_2=\bigl([-2,2]\times[-1,1]\bigr)\cup\bigl([-1,1]\times[-2,2]\bigr).
\]
Here, $C_1$ is the convex unit disk, while $C_2$ is a nonconvex cross-shaped set
with
$\ker(C_2)=[-1,1]^2$.
\begin{figure}
     \centering
 \includegraphics[width=\textwidth]{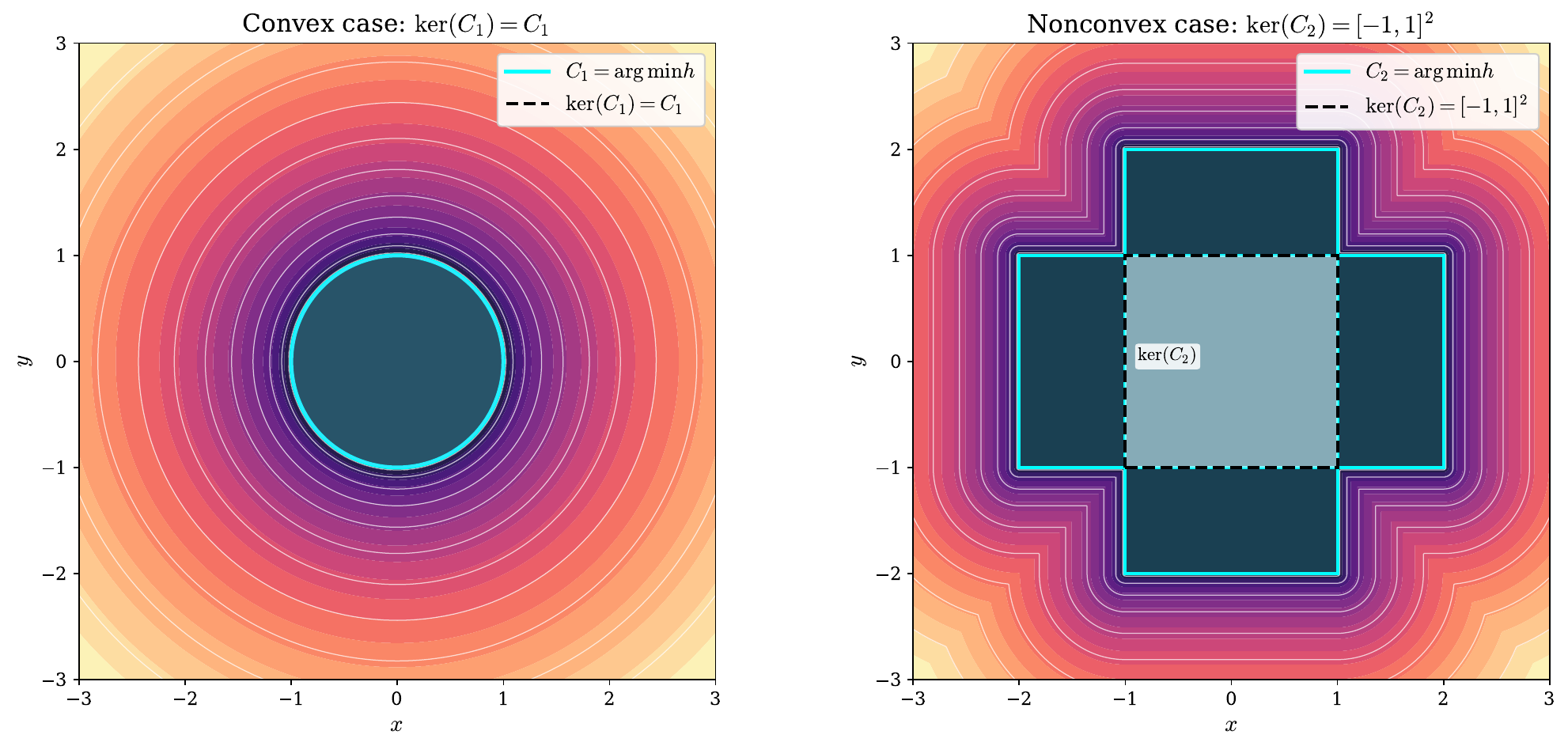}
     \caption{A nonsmooth, nonconvex function that is $\kappa$-quasar-convex with respect to some minimizers; see Example~\ref{exam:noncq_for_all}. }
     \label{fig:disk_all_minimizers}
 \end{figure}


\end{example}

\subsection{Basic properties and calculus rules}

We begin with a result showing that strong quasar-convexity is preserved under positive scaling, translation, and a reduction of the quasar parameter $\kappa$, with a corresponding adjustment of the parameter $\gamma$.

\begin{proposition}[\textbf{Scaling, translation, and parameter transformation}]\label{Scaling}
 Let $h: \mathbb{R}^{n} \to \overline{\mathbb{R}}$ be a proper function  and let $\overline{x} \in \argmin{\mathbb{R}^{n}}\,h$. Assume that $h$ is $(\kappa, \gamma)$-strongly quasar-convex with respect to  $\overline{x}$, where $\gamma \geq 0$. Then the following assertions hold:
\begin{enumerate}[label=(\textbf{\alph*}), font=\normalfont\bfseries, leftmargin=0.7cm]
  \item\label{Scaling:a} For every $\alpha > 0$, the function $g(x):=\alpha h(x)$ is $(\kappa, \alpha \gamma) $-strongly quasar-convex with respect to $\overline{x}$.

   \item\label{Scaling:a2} Define $g(z) := h(\ov{x}+z)$. Then $g$ is $(\kappa, \gamma)$-strongly quasar-convex with respect to $\ov{z}=0\in \argmin{\R^n} g$. 

  \item\label{Scaling:b} For every  $\theta \in \, (0, 1]$, the function $h$ is $(\theta \kappa, \frac{\gamma}{\theta})$-strongly quasar-convex with respect to $\overline{x}$.
 \end{enumerate}
\end{proposition}
\begin{proof}
 \ref{Scaling:a}: It is straightforward. 
 \\
\ref{Scaling:a2} 
We have $\dom g = \dom h - \ov{x}$. Hence, $\dom g$ is convex and $0\in \dom g$.
Set $\ov{z}=0$. Since $\ov{x}\in \argmin{\R^n} h$, for every $z\in\R^n$, it holds that
$g(\ov{z})=h(\ov{x})\le h(\ov{x}+z)=g(z)$.
Thus $\ov{z} = 0 \in \argmin{\R^n} g$. Now, let $z\in \dom g$ and $\lambda\in[0,1]$, i.e.,
\begin{equation}\label{eq:Scaling:a2:1}
    g(\lambda\ov{z}+(1-\lambda)z)=g((1-\lambda)z)=h(\ov{x}+(1-\lambda)z).
\end{equation}
Applying the $(\kappa,\gamma)$-strong quasar-convexity of $h$ at the point $y:=\ov{x}+z$ ensures
\begin{align*}
h(\ov{x}+(1-\lambda)z)
&=h\bigl(\lambda \ov{x}+(1-\lambda)(\ov{x}+z)\bigr) \\
&\le\kappa\lambda h(\ov{x})+(1-\kappa\lambda)h(\ov{x}+z)-\lambda\left(1-\frac{\lambda}{2-\kappa}\right)\frac{\kappa\gamma}{2}
\|(\ov{x}+z)-\ov{x}\|^2.
\end{align*}
This, together with \eqref{eq:Scaling:a2:1}, implies that $g$ is $(\kappa,\gamma)$-strongly quasar-convex with respect to $\ov{z}$.
 \\
\ref{Scaling:b}: Let $x \in \dom{h}$ and $\lambda\in[0,1]$. In virtue of the $(\kappa, \gamma)$-strong quasar-convexity of $h$  with respect to $\overline{x}$, we define:
\begin{align*}
 & R_{\kappa,\gamma} (x,\lambda) := \kappa \lambda h(\overline{x})+(1-\kappa\lambda)h(x) -\lambda \left(1-\frac{\lambda}{2-\kappa} \right) \frac{\kappa \gamma}{2}\|x-\overline{x}\|^2, \\
 & R_{\kappa,\gamma}^{\theta} (x,\lambda) := \theta \kappa \lambda h(\overline{x})+(1-\theta\kappa \lambda) h(x) -\lambda \left(1-\frac{\lambda}{2-\theta \kappa} \right) \frac{\kappa\gamma}{2}\|x-\overline{x}\|^2. 
\end{align*} 
Thus, it is enough to prove 
\begin{equation}\label{eq:lemma}
 R_{\kappa,\gamma} (x,\lambda) \le R_{\kappa, \gamma}^{\theta} (x,\lambda), \qquad \forall ~ x \in \dom{h}.
\end{equation}
A direct computation gives
$$ R_{\kappa,\gamma}^{\theta} (x,\lambda) -R_{ \kappa, \gamma} (x,\lambda) = \lambda \kappa (1-\theta)\big(h(x)-h(\overline{x}) \big) - \frac{\lambda^2 \kappa^2 \gamma(1-\theta)}{2(2-\kappa)(2-\theta\kappa)} \,\|x-\overline{x}\|^2,$$
i.e.,
\begin{equation}\label{eq:001}
 R_{\kappa,\gamma}^{\theta} (x,\lambda) - R_{\kappa, \gamma} (x,\lambda) = \lambda\kappa (1-\theta) \left(
h(x) - h(\overline{x}) - \frac{\lambda \kappa \gamma}{2 (2-\kappa) (2-\theta \kappa)} \|x-\overline{x}\|^2 \right).
\end{equation}
Since $\lambda \in [0, 1]$ and $2 - \theta \kappa \ge 1$, we get
$ \frac{\lambda}{2-\theta\kappa}\le 1$, i.e.,
\begin{equation} \label{equa32} \frac{\lambda \kappa \gamma}{2 (2- \kappa) (2-\theta \kappa)}\,\|x-\overline{x}\|^2 \le \frac{\kappa\gamma}{2(2-\kappa)} \,\|x-\overline{x}\|^2.\end{equation}
Using \eqref{equa32} and relation \eqref{qwc:sconvex}, the right-hand side in \eqref{eq:001} is nonnegative, and \eqref{eq:lemma} holds. Consequently, we have $h(\lambda \overline{x}+(1-\lambda)x)\le R_{\kappa,\gamma}^{\theta} (x,\lambda)$  for all $x \in \dom{h}$, i.e., $h$ is $(\theta \kappa, \gamma/\theta)$-strongly quasar-convex with respect to $\overline{x}$.
\end{proof}
We next establish that the class of strongly quasar-convex functions is closed under finite positive weighted sums, provided that the functions share a common minimizer.
\begin{proposition}[\textbf{Closure under finite positive weighted sums}]\label{prop:sum:weights}
 Let $h_i: \mathbb{R}^{n} \to \overline{\mathbb{R}}$ be proper functions for $i \in \{1,\ldots,m\}$, and assume that
 $\bigcap^{m}_{i=1} \dom{h_{i}} \neq \emptyset$ and $\overline{x} \in \bigcap_{i=1}^m \argmin{\mathbb{R}^n }\,h_i$. Let $\alpha_i>0$ for all $i$, and define
 \begin{equation}\label{sum:weights}
  h(x) := \sum_{i=1}^m \alpha_i h_i(x).
 \end{equation} 
 Assume that each $h_i$ is $(\kappa_i, \gamma_i)$-strongly quasar-convex with respect to $\overline{x}$, where $\gamma_{i} \geq 0$. Then $h$ is $(\kappa, \gamma)$-strongly quasar-convex with respect to $\overline{x}$, where $\kappa := \min_{1 \le i\le m} \kappa_i$ and $\gamma := \frac{1}{\kappa} \sum_{i=1}^m \alpha_i \kappa_i\gamma_i$.
\end{proposition}
\begin{proof}
 Since each $h_i$ is $(\kappa_i,\gamma_i)$-strongly quasar-convex, the set $\dom{h_i}$ is convex for every $i$. Hence,
$\dom h=\bigcap_{i=1}^m \dom{h_i}$ is convex. Moreover, because $\ov{x} \in \argmin{ \mathbb{R}^n}\,h_i$ for every $i$ and $\alpha_i>0$, we have $\ov{x} \in \argmin{ \mathbb{R}^n} \,h$. 
 For each $i\in\{1,\dots,m\}$,
since $\kappa\le \kappa_i$, define $\theta_i:=\frac{\kappa}{\kappa_i}\in(0,1]$.
By Proposition \ref{Scaling}~\ref{Scaling:b}, 
the $(\kappa_i,\gamma_i)$-strong quasar-convexity of $h_i$ implies that $h_i$ is also
$\left(\theta_i\kappa_i,\frac{\gamma_i}{\theta_i}\right)
=\left(\kappa,\frac{\kappa_i}{\kappa}\gamma_i\right)$-strongly quasar-convex with respect to $\ov{x}$. Hence,
\[
h_i(\lambda \ov x+(1-\lambda)x)
\leq \kappa\lambda h_i(\ov x)+(1-\kappa\lambda)h_i(x)-\lambda\left(1-\frac{\lambda}{2-\kappa}\right)
\frac{\kappa_i\gamma_i}{2}\|x-\ov x\|^2,\qquad \forall x\in\dom h, \lambda\in [0,1],
\]
for $i=1,\dots,m$. Multiplying by $\alpha_i>0$ and summing over $i=1,\dots,m$, for every $x\in\dom h$ and every $\lambda\in[0,1]$, we obtain
\[
h(\lambda \ov x+(1-\lambda)x)\leq\kappa\lambda h(\ov x)+(1-\kappa\lambda)h(x)-\lambda\left(1-\frac{\lambda}{2-\kappa}\right)
\frac{1}{2}\left(\sum_{i=1}^m \alpha_i\kappa_i\gamma_i\right)\|x-\ov x\|^2.
\]
Since $\sum_{i=1}^m \alpha_i\kappa_i\gamma_i=\kappa\gamma$, it follows that
\[
h(\lambda \ov x+(1-\lambda)x) \leq\kappa\lambda h(\ov x)+(1-\kappa\lambda)h(x)-
\lambda\left(1-\frac{\lambda}{2-\kappa}\right)\frac{\kappa\gamma}{2}\|x-\ov x\|^2.
\]
Therefore, $h$ is $(\kappa,\gamma)$-strongly quasar-convex with respect to $\ov x$.
\end{proof}

The following proposition shows that $\kappa$-quasar-convex functions admit no spurious local minima and that local maxima, if they exist, must occur at global minimizers.
\begin{proposition}[\textbf{Local extrema of quasar-convex functions}]\label{maxminlocal}
Let $h: \mathbb R^n \to \overline{\mathbb{R}}$ be a proper function and let $\overline{x} \in\argmin{\mathbb{R}^{n}} h$. Assume that $h$ is $\kappa$-quasar-convex with respect to $\overline{x}$. Then, the following assertions hold:
 \begin{enumerate}[label=(\textbf{\alph*}), font=\normalfont\bfseries, leftmargin=0.7cm]
  \item\label{maxminlocal:a} If $x$ is a local minimizer of $h$, then $x \in \argmin{\mathbb{R}^{n}} h$.

   \item\label{maxminlocal:b} Every local maximizer of $h$ belongs to $\argmin{\mathbb{R}^n}\,h$. In particular, \(h\) has no strict local maximizers.

 \end{enumerate}
\end{proposition}
\begin{proof}
\ref{maxminlocal:a}: Suppose, by contradiction, that $x$ is a local minimizer of $h$
and $x \notin \argmin{\mathbb{R}^{n}} h$. Then $h(x) > h(\overline{x})$. For $\lambda \in (0,1]$, set $x_\lambda := \lambda\overline x+(1-\lambda)x$. By $\kappa$-quasar-convexity, it is evident that
\[
 h(x_\lambda) \le \kappa\lambda h(\overline x) + (1-\kappa \lambda) h(x) = h(x)-\kappa \lambda ( h(x)-h(\overline x)) < h(x).
\]
Since $x_\lambda\to x$ as $\lambda\downarrow0$, this contradicts the local minimality of $x$. Therefore,
$x \in \argmin{\mathbb{R}^{n}} h$.
\\
\ref{maxminlocal:b}: We prove the contrapositive. Let $x \notin \argmin{\mathbb{R}^{n}} h$. Then $h(x) > h(\overline x)$. For $\varepsilon > 0$, define $y_\varepsilon := \overline x + (1 + \varepsilon) (x-\overline x)$ and $\lambda_\varepsilon := \frac{\varepsilon}{1 + \varepsilon} \in (0,1)$. Then $x = \lambda_\varepsilon \overline x + (1-\lambda_\varepsilon) y_\varepsilon$ and $ y_\varepsilon \to x$ as $\varepsilon \downarrow 0$. Since $h$ is $\kappa$-quasar-convex at $\overline x$, we get
$h(x) \le \kappa \lambda_\varepsilon h(\overline x)
+ (1-\kappa \lambda_\varepsilon) h(y_\varepsilon)$, and hence
$$ h(y_\varepsilon) \ge \frac{h(x)-\kappa\lambda_\varepsilon h(\overline x)}{1-\kappa\lambda_\varepsilon} = h(x) + \frac{\kappa \lambda_\varepsilon}{1- \kappa \lambda_\varepsilon} \big(h(x)-h(\overline x)\big) > h(x).$$
Thus, every neighborhood of $x$ contains a point $y_\varepsilon$ such that
$h(y_\varepsilon)>h(x)$, so $x$ is not a local maximizer.
Therefore, every local maximizer of $h$ belongs to $ \argmin{ \mathbb{R}^n} h$.
Finally, if $x$ were a strict local maximizer, then all nearby points would satisfy
$h(y)<h(x)$, which contradicts the fact that $x\in\argmin{} h$, i.e., $h$ accepts no strict local maximizers.
\end{proof}

Proposition~\ref{maxminlocal} shows that $\kappa$-quasar-convexity with respect to a minimizer $\ov x$ imposes a strong global structure on the function. In particular, the function admits no spurious 
local minima: every local minimizer must be global. Moreover, for every point $x\notin \argmin{ \mathbb{R}^n} h$, there exist points arbitrarily close to $x$ on the line through $\ov x$ and $x$ with strictly smaller and strictly larger function values. Consequently, $h$ has no strict local maximizers.

Next, we study the preservation of strong quasar-convexity under linear composition.
\begin{proposition}[\textbf{Preservation under linear composition}]\label{prop:comp-matrix}
 Let $h: \R^{m} \to \R$ be $(\kappa, \gamma)$-strongly quasar-convex with respect to $\overline{y} \in \argmin{\R^{m}} h$, where  $\gamma \geq 0$, and let $A: \R^{n} \to \R^{m}$ be a linear operator. Define $g := h \circ A$. Then the following assertions hold:
 \begin{enumerate}[label=(\textbf{\alph*}), font=\normalfont\bfseries, leftmargin=0.7cm]
  \item\label{prop:comp-matrix:a} If $\overline{y} \in \mathcal{R}(A)$, then there exists $\overline{x} \in \R^{n}$ such that $A \overline{x} = \overline{y}$, and for every $x \in \overline{x} + \mathcal{R}(A^\top)$ and every $\lambda\in[0,1]$, we have
  \begin{equation}\label{eq:comp-matrix}
   g(\lambda \overline{x} + (1-\lambda)x) \leq \kappa \lambda g( \overline{x}) + (1-\kappa \lambda) g(x) - \lambda \left(1 - \frac{\lambda}{2 - \kappa} \right) \frac{\kappa \gamma \sigma_{\min}^{2} (A)}{2} \lVert x - \overline{x} \rVert^2.
  \end{equation} 

  \item\label{prop:comp-matrix:b} If, in addition, $A$ has full column rank and $\overline{y} \in \mathcal{R}(A)$, then $g$ is $(\kappa, \gamma \sigma_{\min}^2(A))$-strongly quasar-convex on $\R^{n}$ with respect to $\overline{x}$.
 \end{enumerate}
\end{proposition}

\begin{proof}
 \ref{prop:comp-matrix:a}: Since $\overline{y} \in \mathcal{R}(A)$, choose $\overline{x}$ such that $A \overline{x} = \overline{y}$.
  Let $x\in \ov x+\mathcal R(A^\top)$ and
$\lambda\in[0,1]$. Since $h$ is $(\kappa,\gamma)$-strongly quasar-convex
with respect to $\ov y$, we have
 \begin{equation}\label{linear_quasar}
\begin{aligned}
 g(\lambda \overline{x} + (1-\lambda)x) 
 & \leq \kappa \lambda h(\overline{y}) + (1-\kappa \lambda) h(Ax) - \lambda \left( 1 - \frac{\lambda}{2 - \kappa} \right) \frac{\kappa \gamma}{2} \lVert Ax - \overline{y} \rVert^2 \\
 & = \kappa \lambda g(\overline{x}) + (1-\kappa \lambda) g(x) - \lambda \left( 1 - \frac{\lambda}{2 - \kappa} \right) \frac{\kappa \gamma}{2} \lVert A(x - \overline{x} ) \rVert^2.
\end{aligned}
\end{equation}
Since $x - \overline{x} \in \mathcal{R}(A^\top)$, we have $\lVert A(x - \overline{x} )\rVert^2 \geq \sigma_{\min}^2(A) \lVert x - \overline{x} \rVert^2$, i.e., \eqref{eq:comp-matrix} follows.
\\
\ref{prop:comp-matrix:b}: 
Assume that $A$ has full column rank. Then
$\mathcal N(A)=\{0\}$, and therefore $\mathcal R(A^\top)=\mathbb{R}^n$, leading to \eqref{eq:comp-matrix} for every $x\in\mathbb{R}^n$.
Since $\ov y\in\argmin{\mathbb{R}^{m}} h$ and $A\ov x=\ov y$, it follows that
\[
g(\ov x)=h(\ov y)\le h(Ax)=g(x)\qquad \forall ~ x \in\mathbb{R}^n,
\]
i.e., $\ov x\in\argmin{\mathbb{R}^{n}} g$. Therefore $g$ is
$(\kappa,\gamma\sigma_{\min}^2(A))$-strongly quasar-convex with respect to
$\ov x$.

\end{proof}

\begin{remark}[\textbf{Necessity of full column rank}]
Even when $A$ does not have full column rank, the proof of Proposition~\ref{prop:comp-matrix}~\ref{prop:comp-matrix:a} shows that, whenever $\ov y\in\mathcal R(A)$ and $A\ov x=\ov y$, the composition $g=h\circ A$ is $\kappa$-quasar-convex with respect to $\ov x$ on $\R^n$; that is, the corresponding inequality holds without the quadratic term.
On the other hand, if $A$ does not have full column rank, then $\mathcal N(A)\neq\{0\}$. Hence there exists $d\neq 0$ such that $Ad=0$,
and therefore
\[
g(x+td)=h(A(x+td))=h(Ax)=g(x),\quad \forall ~ x \in \mathbb{R}^n, \quad \forall ~ t \in\mathbb{R}.
\]
Thus, $g$ is constant along a nontrivial affine line and cannot be
$(\kappa,\tilde\gamma)$-strongly quasar-convex on $\mathbb{R}^n$ with respect to $\ov x$ for any $\tilde\gamma>0$.
\end{remark}

The following example shows that the assumption $\overline{y} \in \mathcal{R} (A)$ in Proposition \ref{prop:comp-matrix} is essential.

 \begin{example}[\textbf{Necessity of the range condition in linear composition}] \label{ex:range-necessary}
 The condition $\ov y\in\mathcal R(A)$ in Proposition \ref{prop:comp-matrix}
is essential. For example, define $h:\mathbb R^2\to\mathbb R$ by
\[
h(y):=
\begin{cases}
  \|y-\overline y\|^2 \left(1 + \frac12\sin^2 \big(3~ \bs\arg(y-\overline y) \big) \right), & y\neq \ov{y}, \\
  0, & y=\ov{y}.
\end{cases}
\]
where $\bs\arg(z) \in (-\pi,\pi]$ denotes the principal argument of $z\in\mathbb R^2 \setminus\{0\}$. The function $h$ is nonconvex, but it is $(\kappa, \gamma)$-strongly quasar-convex with respect to $\overline y \in \argmin{\R^2}\,h$; see \cite[Proposition 8]{HADR}.

Consider the linear operator $A:\mathbb{R}\to\mathbb{R}^2$ defined by $Ax := (x,0)$.
Then $\mathcal{R} (A) = \{(t,0): t \in \mathbb{R}\}$, and hence, for $\ov y:=(0,1)$, we have $\ov y\notin\mathcal R(A)$.
Define $g := h \circ A$. Then for every $x \in \mathbb{R}$, we have $Ax - \overline{y} = (x,-1)$. Let $\theta := \bs\arg(Ax-\overline y)$. Then
\[
\sin(\theta) = -\frac1{\sqrt{1+x^2}},
\]
and hence
\[
\sin(3\theta) = 3\sin(\theta) - 4 \sin^3(\theta) = \frac{1-3x^2}{(1+x^2)^{3/2}}.
\]
Consequently,
\begin{equation*}
 g(x) = h(Ax) = 1+x^2 + \frac{(3x^2-1)^2}{ 2(1+x^2)^2}.
\end{equation*}
Hence, $g(0)=\frac{3}{2}$ and $g(\frac{1}{{\sqrt3}}) = \frac{4}{3} < g(0)$, i.e., $0 \notin \argmin{\R}\,g$. Moreover,
$$ g(x)-g(0) = \frac{x^2(x^4+6x^2-3)}{(1+x^2)^2},$$
and thus $g(x)<g(0)$ for all $0<|x|< \sqrt{-3 + 2 \sqrt3}$. Then $x=0$ is a strict local maximizer of $g$.

By Proposition~\ref{maxminlocal}~\ref{maxminlocal:b}, $g$ cannot be
$\kappa$-quasar-convex with respect to any global minimizer. Thus,
when $\ov y\notin\mathcal R(A)$, the composition $h\circ A$ may fail to
be even $\kappa$-quasar-convex. This shows that the assumption
$\ov y\in\mathcal R(A)$ in Proposition \ref{prop:comp-matrix} cannot be removed.

\end{example}

The following auxiliary lemma will be used in the next proposition. Since its proof is elementary, we omit it.

\begin{lemma}[\textbf{Auxiliary constant}]\label{lem:c-explicit}
Let $\kappa_1 \in \, (0, 1)$ and $\kappa_2 \in \, (0, 1]$. Then, it follows that
\begin{equation}\label{eq:def-c}
 C_{\kappa_1,\kappa_2} := \inft{\lambda\in[0,1]} \frac{ 1-\frac{\lambda}{2 - \kappa_1}} {1 - \frac{\lambda}{2 - \kappa_1 \kappa_2}} = \frac{(1 - \kappa_1) (2 - \kappa_1 \kappa_2)} {(2-\kappa_1) (1-\kappa_1 \kappa_2)} > 0.
\end{equation}
\end{lemma}

As a consequence, we obtain a closure criterion for compositions with a nondecreasing quasar-convex function.

\begin{proposition}[\textbf{Composition with a quasar-convex function}]\label{increa:prop}
 Let $h: \mathbb{R}^{n} \to \overline{\mathbb{R}}$ be a proper function, $\overline{x} \in \argmin{\mathbb{R}^n} h$, $\kappa_1 \in \, ( 0, 1)$, and $\gamma \ge 0$. Assume that $h$ is $(\kappa_1, \gamma)$-strongly quasar-convex with respect to $\overline{x}$. Let $I \subseteq \mathbb{R}$ be an interval such that $h(\mathbb{R}^n) \subseteq I$, and let $\varphi: I \to \mathbb{R}$ be $\kappa_2$-quasar-convex function with respect to $h(\overline{x})$, where $\kappa_2 \in \, (0, 1]$. Suppose in addition that there exists $m \ge 0$ such that
\begin{equation}\label{eq:m-increase}
 \varphi(t) - \varphi(s) \ge m(t-s), \qquad \forall ~ t \ge s, ~ t,s \in I.
\end{equation}
Then, the function $g: \R^n \to \mathbb{R}$ given by $g := \varphi \circ h$
is $(\kappa, \gamma_g)$-strongly quasar-convex with respect to $\overline{x}$, where $\kappa := \kappa_1 \kappa_2$, $\gamma_g := \frac{C_{\kappa_1,\kappa_2}\, m\, \gamma}{\kappa_2}$ and $C_{\kappa_1,\kappa_2}> 0$ is the constant given in Lemma~\ref{lem:c-explicit}. 
\end{proposition}

\begin{proof}
Since $\overline{x} \in \argmin{ \mathbb{R}^n} h$ and $\varphi$ is nondecreasing, $\overline{x}\in \argmin{ \mathbb{R}^n}\,g$. 
Fix $x \in \dom{h}$ and $\lambda \in [0,1]$. Since $h$ is $(\kappa_1, \gamma)$-strongly quasar-convex with respect to $\overline{x}$, we have
\begin{equation*}
 h(\lambda \overline{x} + (1 - \lambda) x) \le \kappa_1 \lambda h(\overline{x}) + (1-\kappa_1 \lambda) h(x) - \lambda \left( 1 - \frac{ \lambda}{2 - \kappa_1} \right) \frac{\kappa_1 \gamma}{2} \|x - \overline{x} \|^2.
\end{equation*}
Set $t := \kappa_1 \lambda h(\overline{x}) + (1 - \kappa_1 \lambda) h(x)$ and  $s := h(\lambda \overline{x} + (1-\lambda)x)$. Then, $t \geq s$. Applying \eqref{eq:m-increase} implies $\varphi(s) \le \varphi(t) - m(t-s)$, i.e.,
\begin{equation}\label{eq:phi-step1}
 g(\lambda\overline{x}+(1-\lambda)x) \le \varphi \big(\kappa_1\lambda h(\overline{x}) + (1-\kappa_1 \lambda) h(x)\big) - \lambda \left(1 - \frac{\lambda}{2 - \kappa_1} \right) \frac{m \kappa_1 \gamma}{2}\|x-\overline{x}\|^2.
\end{equation}
Invoking the $\kappa_2$-quasar-convexity of $\varphi$ with respect to $h(\overline{x})$ ensures
\begin{equation}\label{eq:phi-step2}
 \varphi\big(\kappa_1 \lambda h(\overline{x})+(1-\kappa_1 \lambda)h(x) \big) \le \kappa_2 \kappa_1 \lambda g(\overline{x}) + (1-\kappa_2 \kappa_1 \lambda) g(x).
\end{equation}
Together with \eqref{eq:phi-step1}, \eqref{eq:phi-step2} implies
\begin{equation}\label{eq:g-pre}
 g(\lambda \overline{x} + (1-\lambda)x) \le \kappa_1 \kappa_2 \lambda g(\overline{x}) + (1-\kappa_1 \kappa_2 \lambda)g(x) - \lambda\left(1-\frac{ \lambda}{2 - \kappa_1}\right) \frac{m \kappa_1 \gamma}{2} \|x-\overline{x}\|^2.
\end{equation}
By the definition of $C_{\kappa_1,\kappa_2}$, and applying Lemma \ref{lem:c-explicit} in \eqref{eq:g-pre}, it can be concluded that
$$ g(\lambda \overline{x} + (1-\lambda)x) \le \kappa_1 \kappa_2 \lambda g(\overline{x}) +(1-\kappa_1 \kappa_2 \lambda)g(x) - \lambda\left(1 - \frac{\lambda}{2 - \kappa_1 \kappa_2} \right) \frac{C_{\kappa_1,\kappa_2} \, m \kappa_1 \gamma}{2} \|x-\overline{x} \|^2.$$
Since $\kappa = \kappa_1 \kappa_2$, this shows that $g$ is $(\kappa,\gamma_g)$-strongly quasar-convex with respect to $\overline{x}$.  
\end{proof}

\begin{remark}\label{rem:kappa1-one}
In Proposition \ref{increa:prop}, the case $\kappa_1 = 1$ can be handled by a scaling argument.  
Assume that $h$ is $(1,\gamma)$-strongly quasar-convex with respect to $\overline{x}$. Fix any $\theta \in (0,1)$, by Proposition \ref{Scaling}~\ref{Scaling:b}, the function $h$ is also $(\theta, \gamma/\theta)$-strongly quasar-convex with respect to $\overline{x}$. Applying Proposition \ref{increa:prop}, we obtain that $g = \varphi \circ h$ is $(\theta \kappa_2, \gamma_g (\theta))$-strongly quasar-convex with respect to $\overline{x}$, where $\gamma_g(\theta) := \frac{C_{\theta,\kappa_2}  m \gamma}{\theta \kappa_2}$ and $C_{\theta,\kappa_2} > 0$ is the constant from Lemma \ref{lem:c-explicit}.
\end{remark}


In the case $\gamma=0$, that is, when $h$ is $\kappa_1$-quasar-convex, condition~\eqref{eq:m-increase} reduces to $m=0$, and it is enough to assume that $\varphi$ is nondecreasing.

\begin{corollary}[\textbf{Quasar-convex composition}]\label{increa:coro}
 Let $h: \mathbb{R}^n \to \overline{\mathbb{R}}$ be a proper function, $\overline{x} \in \argmin{\mathbb{R}^n}\,h$, and let $I \subseteq \mathbb{R}$ be such that $h(\mathbb{R}^n) \subseteq I$.  Assume that $h$ is $\kappa_1$-quasar-convex with  respect to $\overline{x}$ for some $\kappa_1 \in \,  (0, 1]$, and let 
 $\varphi: I \to \mathbb{R}$ be nondecreasing and $\kappa_2$-quasar-convex  with respect to $h(\overline{x})$, for some $\kappa_2 \in \, (0, 1]$.  Then $\varphi \circ h$ is $(\kappa_1 \kappa_2)$-quasar-convex with respect to $\overline{x}$.
\end{corollary}
\begin{proof}
This follows directly from Proposition \ref{increa:prop} by taking $\gamma=0$.
\end{proof}
We illustrate Corollary~\ref{increa:coro} with the following example.

\begin{example}[\textbf{Power-root transformation}]\label{ex:root-increa}
Let $p \in \mathbb{N}$ with $p\ge 2$. Let $h: \R^n \to [0,+\infty)$, and let $\ov{x} \in \argmin{\mathbb{R}^n} h$ be such that $h(\ov{x})=0$. Assume that $h$ is $\kappa_1$-quasar-convex with respect to $\ov{x}$, where $\kappa_1 \in (0,1)$.
Define $\varphi(t) := \sqrt[p]{t}$ on $[0, + \infty[$ and set $g := \varphi \circ h$. The function $\varphi$ is nondecreasing, and it follows from \cite[Example 11]{deBrito2025extending} that
$\varphi$ is $\kappa_2$-quasar-convex at $0$, with $\kappa_2=1/p$.
 Hence, by Corollary \ref{increa:coro}, the function $g(x) = \sqrt[p]{h(x)}$ is $\frac{\kappa_1}{p}$-quasar-convex with respect to  $\overline{x}$.
\end{example}

\begin{remark}[\textbf{Comments on the composition rule}]
\begin{enumerate}[label=(\textbf{\alph*}), font=\normalfont\bfseries, leftmargin=0.7cm]
 \item  The assumption that $\varphi$ is quasar-convex in Proposition~\ref{increa:prop} is essential. Consider the one-dimensional case with $\ov{x}=0$, $h(x)=x^2$, $\varphi(t)=1-e^{-t}$ ($t\ge 0$), and $g:=\varphi\circ h$.
Then $h$ is $(1,2)$-strongly quasar-convex with respect to $\ov{x}$, and $\varphi$ is strictly increasing on $[0,+\infty)$. However, $g(x)=1-e^{-x^2}$ is not $\kappa$-quasar-convex with respect to $0$ for any $\kappa\in(0,1]$.
Indeed, suppose by contradiction that, for all $x\in\mathbb{R}$ and $\lambda\in[0,1]$,
\[
g((1-\lambda)x)\le\kappa \lambda g(0)+(1-\kappa \lambda)g(x).
\]
Since $g(0)=0$, choosing $\lambda=\tfrac{1}{2}$ yields
\[
1-e^{-x^2/4}\le\left(1-\frac{\kappa}{2}\right)(1-e^{-x^2}).
\]
Letting $|x|\to\infty$, we obtain $1 \le 1-\frac{\kappa}{2}$,
which is impossible. Hence $g$ is not $\kappa$-quasar-convex for any $\kappa\in(0,1]$.

 \item  When $h$ is quasar-convex and $\varphi$ is increasing, the composition $g := \varphi \circ h$ is star-quasiconvex by \cite[Theorem 11]{lara2026debreu} (see also \cite{khanh2025star,nguyen2026projected}). 
\end{enumerate}
\end{remark}

\subsection{First-order characterizations and further properties}


In this subsection, we derive first-order characterizations of strongly quasar-convex functions and establish several structural properties that will be used in the subsequent analysis. In the nonsmooth setting, strong quasar-convexity admits a characterization in terms of the Clarke subdifferential, extending the classical gradient-based formulation available in the differentiable case. This viewpoint clarifies the underlying geometry and provides a convenient tool for studying optimality conditions, error bounds, and convergence properties of high-order proximal-point methods.

We begin with the following auxiliary lemma, which will be instrumental in deriving the first-order characterization of strong quasar-convexity.

\begin{lemma}[\textbf{Chain rule along a segment}]\label{lem:clarke-chain}
 Let $h: \mathbb R^n \to \overline{\mathbb R}$ be a proper function with convex domain, and assume that $h$ is locally Lipschitz on $\dom{h}$. Let $x, y \in \dom{h}$, and for $\lambda \in [0,1]$ define
 $x_\lambda := (1-\lambda)x + \lambda y$, and $g(\lambda) := h(x_\lambda)$.
Then $g$ is absolutely continuous on $[0,1]$ and, for almost everywhere $\lambda \in (0,1)$, there exists $\xi_\lambda \in \partial^C h(x_\lambda)$ such that 
$$g'(\lambda) = \langle \xi_\lambda, y - x \rangle.$$
\end{lemma}

\begin{proof}
 Since $[x,y]$ is compact and $h$ is locally Lipschitz around each point of this segment, it follows that $h$ is Lipschitz on $[x,y]$. Hence, $g$ is Lipschitz on $[0,1]$, and therefore absolutely continuous and differentiable almost everywhere on $(0,1)$ \cite[Theorem 7.20]{Rudin-1987}. Since $x_\lambda$ is affine with derivative $x_\lambda' = y-x$, Clarke's chain rule \cite[Theorem~2.3.10]{Clarke1990} yields
\[
\partial^C g(\lambda) \subseteq \{ \langle \eta, y-x \rangle : \eta \in \partial^C h(x_\lambda) \}.
\]
At every point $\lambda$ where $g$ is differentiable, we have $\partial^C g(\lambda)=\{g'(\lambda)\}$. Therefore, for almost every $\lambda \in (0,1)$, there exists $\xi_\lambda \in \partial^C h(x_\lambda)$ such that
$g'(\lambda) = \langle \xi_\lambda, y-x \rangle$.
%
%
\end{proof}

When $h$ is locally Lipschitz, strong quasar-convexity admits the following characterization in terms of the Clarke subdifferential. This result extends \cite[Lemma~10]{Hinder2020Near} from the differentiable setting to locally Lipschitz functions.

\begin{theorem}[{\bf Clarke first-order characterization}]
\label{thm:char:clarke}
 Let $h: \mathbb{R}^{n} \to \overline{\mathbb{R}}$ be a proper function with convex domain, and let $\overline{x} \in \argmin{\mathbb R^n} h$. Assume that $h$ is locally Lipschitz on $\dom{h}$. Then $h$ is $(\kappa, \gamma)$-strongly quasar-convex with respect to $\overline{x}$, where $\gamma \ge 0$, if and only if
\begin{equation}\label{eq:foC}
 h(\overline x) \ge h(x) + \frac{1}{\kappa} \langle v, \overline x-x\rangle + \frac{\gamma}{2} \|x-\overline x\|^2, \qquad \forall ~ x \in \dom{h}, ~  \forall ~ v \in \partial^C h(x).
\end{equation}
\end{theorem}

\begin{proof}
See Appendix~\ref{sec:app:a0}.
\end{proof}

As an immediate consequence of Theorem~\ref{thm:char:clarke} (see relation \eqref{eq:g0-lower} in Appendix~\ref{sec:app:a0}), we obtain the following quadratic growth property; see also \cite[Corollary~1]{Hinder2020Near}.

\begin{corollary}[\textbf{Quadratic growth from the Clarke condition}]\label{clarke:qwc}
 Let $h: \mathbb{R}^{n} \to \overline{\mathbb{R}}$ be a proper function with convex domain, let $\overline{x} \in \argmin{\mathbb R^n}\,h$, $\kappa \in (0, 1]$ and $\gamma \geq 0$. Assume that $h$ is locally Lipschitz on $\dom{h}$. If relation \eqref{eq:foC} holds, then
\begin{equation}\label{qwc:revisited}
 h(x) \geq h(\overline{x}) + \frac{\kappa \gamma}{2(2-\kappa)} \|x-\overline x\|^2, \quad \forall~  x \in \R^{n}.
\end{equation}
That is, $h$ satisfies a quadratic growth condition with modulus $\frac{\kappa \gamma}{2(2 - \kappa)} \geq 0$. 
\end{corollary}

The following result shows that strongly quasar-convex functions do not admit spurious Clarke critical points. In particular, every Clarke critical point is a global minimizer.

\begin{proposition}[\textbf{Clarke critical points are global minimizers}]\label{crit:are:glob}
 Let $h:\mathbb{R}^n \to \overline{\mathbb{R}}$ be a proper function with convex domain, and let $\overline{x} \in \argmin{\mathbb{R}^n} h$. Assume that $h$ is $(\kappa,\gamma)$-strongly quasar-convex with respect to $\overline{x}$, where $\gamma \ge 0$, and that $h$ is locally Lipschitz on $\dom h$. If $0 \in \partial^C h(x)$, then $x \in \argmin{\mathbb{R}^{n}} h$.
\end{proposition}
\begin{proof}
Invoking Theorem~\ref{thm:char:clarke}, condition \eqref{eq:foC} holds. Taking $v=0 \in \partial^C h(x)$ in \eqref{eq:foC}, we obtain
\[
 h(\overline{x}) \ge h(x) + \frac{\gamma}{2} \|x - \overline{x}\|^2.
\]
In particular, $h(\overline{x}) \ge h(x)$. Since 
$\overline{x} \in \argmin{\mathbb{R}^{n}} h$, it is clear that
$h(\overline{x}) \le h(x)$, i.e., $h(x) = h(\overline{x})$, and therefore $x \in \argmin{\mathbb{R}^{n}} h$. 
\end{proof}
For the next result, recall that a one-dimensional function is called unimodal if it decreases monotonically to its minimizer and then increases monotonically thereafter.
The following result extends the necessary condition in \cite[Observation 1]{Hinder2020Near} from the differentiable to the nonsmooth setting.

\begin{proposition}[\textbf{Unimodality on an interval}]\label{unimodal}
 Let $f: [a,b] \to \overline{\mathbb R}$ be a proper and locally Lipschitz function, and let  $\overline{x} \in \argmin{[a,b]} f$. If $f$ is $\kappa$-quasar-convex with respect to $\overline{x}$, with $\kappa \in (0,1]$, then $f$ is unimodal and every Clarke critical point is a global minimizer.
\end{proposition}
\begin{proof}
Let $a\leq x_1 < x_2 \leq \overline{x}$ and define
$\lambda := \frac{x_2 - x_1}{\overline{x} - x_1} \in [0,1]$.
Then $x_2 = \lambda \overline{x} + (1-\lambda)x_1$. Using the $\kappa$-quasar-convexity of $f$ and  $f(\overline{x})\le f(x_1)$, it follow that
\[
f(x_2) \leq \kappa \lambda f(\overline{x}) + (1-\kappa \lambda) f(x_1) \leq f(x_1),
\]
i.e., $f$ is nonincreasing on $[a,\overline{x}]$.

Similarly, let $\ov{x} \le x_1 < x_2 \le b$ and define $\lambda = \frac{x_2 - x_1}{x_2 - \overline{x}} \in [0,1]$.
Then $x_1 = \lambda \overline{x} + (1-\lambda)x_2$. Together with the $\kappa$-quasar-convexity of $f$,  $f(\overline{x})\le f(x_2)$ ensures
$$f(x_1) \leq \kappa \lambda f(\overline{x}) + (1-\kappa \lambda) f(x_2) \leq f(x_2),$$
which shows that $f$ is nondecreasing on $[\overline{x},b]$. Therefore, $f$ is unimodal.
The statement regarding the Clarke critical points follows from Proposition~\ref{crit:are:glob} applied with $\gamma=0$.
\end{proof}

The converse of Proposition~\ref{unimodal} fails in general, even for $C^k$ functions with $k\ge 1$, as shown in the following example. This also shows that the converse statement suggested by \cite[Observation~1]{Hinder2020Near} does not hold in general.

\begin{example}[\textbf{A smooth unimodal function that is not quasar-convex}] \label{ex:Ck-counterexample}
 Fix $k \in \mathbb N$ and define the function $f: [-1, 1] \to \mathbb R$ by
$$
 f(x) := \int_0^x \phi(t)\,dt,
$$
where
$$
\phi(t) := 
\begin{cases}
 t^{2k-1} \sin^2(1/t)+t^{2k+1}, & t\neq 0,\\
 0, & t=0.
\end{cases}
$$
Clearly, $\phi \in C^{k-1}([-1,1])$, and therefore $f \in C^{k} ([-1,1])$. Moreover, $f'(0)=\phi(0)=0$, and for $x\neq 0$,
\[
f'(x) = x^{2k-1} \sin^2(1/x) + x^{2k+1}.
\]
Therefore, $f'(x) < 0$ on $(-1, 0)$ and $f'(x) > 0$ on $(0, 1)$,
i.e., $f$ is unimodal. In particular, $x=0$ is its unique critical point and global minimizer. 

We claim that $f$ is not $\kappa$-quasar-convex at $x=0$ for any $\kappa\in(0,1]$. Suppose, by contradiction, that $f$ is $\kappa$-quasar-convex with respect to $x=0$ for some $\kappa\in(0,1]$. Since $f$ is differentiable, applying \eqref{diff:squasar} for $\gamma=0$ implies
$$ f(0) \ge f(x)+\frac{1}{\kappa} f'(x)(0-x), \qquad~  \forall ~ x \in[-1,1].
$$
Since $f(0)=0$ and $f(x)>0$ for every $x\in(0,1]$, we obtain
\begin{equation}\label{for:kappa}
 0 < \kappa \le \frac{x f'(x)}{f(x)},\qquad ~ \forall ~ x \in (0,1].
\end{equation} 
Setting $x_n:=1/(n\pi)$ and $\sin(1/x_n)=0$, we come to $x_n f'(x_n)=x_n^{2k+2}$.
On the other hand,
\begin{align*}
 f(x_n) \ge \int_{x_n/2}^{x_n} t^{2k-1} \sin^{2}  (1/t)\,dt & = \int_{n \pi}^{2n \pi} \frac{\sin^{2} (u)}{u^{2k+1}}\,du \ge \frac{1}{(2n\pi)^{2k+1}} \int_{n \pi}^{2n \pi} \sin^2 (u) \, du = \frac{x_n^{2k}}{2^{2k+2}}.
\end{align*}
Substituting this into \eqref{for:kappa}, it can be concluded that
\[
\frac{x_n f'(x_n)}{f(x_n)}\le 2^{2k+2}x_n^2 \to 0,
\]
which contradicts $\kappa>0$.
Therefore, $f$ is not $\kappa$-quasar-convex with respect to $x=0$ for any $\kappa\in(0,1]$.
\end{example}

We recall the (nonsmooth) Polyak--{\L}ojasiewicz (PL) property (see, e.g., \cite{bolte2007lojasiewicz}), which plays a central role in first-order convergence analysis. A proper function $h:\mathbb{R}^n \to \overline{\mathbb{R}}$ with $\overline{x} \in \argmin{\mathbb{R}^n} h$ is said to satisfy the PL property with modulus $\mu>0$ if
\begin{equation}\label{PL:property}
 \frac{1}{2} \dist^{2} (0,\partial^C h(x)) \; \ge \mu (h(x) - h(\overline{x})), \qquad \forall ~ x \in  
 \dom{h}.
\end{equation}
In the next two results, we assume that the curvature parameter $\gamma$ is strictly positive.

\begin{proposition}[{\bf PL inequality}]\label{prop:nonsmooth-PL}
 Let $h: \mathbb R^n \to \overline{\mathbb{R}}$ be a proper function with convex domain, and let $\overline{x} \in \argmin{\mathbb{R}^{n}} h$. Assume that $h$ is $(\kappa, \gamma)$-strongly quasar-convex  with respect to $\overline{x}$, where $\kappa\in(0,1]$ and $\gamma>0$, and that $h$ is locally Lipschitz on $\dom{h}$. Then $h$ satisfies the PL property with modulus $\mu := \gamma\kappa^2 > 0$.
\end{proposition}

\begin{proof}
 Fix $x\in \dom{h}$ and let $v\in\partial^{C}h(x)$. It follows from Theorem \ref{thm:char:clarke} that
 $$ h(x) - h(\overline{x}) \le \frac{1}{\kappa} \langle v, x - \overline{x} \rangle - \frac{\gamma}{2} \|x - \overline{x}\|^2. $$
Applying Young's inequality,
$\frac{1}{\kappa} \langle v, x - \overline{x} \rangle \le \frac{1}{2\gamma \kappa^2} \|v\|^2 + \frac{\gamma}{2} \|x-\overline{x} \|^2$, leads to
 
 $$ h(x) - h(\overline{x}) \leq \frac{1}{2 \gamma \kappa^2} \|v\|^2.$$
 Since this inequality holds for every $v \in \partial^C h(x)$, taking infimum over $v\in\partial^C h(x)$ yields
 \[
 h(x)-h(\overline{x}) \le \frac{1}{2 \gamma \kappa^2} \dist^{2} (0,\partial^C h(x)), 
 \]
 giving our desired result.
\end{proof}

Next, we show that strongly quasar-convex functions satisfy both value-based and subgradient-based error bounds. In particular, the distance to the solution set can be controlled either by function values or by the magnitude of the Clarke subdifferential.
\begin{proposition}[{\bf Quadratic growth and error bound}]\label{prop:error-bounds}
Let $h:\mathbb{R}^n \to \overline{\mathbb{R}}$ be a proper function and let $\overline x \in \argmin{\mathbb{R}^n} h$. Assume that $h$ is $(\kappa, \gamma)$-strongly quasar-convex  with respect to $\overline{x}$, where $\kappa \in (0,1]$ and $\gamma>0$, and that $h$ is locally Lipschitz on $\dom h$. Then, for every $x \in \dom{h}$, it holds that
\begin{equation} \label{eq:eb-value-ext}
 \|x-\overline x\| \le \sqrt{\frac{2(2-\kappa)}{\kappa \gamma}} \, \sqrt{ h(x) - h(\overline x)},
\end{equation}
and
\begin{equation}
\label{eq:eb-subgrad-ext}
 \|x-\overline x\| \le \frac{2}{\kappa \gamma}\,
 \dist \big(0, \partial^C h(x)\big).
\end{equation}
\end{proposition}

\begin{proof}
It follows from the quadratic growth estimate \eqref{qwc:revisited} that
 $$h(x) - h(\overline x) \ge \frac{\kappa \gamma}{2(2-\kappa)} \|x-\overline x\|^2, \qquad \forall ~ x \in 
\dom{h},$$
which is equivalent to \eqref{eq:eb-value-ext}.
Now, fix $x \in \dom{h}$ and let $v \in \partial^C h(x)$. Invoking Theorem~\ref{thm:char:clarke} ensures
$$ h(\overline x) \ge h(x) + \frac{1}{\kappa} \langle v, \overline x-x \rangle + \frac{\gamma}{2} \|x-\overline x\|^2.$$
Since $\overline{x} \in \argmin{}\,h$, we have $h(\overline x) \le h(x)$, hence using the Cauchy--Schwarz inequality,
\begin{align*}
 \frac{\gamma}{2} \|x-\overline{x} \|^2 \le - \frac{1}{\kappa} \langle v,\overline{x} - x \rangle \le \frac{1}{\kappa} \|v\| \, \|x-\overline{x} \|.
\end{align*}
If $x \neq \overline{x}$, dividing by $\|x-\overline{x}\|$ yields
\[
\|x - \overline{x}\| \le \frac{2}{\kappa\gamma} \|v\|.
\]
The inequality is trivial if $x=\overline{x}$. Taking the infimum for $v \in \partial^C h(x)$ leads to \eqref{eq:eb-subgrad-ext}.
\end{proof}

\subsection{Landscape of (strongly) quasar-convex functions}\label{subsec:no-saddle-dini}
We next refine the geometric picture of quasar-convex functions by focusing on their stationary structure and the shape of their minimizer sets. This viewpoint is useful for two reasons. First, it does not require local Lipschitz continuity and therefore applies beyond the Clarke-subdifferential framework (see Remark~\ref{rem:no_saddle} below). Second, it further clarifies the benign landscape of quasar-convex functions: Proposition~\ref{maxminlocal} already shows that every local minimizer is global and that strict local maximizers do not exist, while the result below excludes saddle points defined through Dini-stationarity.

Another important distinction concerns the geometry of the minimizer set itself. If $h$ is $(\kappa,\gamma)$-strongly quasar-convex with $\gamma>0$, then $\argmin{\R^n} h$ is a singleton by Remark~\ref{rem:on_beg_prop}~\ref{rem:on_beg_prop:b}.
Hence, this class does not admit nonisolated global minimizers. In contrast, in the quasar-convex case $(\gamma=0)$, the set of global minimizers may be nonisolated; if it is not singleton, then it is star-convex by Remark~\ref{rem:on_beg_prop}~\ref{rem:on_beg_prop:a}. Thus, even in the absence of strong quasar-convexity, quasar-convex functions still
enjoy a highly structured and benign landscape.

We first recall the lower Dini directional derivative. For a proper function
$h:\R^n\to\Rinf$, a point $x\in \dom{h}$, and a direction
$d\in\R^n$, define
\[
h^{D-}(x;d):=\liminft{t\downarrow 0}\frac{h(x+td)-h(x)}{t}.
\]
\begin{definition}[\textbf{Stationary point in the Dini sense}]
A point $x\in\dom h$ is called a \textit{stationary point} of $h$ if
\[
h^{D-}(x;d)\ge 0 \qquad \forall d\in\mathbb{R}^n.
\]
The set of all such points is denoted by  $\bs{\rm Dcrit}(h)$.
\end{definition}
The next result shows that quasar-convexity rules out such saddle points.
\begin{theorem}[\textbf{No Dini saddle points}]\label{prop:no_saddle}
Let $h:\R^n\to\Rinf$ be proper, and let $\ov{x}\in \argmin{\mathbb{R}^n} h$.
\begin{enumerate}[label=(\textbf{\alph*}), font=\normalfont\bfseries, leftmargin=0.7cm]
\item \label{prop:no_saddle:a} If $h$ is $\kappa$-quasar-convex
with respect to $\ov{x}$, where $\kappa\in(0,1]$. Then every Dini-stationary point of $h$ is a global minimizer, i.e., $\bs{\rm Dcrit}(h)\subseteq \argmin{\mathbb{R}^n} h$. Consequently, $h$ has no saddle points in the Dini sense.

\item \label{prop:no_saddle:b} If $h$ is $(\kappa,\gamma)$-strongly quasar-convex with $\gamma>0$, then
$\bs{\rm Dcrit}(h)=\{\ov{x}\}$, that is, the unique stationary point is the unique global minimizer.
\end{enumerate}
\end{theorem}
\begin{proof}
\ref{prop:no_saddle:a} 
Let $x\in \bs{\rm Dcrit}(h)$ and assume, by contradiction, that $x\notin \argmin{\R^n} h$. Then
$h(x)>h(\ov{x})$. Set $d:=\ov{x}-x$. For every $t\in(0,1]$, by $\kappa$-quasar-convexity of $h$ with respect to $\ov{x}$,
we have
\[
h(x+td)=h((1-t)x+t\ov{x})
\le \kappa t\,h(\ov{x})+(1-\kappa t)\,h(x)
= h(x)-\kappa t\bigl(h(x)-h(\ov{x})\bigr).
\]
Therefore,
\[
\frac{h(x+td)-h(x)}{t}
\le -\kappa\bigl(h(x)-h(\ov{x})\bigr)<0.
\]
Passing to the lower limit as $t\downarrow 0$, we obtain $h^{D-}(x;d)<0$, which contradicts the stationarity of $x$. Hence
$x\in \argmin{\R^n} h$, and thus $\bs{\rm Dcrit}(h)\subseteq \argmin{\R^n} h$.
Since every stationary point is a global minimizer, it cannot be a saddle point.
\\
\ref{prop:no_saddle:b} If $h$ is $(\kappa,\gamma)$-strongly quasar-convex with $\gamma>0$, then $\argmin{\R^n} h$ is a singleton by Remark~\ref{rem:on_beg_prop}~\ref{rem:on_beg_prop:b}. Hence Assertion~\ref{prop:no_saddle:a} yields
$\bs{\rm Dcrit}(h)\subseteq \{\ov{x}\}$, while $\ov{x}\in \bs{\rm Dcrit}(h)$ is immediate from the definition of the minimizer. Therefore, $\bs{\rm Dcrit}(h)=\{\ov{x}\}$.
\end{proof}

\begin{remark}[\textbf{Dini and Clarke stationary points}]\label{rem:no_saddle}
If $h$ is locally Lipschitz on $\dom{h}$, then by Theorem~\ref{thm:char:clarke} and Proposition~\ref{crit:are:glob}, every Clarke critical point of a $(\kappa,\gamma)$-strongly quasar-convex function is a global minimizer. In particular, there are no spurious stationary points in the Clarke sense. In the strongly quasar-convex case, Remark~\ref{rem:on_beg_prop}~\ref{rem:on_beg_prop:b} implies that this Clarke critical point is unique. Therefore, for locally Lipschitz quasar-convex functions, both the Dini-stationary and the Clarke-critical viewpoints lead to the same conclusion: the only stationary
objects are global minimizers, and in the strongly quasar-convex case the unique stationary point is the unique global minimizer.
\end{remark}

 In the following example, we exhibit a nonconvex quasar-convex function whose set of global minimizers is nonisolated. More precisely, the minimizer set is star-convex.

 \begin{example}[\textbf{Nonisolated star-convex  minimizers}]\label{ex:star_shaped_min}
 Let $R(\theta):=1+0.35\cos(4\theta)$ and, for $x\neq 0$, written in polar form $x=r(\cos\theta,\sin\theta)$, define
 $\rho(x):=\frac{r}{R(\theta)}$ with $\rho(0):=0$, and $h(x):=\max\{0,\rho(x)-1\}$. Then,  for every
 $t\in[0,1]$ and $x\in\R^2$, we have $\rho(tx)=t\rho(x)$, and therefore
 \[
 h(tx)=\max\{0,t\rho(x)-1\}\le t\max\{0,\rho(x)-1\}=t\,h(x).
 \]
 Hence $h$ is $1$-quasar-convex with respect to $\ov x=0$. Moreover,
 \[
 \argmint{x\in\R^2} h(x)=\left\{r(\cos\theta,\sin\theta):0\le r\le R(\theta),\ \theta\in[0,2\pi)\right\},
 \]
 which is star-convex with respect to the origin and not convex. In particular, this example
 shows that a quasar-convex function may be genuinely nonconvex and may possess a
 nonisolated minimizer set that is star-convex but not convex. A visualization of this example is provided in Figure~\ref{fig:quasar_landscape}, which
 highlights both the nonconvexity of the landscape and the star-convex, nonconvex geometry
 of the minimizer set.
 \end{example}

 \begin{figure}
     \centering
 \includegraphics[width=\textwidth]{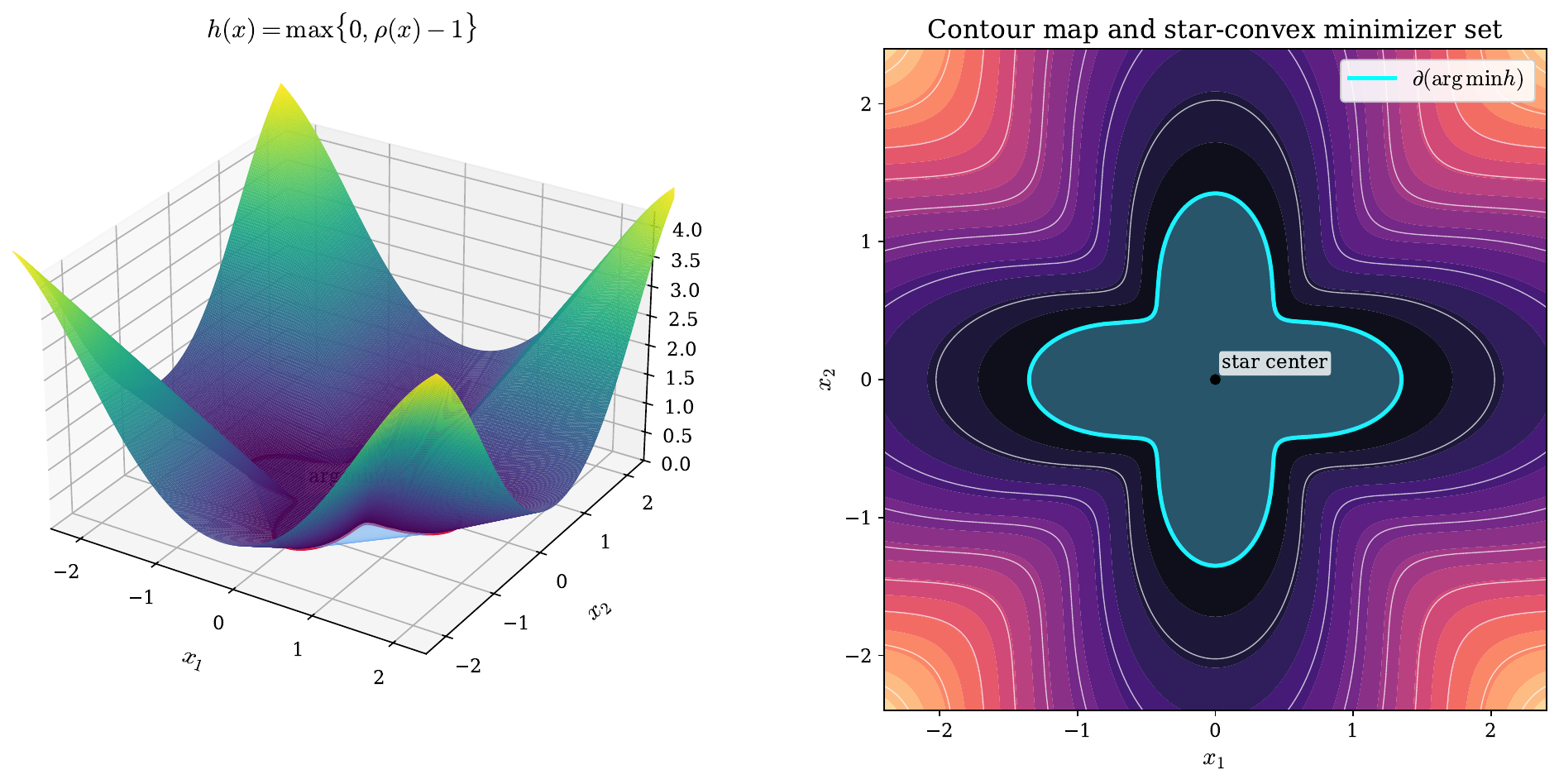}
     \caption{A nonconvex quasar-convex function whose minimizer set is star-convex  but not convex; see Example~\ref{ex:star_shaped_min}. }
     \label{fig:quasar_landscape}
 \end{figure}


\section{High-order proximal-point methods}\label{sec:04}
In this section, we study the high-order proximal-point algorithm (HiPPA) (see, e.g., \cite{Ahookhosh24,Ahookhosh23,Ahookhosh2025Asymptotic,Kabgani25itsdeal,Kabgani2025Moreau,Kabgani24itsopt,nesterov2021inexact,Nesterov2023a}), in which the regularization term is of the form $\|\cdot\|^{p}$ with $p>1$. As noted in \cite{Ahookhosh2025Asymptotic}, HiPPA exhibits superlinear convergence for uniformly quasiconvex functions, a class that includes uniformly convex and strongly quasiconvex functions. We show here that HiPPA with $p>2$ also attains superlinear convergence for strongly quasar-convex functions.

Let us begin by introducing the high-order proximal operator, which serves as the fundamental building block of the algorithmic framework developed in this section. It can be viewed as a natural generalization of the classical proximal mapping (corresponding to $p=2$), where the quadratic regularization is replaced by a higher-order power term

\begin{definition}[{\bf High-order proximal operator}]\label{def:hord:prox}
Let $p>1$, $\beta>0$, and let $h: \R^n \to \Rinf$ be a proper lsc function. The high-order proximal operator (HOPE) of $h$ with parameter $\beta$, denoted 
    $\prox{h}{\beta}{p}(\cdot): \R^n \rightrightarrows \R^n$, is defined as 
   \begin{equation}\label{eq:Hiorder-Moreau prox}
    \prox{h}{\beta}{p} (x) := \argmint{y\in \R^n} \left\{h(y)+\frac{1}{p\beta}\Vert x- y\Vert^p\right\}.
    \end{equation}     
\end{definition}
In general, the mapping $\prox{h}{\beta}{p}$ is set-valued. However, if
$h$ is convex, then the function
\[
y \mapsto h(y) + \frac{1}{p\beta}\|x - y\|^p
\]
is strictly convex for $p>1$, and therefore admits a unique minimizer.
In this case, $\prox{h}{\beta}{p}$ is single-valued (see, for instance,
\cite[Proposition~12.15]{Bauschke17}).

We next introduce the notion of a fixed point of the high-order proximal mapping, which will play a central role in the analysis of the iterative scheme. Under suitable assumptions, fixed points of $\prox{h}{\beta}{p}$ characterize minimizers of the original problem.

\begin{definition}[{\bf Proximal fixed point}]\label{def:fixedpoint}
Let $p>1$, $\beta>0$, and let $h: \R^n \to \Rinf$ be a proper lsc function. 
A point $x^*\in \dom{h}$ is called a \textit{proximal fixed point} if $x^*\in \prox{h}{\beta}{p}(x^*)$. The set of all such points is denoted by
$\Fix(\prox{h}{\beta}{p})$.
\end{definition}

The following theorem collects basic properties of HOPE. Note that the result holds for 
$(\kappa, \gamma)$-strongly quasar-convex functions with modulus $\gamma \geq 0$, that is, including the merely quasar-convex case.

\begin{theorem}[{\bf Properties of HOPE}]\label{th:nonemprox}
Let $h:\R^n \to \Rinf$ be a proper lsc function that is
$(\kappa,\gamma)$-strongly quasar-convex with $\kappa \in (0,1]$ and $\gamma\geq 0$.
The following statements hold:
\begin{enumerate}[label=(\textbf{\alph*}), font=\normalfont\bfseries, leftmargin=0.7cm]
\item\label{th:nonemprox:a} For each $\widehat{x}\in \R^n$, $\beta > 0$, and $p>1$, the set $\prox{h}{\beta}{p}(\widehat{x})$ is nonempty and compact;
\item\label{th:nonemprox:b} For each $\beta > 0$ and $p>1$,
\[
\argmint{x\in \R^n} \, h(x)= \Fix(\prox{h}{\beta}{p}).
\]
\item\label{th:nonemprox:c}
Let $\{x^k\}_{k\in \mathbb N}\subseteq \R^n$ and $\{\beta_k\}_{k\in \Nz}\subseteq (0,+\infty)$ satisfy
$x^k\to \widehat{x}$ and $\beta_k\to \beta>0$.
If $y^k \in \prox{h}{\beta_k}{p}(x^k)$ for all $k$, then the sequence $\{y^k\}_{k\in\Nz}$ is bounded, and all its cluster points belongs to $\prox{h}{\beta}{p}(\widehat{x})$.
\end{enumerate}
\end{theorem}
\begin{proof}
\ref{th:nonemprox:a} 
Fix $\widehat{x}\in \R^n$, $\beta>0$, and define
\[
F(y):= h(y)+\frac{1}{p\beta}\|y-\widehat{x}\|^p.
\]
If $\gamma>0$, then $h$ is $2$-supercoercive by \cite[Proposition~12]{deBrito2025extending}, and therefore coercive. Therefore $F$ is coercive and lsc.
If $\gamma=0$, then $h$ is quasar-convex and admits a global minimizer,
hence it is bounded from below. Since $p>1$, the function
$\|y - \widehat{x}\|^p$ is coercive, and thus the sum $h(y) + \frac{1}{p\beta}\|y - \widehat{x}\|^p$
is coercive.
In both cases, $F$ is coercive and lsc.
By \cite[Corollary~1.10]{Rockafellar09}, the compactness of sublevel sets of theses functions implies that
\[
\argmint{y\in \R^n} \left\{h(y)+\frac{1}{p\beta}\Vert \widehat{x}- y\Vert^p\right\}\neq \emptyset,
\]
and $\prox{h}{\beta}{p} (\widehat{x})$ is compact.
\\
\ref{th:nonemprox:b} Let $\widehat{x} \in \Fix(\prox{h}{\beta}{p})$ and $\beta>0$ be arbitrary. Then,
\begin{equation}
 h(\widehat{x}) \leq h(x) + \frac{1}{p \beta} \lVert \widehat{x} - x \rVert^{p}, \qquad\forall ~ x \in \mathbb{R}^{n}. \nonumber
\end{equation}
Take $x_{\lambda} := \lambda \ov{x} + (1-\lambda) \widehat{x}$ with $\ov{x} \in \argmin{ \mathbb{R}^{n}}\,h$ and $\lambda \in (0,1)$. From  $(\kappa,\gamma)$-strong quasar-convexity of $h$, for any $\lambda\in (0,1)$, we have
\begin{align*}
  h(\widehat{x}) &\leq h(\lambda \ov{x} + (1-\lambda) \widehat{x}) + \frac{1}{p \beta} \lVert \lambda (\ov{x} - \widehat{x} ) \rVert^{p} \\
 & \leq \lambda \kappa h(\ov{x}) + (1-\lambda \kappa) h(\widehat{x}) - \lambda \left( 1 - \frac{\lambda}{2-\kappa} \right) \frac{\kappa \gamma}{2} \lVert \widehat{x} - \ov{x} \rVert^{2} + \frac{\lambda^{p}}{p \beta} \lVert \widehat{x} - \ov{x} \rVert^{p}.
\end{align*}
This leads to,
\[
\lambda \kappa (h(\widehat{x}) - h(\ov{x})) \leq - \lambda \left( 1 - \frac{\lambda}{2-\kappa} \right) \frac{\kappa \gamma}{2} \lVert \widehat{x} - \ov{x} \rVert^{2} + \frac{\lambda^{p}}{p \beta} \lVert \widehat{x} - \ov{x} \rVert^{p},
\]
which in turn implies,
\[
h(\widehat{x}) - h(\ov{x}) \leq - \left( 1 - \frac{\lambda}{2-\kappa} \right) \frac{\gamma}{2} \lVert \widehat{x} - \ov{x} \rVert^{2} + \frac{\lambda^{p-1}}{p \kappa \beta} \lVert \widehat{x} - \ov{x} \rVert^{p}.
\]
By letting $\lambda \downarrow 0$, it comes to
\[
h(\widehat{x}) - h(\ov{x}) \leq - \frac{\gamma}{2} \lVert \widehat{x} - \ov{x} \rVert^{2},
\]
i.e., $h(\widehat{x}) \leq h(\ov{x})$, thus $\widehat{x} \in \argmin{\mathbb{R}^{n}}\,h$.
Now, let $\ov{x} \in \argmin{ \mathbb{R}^{n}}\,h$. For every $y \in \mathbb{R}^{n}$,
$$h(\ov{x}) + \frac{1}{p \beta} \lVert \ov{x} - \ov{x} \rVert^{p} = h(\ov{x}) \leq h(y) \leq h(y) + \frac{1}{p \beta} \lVert y - \ov{x} \rVert^{p},$$
and the result follows.
\\
\ref{th:nonemprox:c} Since $h$ is quasar-convex, it admits a global minimizer, hence it is bounded from below. Since it is also lsc, it satisfies the assumptions of \cite[Theorem~3.4]{Kabgani24itsopt}, and the result follows from \cite[Theorem~1.17~(b)]{Rockafellar09}.
\end{proof}

We are now in a position to introduce the high-order proximal-point algorithm (HiPPA) associated with the mapping $\prox{h}{\beta}{p}$. Starting from an initial point $x^0 \in \R^n$, the method generates a sequence $\{x^k\}_{k\in\Nz}$ by successively solving high-order proximal subproblems with parameters $\{\beta_k\}_{k\in\Nz}$. In view of Theorem~\ref{th:nonemprox}~\ref{th:nonemprox:b}, if the stopping criterion $x^{k+1}=x^k$ is satisfied, then $x^k$ is a proximal fixed point and therefore a global minimizer of $h$.

\begin{algorithm}[H]
\caption{HiPPA (High-Order Proximal-Point Algorithm)}\label{ppa:sqcx:2}
\begin{description}
 \item[Step 0.] Choose an initial point $x^0 \in \R^n$, a parameter $p>1$, and a sequence
$\{\beta_k\}_{k\in\Nz} \subseteq (0,+\infty)$ bounded away from zero. Set $k=0$.
 \item[Step 1.] Choose
 \begin{equation}\label{step:sqcx}
  x^{k+1} \in \prox{h}{\beta_k}{p} (x^{k}).
 \end{equation}

 \item[Step 2.] If $x^{k+1} = x^{k}$, then STOP; in this case, $x^{k} \in \argmin{\R^n} h$. Otherwise, set $k=k+1$ and return to \textbf{Step 1}.
 \end{description}
\end{algorithm}

A natural stopping criterion for Algorithm~\ref{ppa:sqcx:2} is $x^{k+1}=x^k$. In practice, however, it is typically replaced by the condition $\|x^{k+1}-x^k\|\le \varepsilon$, where $\varepsilon>0$ is a prescribed tolerance. The following result provides an upper bound on the number of iterations required to satisfy this practical stopping rule.

\begin{proposition}[{\bf Iteration complexity}]\label{prop:reqexe:p}
Let $h: \R^n \to \Rinf$ be a proper, lsc, and $(\kappa,\gamma)$-strongly
quasar-convex function with respect to some $\ov{x} \in \argmin{\mathbb{R}^{n}}\,h$, where $\kappa \in (0,1]$ and
$\gamma \geq 0$. Assume that
\begin{equation}\label{param:cond:2}
0<\beta'\leq \beta_k\leq \beta'', \qquad \forall k\in\Nz.
\end{equation}
 Let $\{x^k\}_{k \in \Nz}$ be the sequence generated by Algorithm \ref{ppa:sqcx:2}. Then, for every $\varepsilon>0$, the stopping criterion $\|x^{k+1}-x^k\|\le \varepsilon$ is satisfied after at most
 \begin{equation}\label{comp:bound:p}
 \left\lceil\frac{p \beta^{\prime \prime} (h(x^0) - \min_{x \in \mathbb{R}^{n}} h(x))}{\varepsilon^{p}}\right\rceil,
 \end{equation}
 iterations.
\end{proposition}
\begin{proof}
If $x^{k+1}=x^k$ for some $k$, then the conclusion is immediate. Thus, assume that $x^{k+1} \neq x^k$ for all $k$.  Since 
$x^{k+1} \in \prox{h}{\beta_k}{p}(x^k)$, we have
\begin{equation*}\label{eqnew:teles}
 h (x^{k+1}) < h(x^{k+1}) + \frac{1}{p \beta_k} \Vert x^{k+1} - x^k \Vert^p \leq h(x^k),
\end{equation*}
which implies
\begin{equation}\label{tel:one}
 \Vert x^{k+1} - x^k \Vert^p \leq p \beta_k (h(x^k) - h(x^{k+1})),\qquad ~ \forall ~ k \in \mathbb{N}_0.
\end{equation}
Summing \eqref{tel:one} for $k=0,\dots,N-1$ and using $\beta_k\le \beta''$, we obtain
\[
\sum_{k=0}^{N-1}\Vert x^{k+1} - x^k\Vert^{p} \leq p \beta^{\prime \prime} (h(x^0) - \mathop{\bs\min}\limits_{x \in \mathbb{R}^{n}} h(x)),
\]
i.e.,
\[
N \mathop{\bs\min}\limits_{0\leq k\leq N-1} \Vert x^{k+1} - x^k\Vert^{p} \leq p \beta^{\prime \prime} (h(x^0) - \mathop{\bs\min}\limits_{x \in \mathbb{R}^{n}} h(x)).
\]
Hence, if
\[
N \ge \frac{p\beta''\bigl(h(x^0)-\mathop{\bs\min}\limits_{x\in\R^n} h(x)\bigr)}{\varepsilon^p},
\]
then
\[
\mathop{\bs\min}\limits_{0\le k\le N-1}\|x^{k+1}-x^k\| \le \varepsilon.
\]
In other words, there exists some $k\in\{0,\dots,N-1\}$ such that $\|x^{k+1}-x^k\|\le \varepsilon$, adjusting our desired result.
\end{proof}
Note that Proposition~\ref{prop:reqexe:p} provides an iteration bound for the
practical stopping criterion based on the successive iterations. It does not,
by itself, quantify objective accuracy or distance to the minimizer.

In the subsequent subsections, we analyze the behavior of the sequence $\{x^k\}_{k\in\Nz}$ generated by Algorithm~\ref{ppa:sqcx:2}. We first establish basic properties of the iterations, including well-posedness and descent
properties of the proximal steps. We then prove global convergence of the method under quasar-convexity assumptions. Finally,  we derive convergence rates and complexity estimates,
highlighting the influence of the parameter $p$ on the asymptotic behavior
of the algorithm.

\subsection{Global convergence}
As the first step, we establish fundamental properties of the sequence generated by Algorithm~\ref{ppa:sqcx:2}, including descent, the summability of the steps, boundedness, and convergence to global minimizers.

\begin{theorem}[{\bf Convergence analysis}]\label{basic:prop:hippa}
Let $h: \R^n \to \Rinf$ be a proper, lsc, and $(\kappa,\gamma)$-strongly
quasar-convex function with respect to some
$\ov{x} \in \argmin{\mathbb{R}^{n}}\,h$, for $\kappa \in (0,1]$ and
$\gamma \geq 0$. Assume that relation \eqref{param:cond:2} holds.
Let $\{x^k\}_{k\in\Nz}$ be generated by Algorithm~\ref{ppa:sqcx:2}, with $p>1$. Then the following assertions hold:
\begin{enumerate}[label=(\textbf{\alph*}), font=\normalfont\bfseries, leftmargin=0.7cm]
 \item\label{basic:prop:hippa:a} The sequence $\{h(x^k)\}_{k \in \Nz}$ is  decreasing.
 
\item\label{basic:prop:hippa:b} $\sum_{k=0}^{\infty} \Vert x^{k+1} - x^k \Vert^p < \infty$.

\item\label{basic:prop:hippa:c} The sequence $\{x^k\}_{k \in \Nz}$ is bounded  and the sequence 
$\{\|x^k-\ov x\|\}_{k\in\Nz}$ is nonincreasing.

\item\label{basic:prop:hippa:d1} The sequence $\{h(x^k)\}_{k\in\Nz}$ converges to $\min_{x\in \mathbb{R}^n} h(x)$.

\item \label{basic:prop:hippa:d}
Every cluster point of $\{x^k\}_{k\in\Nz}$ belongs to $\argmin{\mathbb{R}^{n}}\,h$.

 \item\label{basic:prop:hippa:e1} 

If $\gamma=0$ and $h$ is $\kappa$-quasar-convex with respect to every
$\ov x\in\argmin{\mathbb{R}^{n}}\,h$, then $\{x^k\}_{k\in\Nz}$ converges to a point in
$\argmin{\mathbb{R}^{n}}\,h$.

 \item\label{basic:prop:hippa:e2}  If $\gamma>0$, then $\{x^k\}_{k \in \Nz}$ converges to $\{\ov{x}\} = \argmin{\mathbb{R}^{n}}\,h$.

\end{enumerate}
\end{theorem}

\begin{proof}
\ref{basic:prop:hippa:a} Suppose that $x^{k+1} \neq x^{k}$ for all $k$.  Since 
$x^{k+1} \in \prox{h}{\beta_k}{p}(x^k)$, we have
\begin{equation}\label{eq1:teles}
 h (x^{k+1}) < h(x^{k+1}) + \frac{1}{p \beta_k} \Vert x^{k+1} - x^k \Vert^p \leq h(x^k),
\end{equation}
and the result follows.
\\
\ref{basic:prop:hippa:b}: It follows immediately by telescoping \eqref{eq1:teles} and using $\beta_k \ge \beta'>0$.
\\
\ref{basic:prop:hippa:c}: 
Since every strongly quasar-convex function is quasar-convex, we only need to prove the case $\gamma=0$. Indeed,
we will show that $\{x^k\}_{k\in\Nz}$ is Fej\'er monotone with respect to $\{\ov x\}$ 
(see \cite[Definition 5.1]{Bauschke17}). Fix $k\in\Nz$ and $\lambda\in[0,1]$. Since $x^{k+1} \in \prox{h}{\beta_k}{p}(x^k)$, taking
\[
x=\lambda\ov x+(1-\lambda)x^{k+1},
\]
we obtain
\begin{align}\label{eq:qc-prox-step}
h(x^{k+1})+\frac{1}{p\beta_k}\|x^{k+1}-x^k\|^p
&\leq h(\lambda\ov x+(1-\lambda)x^{k+1})
+\frac{1}{p\beta_k}\|\lambda\ov x+(1-\lambda)x^{k+1}-x^k\|^p \notag \\
&\leq \kappa\lambda h(\ov x)+(1-\kappa\lambda)h(x^{k+1})
+\frac{1}{p\beta_k}\|(x^{k+1}-x^k)+\lambda(\ov x-x^{k+1})\|^p,
\end{align}
where the second inequality follows from the $\kappa$-quasar-convexity of $h$ with respect to $\ov x$.
Rearranging \eqref{eq:qc-prox-step}, we get
\begin{equation}\label{eq:before-limit}
\beta_k\kappa\big(h(x^{k+1})-h(\ov x)\big)
\leq \frac{\|(x^{k+1}-x^k)+\lambda(\ov x-x^{k+1})\|^p-\|x^{k+1}-x^k\|^p}{p\lambda}.
\end{equation}
Since the mapping $z\mapsto \|z\|^p$ is Fr\'echet differentiable on $\mathbb{R}^n$ for every $p>1$,
letting $\lambda\downarrow 0$ in \eqref{eq:before-limit} yields
\begin{equation}\label{eq:key-fejer}
\beta_k\kappa\left(h(x^{k+1})-h(\ov x)\right)
\leq\|x^{k+1}-x^k\|^{p-2}\langle x^{k+1}-x^k,\ov x-x^{k+1}\rangle.
\end{equation}
Applying the identity \eqref{3:points} yields
\[
2\langle x^{k+1}-x^k,\ov x-x^{k+1}\rangle =\|x^k-\ov x\|^2-\|x^{k+1}-\ov x\|^2-\|x^{k+1}-x^k\|^2.
\]
Together with \eqref{eq:key-fejer} and  $h(x^{k+1})\geq h(\ov x)$, this ensures
\begin{equation}\label{eq:b-eq03}
0\leq2\beta_k\kappa\big(h(x^{k+1})-h(\ov x)\big)\leq\|x^{k+1}-x^k\|^{p-2}
\left(\|x^k-\ov x\|^2-\|x^{k+1}-\ov x\|^2-\|x^{k+1}-x^k\|^2\right),
\end{equation}
i.e.,
\[
\|x^{k+1}-\ov x\|\leq \|x^k-\ov x\|, \qquad \forall k\in\Nz,
\]
i.e., $\{x^k\}_{k\in\Nz}$ is Fej\'er monotone with respect to $\{\ov{x}\}$.
In particular, it is bounded. 
\\
\ref{basic:prop:hippa:d1}: Applying \eqref{eq:key-fejer} and the Cauchy--Schwarz inequality, it follows that
\[
0\le \beta_k\kappa\big(h(x^{k+1})-h(\ov x)\big)
\leq \|x^{k+1}-x^k\|^{p-1}\|\ov x-x^{k+1}\|.
\]
By Assertions~\ref{basic:prop:hippa:b} and \ref{basic:prop:hippa:c},
$\|x^{k+1}-x^k\|\to0$ and $\{x^k\}_{k\in\Nz}$ is bounded. Since $p>1$ and
$\beta_k\ge\beta'>0$, we obtain
\[
h(x^{k+1})\to h(\ov x)=\min_{x\in \mathbb{R}^n} h(x).
\]
\ref{basic:prop:hippa:d}: We define $y^k := x^{k+1}$. Let $\widehat{x}$ be a cluster point of the sequence $\{x^k\}_{k \in \mathbb{N}_0}$. Then there exists a subsequence $\{x^{k_j}\}_{j \in \mathbb{N}_0}$ such that $x^{k_j} \to \widehat{x}$ as $j \to +\infty$. By considering the corresponding subsequence $\{y^{k_j}\}_{j \in \mathbb{N}_0}$, Assertion~\ref{basic:prop:hippa:b} implies that $\Vert y^{k_j} - x^{k_j} \Vert  \to 0.$ Consequently, we also have $y^{k_j} \to \widehat{x}$ as $j \to +\infty$. Since the sequence $\{\beta_k\}_{k \in \mathbb{N}_0}$ is bounded, we may assume, by passing to a further subsequence if necessary, that $\beta_{k_j} \to \beta$ as $j \to +\infty$ for some $\beta > 0$.
By the definition of iterations, $y^{k_j} \in \prox{h}{\beta_{k_j}}{p}(x^{k_j})$. Applying Theorem \ref{th:nonemprox}~\ref{th:nonemprox:c}, we conclude that $\widehat{x} \in \prox{h}{\beta}{p}(\widehat{x})$, i.e., $\widehat{x}$ is a fixed point. Finally, by Theorem~\ref{th:nonemprox}~\ref{th:nonemprox:b}, we have $\widehat{x} \in \argmin{\mathbb{R}^{n}}\,h$.
\\
\ref{basic:prop:hippa:e1}: Assume that $\gamma=0$. By Assertion~\ref{basic:prop:hippa:c}, the sequence 
$\{x^k\}_{k\in\Nz}$ is bounded, i.e., it admits a cluster point. By Assertion~\ref{basic:prop:hippa:d}, every cluster point belongs to $\argmin{\mathbb{R}^n} h$. Since, in addition,
$\{x^k\}_{k\in\Nz}$ is Fej\'er monotone with respect to $\argmin{\mathbb{R}^n} h$, it follows from
\cite[Theorem 5.11]{Bauschke17} that $\{x^k\}_{k\in\Nz}$ converges to a point in
$\argmin{\mathbb{R}^n} h$.
\\
\ref{basic:prop:hippa:e2}: It follows from Assertion~\ref{basic:prop:hippa:c} and \ref{basic:prop:hippa:d} and because the set of minimizers of strongly quasar-convex functions is a singleton. 
\end{proof}

\begin{remark}
\begin{itemize}
 \item[$(i)$] Theorem~\ref{basic:prop:hippa} shows that HiPPA inherits the main structural properties of the classical proximal-point method, while accommodating the broader class of quasar-convex functions. In particular, the method produces a descent sequence with vanishing steps, and all accumulation points are global minimizers.

 \item[$(ii)$] Note that in Theorem \ref{basic:prop:hippa}, part \ref{basic:prop:hippa:e1}, function $h$ needs to be quasar-convex with respect to every $\ov x\in\argmin{\mathbb{R}^{n}}\,h$. This assumption can not be weakened, as there are functions which are quasar-convex with respect to some minimizer and that are not quasar-convex with respect to any other minimizer (see Example \ref{exam:noncq_for_all}).
 \end{itemize}
\end{remark}

\subsection{Convergence rate analysis}
We now analyze the asymptotic behavior of HiPPA in the strongly quasar-convex regime, namely when $\gamma>0$. In this setting, the minimizer is unique, and the geometry induced by the power $p$ leads to different local estimates depending on whether $p<2$, $p=2$, or $p>2$.

\begin{theorem}[{\bf Convergence rates of the sequence of iterations: $\gamma>0$}]\label{th:con:rate}
 Let $h: \R^n \to \Rinf$ be a proper, lsc, and $(\kappa,\gamma)$-strongly quasar-convex function with respect to its unique minimizer $\ov{x}\in\argmin{\R^n} h$, for $\kappa \in (0,1]$ and $\gamma>0$. Let $\{x^{k}\}_{k\in \Nz}$ be generated by Algorithm \ref{ppa:sqcx:2}, and assume that relation \eqref{param:cond:2} holds. Then the following assertions hold:
\begin{enumerate}[label=(\textbf{\alph*}), font=\normalfont\bfseries, leftmargin=0.7cm]
  \item\label{th:con:rate:a} for $p \in (1,2)$, the convergence rate is locally linear, i.e., there exist $\ov{k}_p \in \Nz$ and $\eta_p\in(0,1)$ such that
\begin{equation}\label{locally:linear:0-1}
\frac{\|x^{k+1}-\ov{x}\|}{\|x^k-\ov{x}\|} \le \eta_p,
\qquad \forall ~ k\ge \ov{k}_p.
\end{equation}


\item\label{th:con:rate:b} for $p=2$, the convergence rate is linear, i.e.,
\begin{equation}\label{conv:rate}
 \frac{\lVert x^{k+1} - \ov{x} \rVert}{\Vert x^{k} - \ov{x} \Vert} \leq \frac{1}{\sqrt{1 + \kappa \beta^{\prime} \gamma + (\frac{\kappa^{2} \beta^{\prime} \gamma}{2 - \kappa})}}, \qquad \forall ~ k \in \Nz.
\end{equation}

\item\label{th:con:rate:c} for $p> 2$, the convergence rate is superlinear, i.e.,
\begin{equation}\label{eq:bound:1}
 \frac{\lVert x^{k+1} - \ov{x} \rVert}{\Vert x^{k} - \ov{x} \Vert^{p-1}} \leq c, \qquad \forall ~ k \in \Nz,
\end{equation}
with $c:=\frac{2-\kappa}{\beta'(\kappa\gamma)}$.
\end{enumerate}
\end{theorem}

\begin{proof}
From the definition, for each $p>1$, we obtain
\begin{equation}\label{eq1:th:con:stqconvex}
 h (x^{k+1}) + \frac{1}{p \beta_{k}} \lVert x^{k+1} - x^{k} \rVert^{p}\leq h (x) + \frac{1}{p \beta_{k}} \lVert x - x^{k} \rVert^{p}, \qquad \forall ~ x \in \R^n.
\end{equation}
 For each $\lambda \in [0, 1]$, we have $x_\lambda := \lambda \ov{x} + (1-\lambda)x^{k+1} \in  \dom{h}$. Now, we prove the assertions. \\
\ref{th:con:rate:a}:
Let $p\in(1,2)$ be arbitrary. Without loss of generality, we may assume that $\ov{x}=0$. Indeed, defining $g(z):=h(\ov{x}+z)$, the function $g$ is $(\kappa,\gamma)$-strongly quasar-convex with respect to $0$ (see Proposition~\ref{Scaling}~\ref{Scaling:a2}), and the translated iterations satisfy the same recursion. Thus, it is enough to prove the result when $\ov{x}=0$.

Fix $r>0$ such that $r\le \left(\frac{\kappa_p}{4}\right)^{\frac{1}{2-p}}$, where $\kappa_p$ is introduced in \eqref{eq:def:kappa_p}. Since $x^k\to \ov{x}=0$ by Theorem~\ref{basic:prop:hippa}~\ref{basic:prop:hippa:e2} and $\|x^{k+1}-x^k\|\to 0$ by Theorem~\ref{basic:prop:hippa}~\ref{basic:prop:hippa:b}, there exists
$\ov{k}_p\in\Nz$ such that $x^k\in \mb(0, r)$ for all $ k\ge \ov{k}_p$.
Set $\sigma_p:=\frac{\kappa_p r^{p-2}}{2}>0$. Then, by Fact \ref{fact:norm:unif}~\ref{fact:norm:unif:b}, for every $k \ge \ov{k}_p$ and every $\lambda \in [0,1]$,
\begin{align*}
 h(x^{k+1})+\frac{1}{p\beta_k}\|x^{k+1}-x^k\|^p 
 & \le h(\lambda \ov{x}+(1-\lambda)x^{k+1}) + \frac{1}{p\beta_k} \|\lambda(\ov{x}-x^k) + (1-\lambda) (x^{k+1}-x^k)\|^p \\
 & \le \lambda\kappa h(\ov{x})+(1-\lambda \kappa) h(x^{k+1}) -\lambda\Bigl( 1-\frac{ \lambda}{2 - \kappa}\Bigr) \frac{\kappa\gamma}{2} \|x^{k+1}-\ov{x}\|^2 \\
 & \quad + \frac{\lambda}{p\beta_k}\|\ov{x}-x^k\|^p+\frac{1-\lambda}{p\beta_k}\|x^{k+1}-x^k\|^p - \frac{\lambda (1-\lambda) \sigma_p}{p\beta_k}\|\ov{x}-x^{k+1}\|^p.
\end{align*}
Since $\lambda \kappa \, h(\ov{x}) + (1-\lambda \kappa) h(x^{k+1})\leq h(x^{k+1})$, it follows that
\begin{align*}
 & \frac{\lambda}{p \beta_{k}} \lVert x^{k+1} - x^{k} \rVert^{p} \leq - \lambda \left( 1 - \frac{\lambda}{2 - \kappa} \right) \frac{\kappa \gamma}{2} \lVert x^{k+1} - \ov{x} \rVert^{2} + \frac{\lambda}{p \beta_{k}} \lVert \ov{x} - x^{k} \rVert^{p}
 - \frac{\lambda (1-\lambda)\sigma_p}{p \beta_{k}} \lVert \ov{x} - x^{k+1} \rVert^{p}. \notag
\end{align*} 
By dividing both sides by $\lambda$ and letting $\lambda\downarrow 0$, we get
\begin{equation}\label{eq2:th:con:stqconvexp2}
\frac{\kappa \gamma}{2}\lVert x^{k+1} - \ov{x} \rVert^{2}+\frac{\sigma_p}{p \beta_{k}}\lVert \ov{x} - x^{k+1} \rVert^{p}
   \leq  \frac{1}{p \beta_{k}} \lVert \ov{x} - x^{k} \rVert^{p} -\frac{1}{p \beta_{k}} \lVert x^{k+1} - x^{k} \rVert^{p}.
\end{equation} 
In light of \eqref{eq:ineqp}, we get
\[
\Vert x+y\Vert^p\leq \Vert x\Vert^p+p\Vert x+y\Vert^{p-2}\langle x+y, y\rangle.
\]
by setting $x=x^{k+1}-x^{k}$ and $y=\ov{x}-x^{k+1}$, we get
\[
\Vert \ov{x}-x^{k}\Vert^p\leq \Vert x^{k+1}-x^{k}\Vert^p +p\Vert \ov{x}-x^{k}\Vert^{p-1}\lVert\ov{x}-x^{k+1}\rVert.
\]
Combining this and \eqref{eq2:th:con:stqconvexp2} imply
\[
\frac{\kappa \gamma}{2}\lVert x^{k+1} - \ov{x} \rVert^{2}+\frac{\sigma_p}{p \beta_{k}}\lVert \ov{x} - x^{k+1} \rVert^{p}
   \leq  \frac{1}{\beta_{k}}\Vert \ov{x}-x^{k}\Vert^{p-1}\lVert\ov{x}-x^{k+1}\rVert.
\]
Dropping the nonnegative first term on the left, we obtain
\[
\frac{\sigma_p}{p\beta_k}\|\ov{x}-x^{k+1}\|^p
\le\frac{1}{\beta_k}\|\ov{x}-x^k\|^{p-1}\|\ov{x}-x^{k+1}\|,
\]
i.e.,
\[
\|\ov{x}-x^{k+1}\|\le\left(\frac{p}{\sigma_p}\right)^{\frac{1}{p-1}}\|\ov{x}-x^k\|.
\]
Since $r\le \left(\frac{\kappa_p}{4}\right)^{\frac{1}{2-p}}$, we have $\sigma_p=\frac{\kappa_p r^{p-2}}{2}\ge 2 > p$. Therefore, $\eta_p := \left(\frac{p}{\sigma_p} \right)^{\frac{1}{p-1}}\in(0,1)$, and
\[
\|x^{k+1}-\ov{x}\|\le \eta_p \|x^k-\ov{x}\|,
\qquad \forall ~ k\ge \ov{k}_p.
\]
This proves that the convergence is eventually locally linear. 
\\
\ref{th:con:rate:b}: See \cite[Theorem 18]{deBrito2025extending}.
\\
\ref{th:con:rate:c}:  Invoking Fact~\ref{fact:norm:unif}~\ref{fact:norm:unif:c} and \eqref{eq1:th:con:stqconvex}, for every $p>2$ and $\lambda \in [0,1]$, leads to
 \begin{align*}
  & h(x^{k+1}) + \frac{1}{p \beta_{k}} \lVert x^{k+1} - x^{k} \rVert^{p} \\
  & \leq h(\lambda \ov{x} + (1-\lambda) x^{k+1}) + \frac{1}{p \beta_{k}}
  \lVert \lambda (\ov{x} - x^{k}) + (1-\lambda) (x^{k+1} - x^{k}) \rVert^{p},
  \notag \\
  & \leq \lambda \kappa \, h(\ov{x}) + (1-\lambda \kappa) h(x^{k+1}) -
  \lambda \left( 1 - \frac{\lambda}{2 - \kappa} \right) \frac{\kappa \gamma}{2}
  \lVert x^{k+1} - \ov{x} \rVert^{2} + \frac{\lambda}{p \beta_{k}} 
  \lVert x^{k} - \ov{x} \rVert^{p} \notag \\
  & ~~~~ +  \frac{(1-\lambda)}{p \beta_{k}} \lVert x^{k+1}-x^{k} 
  \rVert^{p} - \frac{\lambda (1-\lambda)\hat{\sigma}_p}{p \beta_{k}} 
  \lVert x^{k+1} - \ov{x} \rVert^{p}, \notag
 \end{align*} 
 for $\hat{\sigma}_p := \left(\frac{1}{2}\right)^\frac{3p-2}{2}$, i.e.,
 \begin{align*}
  & \lambda \left( 1 - \frac{\lambda}{2 - \kappa} \right) \frac{\kappa \gamma
  }{2} \lVert x^{k+1} - \ov{x} \rVert^{2} + \frac{\lambda (1-\lambda)
  \hat{\sigma}_p}{p \beta_{k}} \lVert x^{k+1} - \ov{x} \rVert^{p} \\ 
  & \leq \lambda \kappa  (h(\ov{x}) - h(x^{k+1}) ) + \frac{\lambda}{p
  \beta_{k}} \lVert x^{k} - \ov{x} \rVert^{p} - \frac{\lambda}{p \beta_{k}} 
  \lVert x^{k+1} - x^{k} \rVert^{p}. 
 \end{align*}
By dividing both sides by $\lambda$ and letting $\lambda\downarrow 0$, we get
 \begin{align}
  \frac{\kappa \gamma}{2} \lVert x^{k+1} - \ov{x} \rVert^{2} + \frac{\hat{\sigma}_p}{p \beta_{k}} \lVert x^{k+1} - \ov{x} \rVert^{p} \leq & \, \kappa (h(\ov{x}) - h(x^{k+1}) ) + \frac{1}{p \beta_{k}} \lVert x^{k} - \ov{x} \rVert^{p} \notag \\
 & - \frac{1}{p \beta_{k}} \lVert x^{k+1} - x^{k} \rVert^{p}. \label{e:02}
 \end{align} 
 Combining this with \eqref{eq:ineqp} results in
\begin{align*}
 \Vert x^{k} - \ov{x} \Vert^{p} &=\Vert (x^{k} - x^{k+1}) - (\ov{x} - x^{k+1})
 \Vert^{p} \\ 
 & \leq \Vert x^{k+1} - x^{k} \Vert^{p} - p \Vert x^{k} - \ov{x} \Vert^{p-2}
 \langle x^{k} - \ov{x}, \ov{x} - x^{k+1} \rangle \\ 
 & =  \Vert x^{k+1} - x^{k} \Vert^p + p \Vert x^{k} - \ov{x} \Vert^{p-2} 
 \langle x^{k} - \ov{x}, x^{k+1} - \ov{x} \rangle \\ 
 & \leq \Vert x^{k+1}-x^{k} \Vert^{p} + p\Vert x^{k} - \ov{x} \Vert^{p-1} 
 \lVert x^{k+1} - \ov{x} \rVert.
\end{align*} 
It follows from this and \eqref{e:02} that
\[
 \frac{\kappa \gamma}{2} \lVert x^{k+1} - \ov{x} \rVert^{2} +
 \frac{\hat{\sigma}_p}{p \beta_{k}} \lVert x^{k+1} - \ov{x} \rVert^{p} 
 \leq \kappa (h(\ov{x}) - h(x^{k+1}) ) + \frac{1}{\beta_{k}} \Vert
 x^{k} - \ov{x} \Vert^{p-1} \lVert x^{k+1} - \ov{x} \rVert. 
\]
Using \eqref{qwc:sconvex} with $y=x^{k+1}$, we get $h(\ov{x}) - h(x^{k+1}) \leq - \frac{\kappa \gamma}{2(2 - \kappa)} \lVert x^{k+1} - \ov{x} \rVert^{2}$, i.e.,
\[
\left( \frac{\kappa^{2} \gamma}{2(2 - \kappa)} + \frac{\kappa \gamma}{2}
 \right) \lVert x^{k+1} - \ov{x} \rVert^{2} +
 \frac{\hat{\sigma}_p}{p \beta_{k}} \lVert x^{k+1} - \ov{x} \rVert^{p} 
 \leq \frac{1}{\beta_{k}} \Vert x^{k} - \ov{x} \Vert^{p-1} \lVert x^{k+1} - \ov{x} \rVert,
\]
leading to
\[
\left( \frac{\kappa^{2} \gamma}{2(2 - \kappa)} + \frac{\kappa \gamma}{2} \right) \lVert x^{k+1} - \ov{x} \rVert +
 \frac{\hat{\sigma}_p}{p \beta_{k}} \lVert x^{k+1} - \ov{x} \rVert^{p-1} 
 \leq \frac{1}{\beta_{k}} \Vert x^{k} - \ov{x} \Vert^{p-1},
\]
i.e.,
\[
\left( \frac{\kappa^{2} \gamma}{2(2 - \kappa)} + \frac{\kappa \gamma}{2} \right) \lVert x^{k+1} - \ov{x} \rVert 
 \leq \frac{1}{\beta_{k}} \Vert x^{k} - \ov{x} \Vert^{p-1}.
\]
This consequently ensures
\[
\frac{\lVert x^{k+1} - \ov{x} \rVert}{\Vert x^{k} - \ov{x} \Vert^{p-1}} 
 \leq \frac{2-\kappa}{\beta'(\kappa\gamma)} .
\]
Since $p>2$ and $x^{k} \rightarrow \ov{x}$ by Theorem~\ref{basic:prop:hippa}~\ref{basic:prop:hippa:e2}, it follows that $\{x^{k}\}_{k\in \Nz}$ converges superlinearly to the unique solution $\ov{x}$, which completes
the proof.
\end{proof}

Note that the convergence rate provided in Theorem~\ref{th:con:rate}~\ref{th:con:rate:c} is superlinear, which is
\[
\frac{\|x^{k+1}-\ov{x}\|}{\|x^k-\ov{x}\|}\to 0 \quad \text{as } k\to\infty.
\]
Moreover, if
\[
\|x^0-\ov{x}\|\le c^{-\frac{1}{p-2}},
\]
with $c:=\frac{2-\kappa}{\beta'(\kappa\gamma)}$, one obtains the monotonic decay
\[
\|x^{k+1}-\ov{x}\|\le \|x^k-\ov{x}\|,
\qquad \forall k\in\Nz.
\]

\begin{theorem}[{\bf Convergence rates for the sequence of function values: $\gamma>0$}]\label{th:convrate:fun_val}
 Let $h: \R^n \to \Rinf$ be a proper, lsc and $(\kappa,\gamma)$-strongly quasar-convex function 
 with respect to $\{\ov x\}=\argmin{\mathbb{R}^n} h$, where $\kappa \in (0,1]$ and $\gamma>0$.
 Let $\{x^{k}\}_{k\in \Nz}$ be the se\-quen\-ce generated by Algorithm \ref{ppa:sqcx:2} and suppose that relation \eqref{param:cond:2} holds. Then the following assertions hold:
\begin{enumerate}[label=(\textbf{\alph*}), font=\normalfont\bfseries, leftmargin=0.7cm]
\item \label{th:convrate:fun_val:p_1_2} 
If $p\in (1,2)$, then there exist $\ov{k}_p \in \Nz$ and $\eta_p\in(0,1)$ such that
for every $k\ge \ov{k}_p$,
\begin{equation}\label{eq:th:convrate:fun_val:p_1_2}
h(x^{k+1})-h(\ov x)\le \frac{1}{p\beta'}\,\eta_p^{\,p(k-\ov{k}_p)}\|x^0-\ov x\|^p.
\end{equation}


\item \label{th:convrate:fun_val:a} If $p=2$, then for every $k\in\Nz$,
\begin{equation}\label{eq:th:convrate:fun_val:a}
h(x^{k+1})-h(\ov x) \leq
\frac{1}{2\beta'}\,r^{2k}\,\|x^0-\ov x\|^2,
\end{equation}
for $r:=\frac{1}{\sqrt{1+\kappa\beta'\gamma+\frac{\kappa^2\beta'\gamma}{2-\kappa}}}\in(0,1)$.

\item\label{th:convrate:fun_val:b}  If $p>2$, define $c:=\frac{2-\kappa}{\beta'\kappa\gamma}$.
Then, for every $k\in\Nz$,
\begin{equation}\label{eq:th:convrate:fun_val:b}
h(x^{k+1})-h(\ov x)
\le
\frac{1}{p\beta'}\,c^{-p/(p-2)}
\Big(c^{1/(p-2)}\|x^0-\ov x\|\Big)^{p(p-1)^k}.
\end{equation}
\end{enumerate}
 \end{theorem}
 \begin{proof}
For every $k\in\Nz$, the inclusion $x^{k+1} \in \prox{h}{\beta_k}{p}(x^k)$ and $x=\ov x$ imply
\[
h(x^{k+1})+\frac{1}{p\beta_k}\|x^{k+1}-x^k\|^p
\le
h(\ov x)+\frac{1}{p\beta_k}\|x^k-\ov x\|^p,
\]
i.e.,
\begin{equation}\label{eq:basic-fval}
h(x^{k+1})-h(\ov x)\le \frac{1}{p\beta_k}\|x^k-\ov x\|^p
\le \frac{1}{p\beta'}\|x^k-\ov x\|^p.
\end{equation}
\ref{th:convrate:fun_val:p_1_2}
In light of Theorem~\ref{th:con:rate}~\ref{th:con:rate:a}, there exist $\ov{k}_p \in \Nz$ and $\eta_p\in(0,1)$ such that
\eqref{locally:linear:0-1} holds. Iterating this inequality yields
\[
\|x^k-\ov x\|\le \eta_p^{\,k-\ov{k}_p}\|x^{\ov{k}_p}-\ov x\|,\qquad \forall k\ge \ov{k}_p.
\]
Since $\{\|x^k-\ov x\|\}_{k\in\Nz}$ is nonincreasing by Theorem~\ref{basic:prop:hippa}~\ref{basic:prop:hippa:c}, we have
\[
\|x^{\ov{k}_p}-\ov x\|\le \|x^0-\ov x\|.
\]
Substituting into \eqref{eq:basic-fval}, we obtain \eqref{eq:th:convrate:fun_val:p_1_2}.
\\
\ref{th:convrate:fun_val:a}:  From Theorem~\ref{th:con:rate}~\ref{th:con:rate:b},
 for every $k\in\mathbb{N}_0$, \eqref{conv:rate} holds.
Substituting this into \eqref{eq:basic-fval}, we come to \eqref{eq:th:convrate:fun_val:a}. 
\\
\ref{th:convrate:fun_val:b}: From Theorem~\ref{th:con:rate}~\ref{th:con:rate:c},
\eqref{eq:bound:1} holds, which implies
\[
\|x^k-\ov x\|
\le
c^{-1/(p-2)}
\Big(c^{1/(p-2)}\|x^0-\ov x\|\Big)^{(p-1)^k},
\qquad \forall k\in\Nz.
\]
Substituting this into the inequality \eqref{eq:basic-fval} ensures \eqref{eq:th:convrate:fun_val:b}.
\end{proof}
In the next result, we investigate the convergence rates of the function value sequence generated by HiPPA for $\kappa$-quasar-convex functions; i.e., $\gamma=0$.

\begin{theorem}[{\bf Convergence rates for the sequence of function values: $\gamma=0$}]\label{th:conv_rate:fun}
Let $h: \mathbb{R}^n \to \overline{\mathbb{R}}$ be a proper, lsc and $\kappa$-quasar-convex function (i.e., $\gamma = 0$) with respect to every element of $\argmin{\mathbb{R}^{n}}\,h$. Suppose that relation \eqref{param:cond:2} holds. 
Let $\{x^{k}\}_{k\in \Nz}$ be generated by Algorithm \ref{ppa:sqcx:2}, and assume that $x^k\to\ov{x}\in\argmin{\mathbb{R}^{n}}\,h$. Then the following assertions hold:
\begin{enumerate}[label=(\textbf{\alph*}), font=\normalfont\bfseries, leftmargin=0.7cm]
 \item \label{th:conv_rate:fun:p_in_(1,2)}
  If $p\in(1,2)$, then for every $k\in\mathbb{N}$,
\begin{equation}\label{rate:1<p<2}
 h(x^{k})-h(\ov x) \leq \frac{ \|x^0-\ov{x}\|^{p} }{2^{p-1}\beta^{\prime}\kappa} \frac{1}{k^{p-1}}.
\end{equation}

\item \label{th:conv_rate:fun:p is 2}
If $p=2$, then for every $k\in\mathbb{N}$,
\begin{equation}\label{rate:p=2}
 h(x^{k})- h(\ov x) \le \frac{ \|x^0-\ov{x} \|^{2}}{2\kappa \beta'} \frac{1}{k}.
\end{equation}

\item \label{th:conv_rate:fun:p>2}
If $p>2$, then for every $k\in\mathbb{N}$,
\begin{equation}\label{rate:p>2}
 h(x^{k})- h(\ov x) \leq \frac{1}{\beta' \kappa p} \left(\frac{p-2}{p} \right)^{\frac{p-2}{2}} \frac{ \|x^{0} - \ov{x}\|^{p}}{k^{\frac{p}{2}}},
\end{equation}
\end{enumerate}
\end{theorem}
\begin{proof}
We recall first from the proof of Theorem~\ref{basic:prop:hippa}~\ref{basic:prop:hippa:c} (see \eqref{eq:b-eq03}) that
\begin{align}\label{eq03}
2\beta^{\prime} \kappa (h(x^{k+1})-h(\ov{x})) 
\leq \lVert x^{k+1} - x^{k} \Vert^{p-2} \left( \lVert x^k - \ov{x} \rVert^2 - \lVert x^{k+1} - \ov{x} \rVert^2 - \lVert x^{k+1} -x^{k} \rVert^2 \right). 
\end{align}
Let us set
\[
\Delta_k:=\|x^k-\ov{x}\|^2-\|x^{k+1}-\ov{x}\|^2.
\]
Note that from Theorem~\ref{basic:prop:hippa}~\ref{basic:prop:hippa:c}, $\Delta_k\geq 0$ for all $k\in \Nz$.
\\
\ref{th:conv_rate:fun:p_in_(1,2)}: Since $p\in (1,2)$, we have $p-2<0$.
By the triangle inequality,
$$
\|x^{k+1}-x^k\|
\ge
\big|\|x^k-\ov x\|-\|x^{k+1}-\ov x\|\big|.
$$
Since the function $t\mapsto t^{p-2}$ is decreasing on $(0,+\infty)$, it follows from
\eqref{eq:b-eq03} and \eqref{eq03} that
\begin{equation}\label{eq07}
2\beta' \kappa \big(h(x^{k+1})-h(\ov x)\big)
\le
\big(\|x^k-\ov x\|-\|x^{k+1}-\ov x\|\big)^{p-2}\,\Delta_k.
\end{equation}
Applying the identity
$$
\Delta_k
=
\big(\|x^k-\ov x\|-\|x^{k+1}-\ov x\|\big)
\big(\|x^k-\ov x\|+\|x^{k+1}-\ov x\|\big),
$$
it can be concluded that
\begin{equation}\label{eq08}
2\beta^{\prime} \kappa \big(h(x^{k+1})-h(\ov x)\big)
\le
\big(\|x^k-\ov x\|+\|x^{k+1}-\ov x\|\big)^{2-p}
\Delta_k^{\,p-1}.
\end{equation}
Since $\{\|x^k-\ov x\|\}$ is nonincreasing, it follows that
$$
\|x^k-\ov x\|+\|x^{k+1}-\ov x\|
\leq
2\|x^0-\ov x\|.
$$
It follows from $2-p>0$ and the inequality \eqref{eq08} that
\begin{equation}\label{ineq_key}
h(x^{k+1})-h(\ov x)
\le
\frac{(2\|x^0-\ov x\|)^{2-p}}{2\beta'\kappa}\,
\Delta_k^{\,p-1}.
\end{equation}
Together with \eqref{ineq_key} and the fact that $\{h(x^k)\}_{k\in \mathbb{N}_0}$ is nonincreasing, this ensures
$$
k\big(h(x^k)-h(\ov x)\big)\le\frac{(2\|x^0-\ov x\|)^{2-p}}{2\beta'\kappa}
\sum_{j=0}^{k-1}\Delta_j^{\,p-1}.
$$
Since the function $t\mapsto t^{p-1}$ is concave on $[0,+\infty)$, Jensen's inequality ensures
$$
\dfrac{1}{k}\sum_{j=0}^{k-1}\Delta_j^{\,p-1}
\le\left(\dfrac{1}{k}\sum_{j=0}^{k-1}\Delta_j\right)^{p-1}.
$$
Using the telescoping sum,
$$
\sum_{j=0}^{k-1}\Delta_j=\|x^0-\ov x\|^2-\|x^k-\ov x\|^2\le\|x^0-\ov x\|^2,
$$
it can be concluded that
$$
h(x^k)-h(\ov x)\le\frac{(2\|x^0-\ov x\|)^{2-p}}{2\beta'\kappa}\,\|x^0-\ov x\|^{2(p-1)}
\frac{1}{k^{p-1}}.
$$
The inequalities \eqref{rate:1<p<2} and \eqref{compl:1<p<2} follow directly from the last inequality.
\\
\ref{th:conv_rate:fun:p is 2}: Follows from \cite[Proposition 25]{deBrito2025extending}.
\\
\ref{th:conv_rate:fun:p>2}: Let $d_k=x^{k+1}-x^k$. 
From \eqref{eq03}, we obtain
\begin{equation}\label{eq05}
2\beta^{\prime}\kappa\big(h(x^{k+1})-h(\ov{x})\big)
\leq
\|d_k\|^{p-2}\Delta_k-\|d_k\|^{p}.
\end{equation}
We note that the function $\psi_k(t):= \Delta_k t^{p-2}-t^{p}$ on $t\ge 0$. 
For $p>2$, a simple calculation shows that 
$t_k=\left(\dfrac{p-2}{p}\Delta_k\right)^{\frac{1}{2}}$ 
is the global maximizer of $\psi_k(t)$, and
$$
\maxt{t\ge 0}\psi_k(t)
=
\frac{2}{p}\Big(\frac{p-2}{p}\Big)^{\frac{p-2}{2}} \Delta_k^{p/2}.
$$
Applying this bound in \eqref{eq05} results in
$$
2\beta'\kappa\big(h(x^{k+1})-h(\ov{x})\big)
\leq 
\psi_k(\lVert d_k \rVert )
\leq 
\frac{2}{p}\Big(\frac{p-2}{p}\Big)^{\frac{p-2}{2}} \Delta_k^{p/2},
$$
i.e., 
$$
\big(h(x^{k+1})-h(\ov{x})\big)^{\frac{2}{p}}
\leq   
\left(\frac{1}{\beta'\kappa p}\right)^{\frac{2}{p}}
\Big(\frac{p-2}{p}\Big)^{\frac{p-2}{p}} 
\Delta_k.
$$
Similar to Assertion~\ref{th:conv:fun:p_in_(1,2)}, summing up and using telescoping property yield \eqref{rate:p>2}.
\end{proof}

\subsection{Complexity analysis}
In this section, we study the non-asymptotic behavior of the sequence generated by HiPPA by deriving iteration-complexity bounds for quasar-convex functions in both regimes $\gamma>0$ and $\gamma=0$.

Our first result establishes iteration complexity bounds in terms of the distance to the minimizer, distinguishing the regimes $p\in(1,2)$, $p=2$, and $p>2$.
\begin{theorem}[{\bf Complexity for the sequence of iterations: $\gamma>0$}]\label{th:complexity}
 Let $h: \R^n \to \Rinf$ be a proper, lsc and 
 $(\kappa,\gamma)$-strongly quasar-convex function 
 with respect to $\{\ov x\}=\argmin{\mathbb{R}^n} h$, where $\kappa \in (0,1]$ and $\gamma>0$.
  Let $\{x^{k}\}_{k\in \Nz}$ be the sequence generated by Algorithm \ref{ppa:sqcx:2} and suppose that relation \eqref{param:cond:2} holds. Then the following assertions hold:
 \begin{enumerate}[label=(\textbf{\alph*}), font=\normalfont\bfseries, leftmargin=0.7cm]
 \item\label{th:complexity:p_1_2}
 
If $p\in (1,2)$ and $x^0\in\mb(\ov{x},r)$ with $r\le \left(\frac{\kappa_p}{4}\right)^{\frac{1}{2-p}}$, where $\kappa_p$ is introduced in \eqref{eq:def:kappa_p}, then for a given $\varepsilon\in (0,1)$, the number of iterations to reach
$\Vert \ov{x} - x^k\Vert< \varepsilon$, denoted by $\mathcal{N}(\varepsilon)$, satisfies
\begin{equation}
\label{eq:complexity-bound}
\mathcal{N}(\varepsilon)\leq \left\lceil 1+\frac{\log(\varepsilon^{-1}) + \log\left(\Vert x^{0} - \ov{x}\Vert\right)}{\log(1/\eta_p)}\right\rceil,
\end{equation}
for $\eta_p= (\frac{2p}{\kappa_pr^{p-2}})^{\frac{1}{p-1}}<1$.
In particular, 
$\mathcal{N}(\varepsilon)=\mathcal{O}(\log(\varepsilon^{-1}))$.

\item\label{th:complexity:a} 
If $p=2$, then, for a given $\varepsilon\in (0,1)$, the number of iterations to reach 
$\Vert \ov{x} - x^k\Vert< \varepsilon$, denoted by $\mathcal{N}(\varepsilon)$, satisfies
\begin{equation}
\label{eq:complexity-bound:b}
\mathcal{N}(\varepsilon)\le \left\lceil 1+\frac{\log(\varepsilon^{-1})+\log(\Vert x^0-\ov{x}\Vert)}{\log(1/r)}\right\rceil,
\end{equation}
for $r:=\frac{1}{\sqrt{1+\kappa\beta'\gamma+\frac{\kappa^2\beta'\gamma}{2-\kappa}}}\in(0,1)$. In particular, 
$\mathcal{N}(\varepsilon)=\mathcal{O}(\log(\varepsilon^{-1}))$.

\item\label{th:complexity:b} If $p>2$, define $c:=\frac{2-\kappa}{\beta'\kappa\gamma}$ and assume that
\begin{equation}\label{eq:init-radius}  
\|x^0-\ov x\|\leq c^{-1/(p-2)}.
\end{equation}
Then, for every $k\in\mathbb{N}_0$,
\begin{equation}\label{eq:superlinear-rate-global}
\|x^k-\ov x\|\leq c^{-1/(p-2)}
\Big(c^{1/(p-2)}\|x^0-\ov x\|\Big)^{(p-1)^k}.
\end{equation}
Consequently, for a given $\varepsilon\in(0,c^{-1/(p-2)})$,  the number of iterations to reach $\Vert \ov{x} - x^k\Vert< \varepsilon$, denoted by $\mathcal{N}(\varepsilon)$, satisfies
\begin{equation}\label{eq:complexity-p>2}
\mathcal{N}(\varepsilon)\leq
\left\lceil 1+
\frac{
\log\!\left(
\frac{\log\!\big((c^{1/(p-2)}\varepsilon)^{-1}\big)}
{\log\!\big((c^{1/(p-2)}\|x^0-\ov x\|)^{-1}\big)}
\right)
}{
\log(p-1)
}
\right\rceil.
\end{equation}
In particular, 
$\mathcal{N}(\varepsilon)=\mathcal{O}\big(\log\log(\varepsilon^{-1})\big)$.
 \end{enumerate}
\end{theorem}

\begin{proof}
\ref{th:complexity:p_1_2} If $\Vert \ov{x} - x^0\Vert< \varepsilon$, the stopping criterion is already satisfied and the algorithm terminates at the initial iterate. Let us assume that
 $\Vert \ov{x} - x^0\Vert\geq \varepsilon$, which ensures that
 \begin{equation}\label{eq:th:complexity:non_noneq}
 \log(\varepsilon^{-1}) + \log\left(\Vert x^{0} - \ov{x}\Vert\right)\geq 0.
 \end{equation}
 Since the sequence starts in $\mb(\ov{x},r)$, from Theorem~\ref{basic:prop:hippa}\ref{basic:prop:hippa:c}, 
 $x^k\in \mb(\ov{x},r)$ for all $k\in \Nz$, i.e., for all $k\in\mathbb{N}$, the proof of Theorem~\ref{th:con:rate}~\ref{th:con:rate:a} yields
\[
\|x^{k+1}-\ov{x}\|\le \eta_p \|x^k-\ov{x}\|,
\]
for $\eta_p= (\frac{2p}{\kappa_pr^{p-2}})^{\frac{1}{p-1}}<1$.
Using the induction, this consequently implies
\begin{equation}\label{eq1:th:complexity:funval}
\Vert x^{k} - \ov{x}\Vert \leq \eta_p^{k}\Vert x^{0} - \ov{x}\Vert, \qquad \forall\, k\geq 0.
\end{equation}
Let $\mathcal{K}\in \Nz$ be the smallest integer such that
$ \eta_p^{\mathcal{K}}\Vert x^{0} - \ov{x}\Vert< \varepsilon$, which in turn implies $\Vert x^{\mathcal{K}} - \ov{x}\Vert <\varepsilon$, i.e., $\eta_p^{\mathcal{K}-1}\Vert x^{0} - \ov{x}\Vert\geq \varepsilon$. Taking logarithms,
\[
(\mathcal{K}-1)\log(\eta_p)\geq \log(\varepsilon) - \log\left(\Vert x^{0} - \ov{x}\Vert\right),
\]
i.e.,
\[
\mathcal{K}\leq 1+\frac{\log(\varepsilon^{-1}) + \log\left(\Vert x^{0} - \ov{x}\Vert\right)}{\log(1/\eta_p)}.
\]
Thus, the bound~\eqref{eq:complexity-bound} follows. 
\\
 \ref{th:complexity:a}: 
 Similar to the proof of Assertion~\ref{th:complexity:p_1_2}, if the algorithm does not terminate at the initial iterate, \eqref{eq:th:complexity:non_noneq} holds.  
 From Theorem~\ref{th:con:rate}~\ref{th:con:rate:b},
 $\lVert x^{k+1} - \ov{x} \rVert\leq r\Vert x^{k} - \ov{x} \Vert$.
 Consequently, for every $k\in\mathbb{N}_0$,
\begin{equation}\label{eq:linear-rate-global}
\|x^k-\ov x\|\leq r^k\|x^0-\ov x\|.
\end{equation}
Proceeding as above gives~\eqref{eq:complexity-bound:b}.
\\
\ref{th:complexity:b}: In light of Theorem~\ref{th:con:rate}~\ref{th:con:rate:c}, for every $k\in\mathbb{N}_0$, it follows that
\begin{equation}\label{eq:superlinear-rec}
\|x^{k+1}-\ov x\|\leq c\,\|x^k-\ov x\|^{p-1}.
\end{equation}
Let $e_k:=\|x^k-\ov x\|$ and define $u_k:=c^{1/(p-2)}e_k$. Then, the inequality \eqref{eq:superlinear-rec} becomes
\[
u_{k+1}\leq u_k^{\,p-1},
\qquad \forall k\in\mathbb{N}_0.
\]
By \eqref{eq:init-radius}, we have $u_0\leq 1$, i.e., by the induction, it follows that
\[
u_k\leq u_0^{(p-1)^k},
\qquad \forall k\in\mathbb{N}_0.
\]
Returning to $e_k$, we obtain \eqref{eq:superlinear-rate-global}.
Now, let $\varepsilon\in(0,c^{-1/(p-2)})$. In order to guarantee $e_k<\varepsilon$, it is enough that
\[
c^{-1/(p-2)}
\Big(c^{1/(p-2)}e_0\Big)^{(p-1)^k}<\varepsilon.
\]
Let $\mathcal{K}\in \Nz$ be the smallest integer such that the above inequality holds, i.e.,
\[
c^{-1/(p-2)}
\Big(c^{1/(p-2)}e_0\Big)^{(p-1)^{\mathcal{K}-1}}\geq \varepsilon.
\]
Since $0<c^{1/(p-2)}e_0\leq 1$ and $0<c^{1/(p-2)}\varepsilon<1$, this yields
\[
(p-1)^{\mathcal{K}-1}\leq \frac{\log\left(\left(c^{1/(p-2)}\varepsilon\right)^{-1}\right)}{\log\left(\left(c^{1/(p-2)}e_0\right)^{-1}\right)}.
\]
Since $p-1>1$, taking logarithms again gives~\eqref{eq:complexity-p>2}.
\end{proof}

\begin{theorem}[{\bf Complexity for the sequence of function values: $\gamma>0$}]
\label{thm:function-complexity-strong}
 Let $h: \R^n \to \Rinf$ be a proper, lsc and 
 $(\kappa,\gamma)$-strongly quasar-convex function 
 with respect to $\{\ov x\}=\argmin{\mathbb{R}^n} h$, where $\kappa \in (0,1]$ and $\gamma>0$.
 Let $\{x^{k}\}_{k\in \Nz}$ be the se\-quen\-ce generated by Algorithm \ref{ppa:sqcx:2} and suppose that relation \eqref{param:cond:2} holds. Then the following assertions hold:
\begin{enumerate}[label=(\textbf{\alph*}), font=\normalfont\bfseries, leftmargin=0.7cm]
\item \label{thm:function-complexity-strong:p_1_2} 
If $p\in (1,2)$ and $x^0\in\mb(\ov{x},r)$ with $r\le \left(\frac{\kappa_p}{4}\right)^{\frac{1}{2-p}}$, where $\kappa_p$ is introduced in \eqref{eq:def:kappa_p}, then for any $\varepsilon\in\Big(0,\frac{1}{p\beta'}\|x^0-\ov x\|^p\Big)$, the number of
iterations required to guarantee $h(x^k)-h(\ov x)<\varepsilon$, denoted by $\mathcal{N}_h(\varepsilon)$, satisfies
\begin{equation}\label{eq:fcomplexity-p1_2}
\mathcal{N}_h(\varepsilon)\leq \left\lceil 1+\frac{\log\!\big(\|x^0-\ov x\|^p/(p\beta'\varepsilon)\big)}{p\log(1/\eta_p)}
\right\rceil.
\end{equation}
In particular, $\mathcal{N}_h(\varepsilon)=\mathcal{O}\!\left(\log(\varepsilon^{-1})\right)$.

\item \label{thm:function-complexity-strong:a} If $p=2$, then for any 
$\varepsilon\in\Big(0,\frac{1}{2\beta'}\|x^0-\ov x\|^2\Big)$, the number of
iterations required to guarantee $h(x^k)-h(\ov x)<\varepsilon$, denoted by $\mathcal{N}_h(\varepsilon)$, satisfies
\begin{equation}\label{eq:fcomplexity-p2}
\mathcal{N}_h(\varepsilon)\leq \left\lceil 1+\frac{\log\!\big(\|x^0-\ov x\|^2/(2\beta'\varepsilon)\big)}{2\log(1/r)}
\right\rceil,
\end{equation}
with $r:=\frac{1}{\sqrt{1+\kappa\beta'\gamma+\frac{\kappa^2\beta'\gamma}{2-\kappa}}}\in(0,1)$.
In particular, $\mathcal{N}_h(\varepsilon)=\mathcal{O}\!\left(\log(\varepsilon^{-1}))\right)$.

\item\label{thm:function-complexity-strong:b}  If $p>2$, define $c:=\frac{2-\kappa}{\beta'\kappa\gamma}$, and assume, in addition, that \eqref{eq:init-radius} holds.
Then, for every $k\in\Nz$,
\begin{equation}
h(x^{k+1})-h(\ov x)
\le
\frac{1}{p\beta'}\,c^{-p/(p-2)}
\Big(c^{1/(p-2)}\|x^0-\ov x\|\Big)^{p(p-1)^k}.
\end{equation}
Consequently, for any $\varepsilon\in\Big(0,\frac{1}{p\beta'}\,c^{-p/(p-2)}\Big)$,
the number of iterations required to guarantee $h(x^k)-h(\ov x)<\varepsilon$, denoted by $\mathcal{N}_h(\varepsilon)$, satisfies
\begin{equation}\label{eq:fcomplexity-pgt2}
\mathcal{N}_h(\varepsilon)\leq \left\lceil 1+ \frac{\log\!\left(\frac{\log\!\big((p\beta'c^{p/(p-2)}\varepsilon)^{-1}\big)}{p\log\!\big((c^{1/(p-2)}\|x^0-\ov x\|)^{-1}\big)}
\right)}{\log(p-1)}\right\rceil.
\end{equation}
In particular, $\mathcal{N}_h(\varepsilon)=\mathcal{O}\big(\log\log(\varepsilon^{-1})\big)$.
\end{enumerate}
 \end{theorem}
 \begin{proof}
\ref{thm:function-complexity-strong:p_1_2}
 Since the sequence starts in $\mb(\ov{x},r)$, from Theorem~\ref{basic:prop:hippa}\ref{basic:prop:hippa:c}, 
 $x^k\in \mb(\ov{x},r)$ for all $k\in \Nz$. Hence, from the proof of Theorem~\ref{th:con:rate}~\ref{th:con:rate:a}, for all $k\in\Nz$, $\|x^{k}-\ov{x}\|\le \eta_p^k \|x^0-\ov{x}\|$,
for $\eta_p= (\frac{2p}{\kappa_pr^{p-2}})^{\frac{1}{p-1}}<1$. 
As a result, \eqref{eq:th:convrate:fun_val:p_1_2} holds for all $k\in\Nz$.
If $\frac{1}{p\beta'}\,\eta_p^{pk}\|x^0-\ov x\|^p<\varepsilon$,
then $h(x^{k+1})-h(\ov x)<\varepsilon$. Solving for $k$ similar to the proof of Theorem~\ref{th:complexity}~\ref{th:complexity:p_1_2}, gives \eqref{eq:fcomplexity-p1_2}.
\\
\ref{thm:function-complexity-strong:a}:  From \eqref{eq:th:convrate:fun_val:a}, if
$\frac{1}{2\beta'}\,r^{2k}\|x^0-\ov x\|^2<\varepsilon$,
then $h(x^{k+1})-h(\ov x)<\varepsilon$. Solving for $k$ similar to the proof of Theorem~\ref{th:complexity}~\ref{th:complexity:a}, gives \eqref{eq:fcomplexity-p2}.
\\
\ref{thm:function-complexity-strong:b}: From \eqref{eq:th:convrate:fun_val:b}
and with a similar approach to the proof of Theorem~\ref{th:complexity}~\ref{th:complexity:b}, gives \eqref{eq:fcomplexity-pgt2}.
\end{proof}

The convergence of HiPPA for quasar-convex functions ($\gamma=0$) and every $p>1$ is established in Assertion~\ref{basic:prop:hippa:e1} of Theorem~\ref{basic:prop:hippa}. We next show the complexity counterparts.



\begin{theorem}[{\bf Complexity for the sequence of function values: $\gamma=0$}]\label{th:conv:fun}
Let $h: \mathbb{R}^n \to \overline{\mathbb{R}}$ be a proper, lsc and $\kappa$-quasar-convex function (i.e., $\gamma = 0$) with respect to every element of $\argmin{\mathbb{R}^{n}}\,h$. Suppose that relation \eqref{param:cond:2} holds. 
Let $\{x^{k}\}_{k\in \Nz}$ be generated by Algorithm \ref{ppa:sqcx:2}, and assume that $x^k\to\ov{x}\in\argmin{\mathbb{R}^{n}}\,h$. Then the following assertions hold:
\begin{enumerate}[label=(\textbf{\alph*}), font=\normalfont\bfseries, leftmargin=0.7cm]
 \item \label{th:conv:fun:p_in_(1,2)}
If $p\in(1,2)$, then the number of iterations required to obtain an iterate $x^k$ satisfying
$
h(x^{k})- h(\ov x)\le\varepsilon
$, denoted by $\mathcal{N}(\varepsilon)$, satisfies
\begin{equation}\label{compl:1<p<2}
\mathcal{N}(\varepsilon)\leq
\left\lceil 1+
\left( \frac{\|x_{0}-\ov{x}\| ^p}{2^{p-1}\beta^{\prime} \kappa} \right)^{\frac{1}{p-1}} \dfrac{1}{\varepsilon ^{\frac{1}{p-1}}}
\right\rceil.
\end{equation}
In particular, 
$\mathcal{N}(\varepsilon)=\mathcal{O}\big(\varepsilon^{-\frac{1}{p-1}}\big)$.

\item \label{th:conv:fun:p is 2}
If $p=2$, then the number of iterations required to obtain an iterate $x^k$ satisfying $ h(x^{k})-h(\ov x)\le\varepsilon$, denoted by $\mathcal{N}(\varepsilon)$, satisfies 
\begin{equation}\label{compl:p=2}
\mathcal{N}(\varepsilon)\leq
\left\lceil 1+
\frac{\|x^0-\ov{x}\|^{2}}{2 \kappa \beta^{\prime}\varepsilon}
\right\rceil.
\end{equation}
In particular, 
$\mathcal{N}(\varepsilon)=\mathcal{O}(\varepsilon^{-1})$.

\item \label{th:conv:fun:p>2}
If $p>2$, then the number of iterations required to obtain $x^k$ such that $h(x^k)-h(\ov x)\le \varepsilon$, denoted by $\mathcal{N}(\varepsilon)$, satisfies 
\begin{equation}\label{eq06}
\mathcal{N}(\varepsilon)\leq
\left\lceil 1+
\left( \frac{1}{\beta'\kappa p} \right)^{\frac{2}{p}} \left(\frac{p-2}{p} \right)^{\frac{p-2}{p}} \frac{\|x^{0} - \ov{x}\|^{2}}{\varepsilon^{\frac{2}{p}}}
\right\rceil.
\end{equation}
In particular, 
$\mathcal{N}(\varepsilon)= \mathcal{O}(\varepsilon^{-\frac{2}{p}})$.
\end{enumerate}
\end{theorem}

\begin{proof}
The results are straightforward from Theorem~\ref{th:conv_rate:fun}.
\end{proof}

\begin{remark}
 Even in the absence of strong quasar-convexity, i.e., when $\gamma=0$, Theorem~\ref{th:conv:fun} shows that different values of power $p$ lead to different iteration complexity bounds. The complexity bounds obtained in Theorem \ref{th:conv:fun} can be compared with the classical case with $p=2$. In this case, even for convex functions, PPA attains the iteration complexity $\mathcal{O}(\varepsilon^{-1})$, a rate that has been extended to quasar-convex functions (see \cite[Proposition~25$(d)$]{deBrito2025extending}).

 When $p\in(1,2)$, the resulting complexity $\mathcal{O} (\varepsilon^{-1/(p-1)})$ is slower than the classical $\mathcal{O} (\varepsilon^{-1})$ rate. When $p>2$, the HiPPA method attains the complexity $\mathcal{O} (\varepsilon^{-2/p})$, which is faster than $\mathcal{O} (\varepsilon^{-1})$.
\end{remark}

\section{Numerical experiments}\label{sec:glm_numerical}
In this section, we implement HiPPA and report our numerical results on two prominent applications: (i) Generalized linear model; (ii) Nonconvex robust multi-task regression.
All experiments are conducted in Python on a laptop equipped with a 12th Gen Intel$\circledR$ Core$^{\text{TM}}$ i7-12800H CPU (1.80 GHz) and 16 GB of RAM.

\subsection{Application to generalized linear models}
We first recall the ReLU generalized linear model (GLM) setting and the strong quasar-convexity result that motivates our benchmark. We then present numerical experiments comparing HiPPA with projected first-order methods on the corresponding population loss.

Let $\{(x_i,y_i)\}_{i=1}^N$ be data points, where $x_i\in \R^n$ are sampled i.i.d.\ from a distribution $D_x$, and the responses are generated according to the generalized linear model
\begin{equation}
\label{eq:glm_model}
y_i=\sigma(\langle w_t^\ast,x_i\rangle)+\varepsilon_i,
\end{equation}
where $\sigma:\R\to\R$ is a known monotone activation function, $\varepsilon_i\in\R$ are i.i.d.\ mean-zero noises independent of $x_i$, and $w_t^\ast\in\R^n$ is the unknown ground-truth vector. Denote by $D$ the joint distribution of $(x,y)$.
The empirical squared risk is
\begin{equation}
\label{eq:empirical_risk}
\widehat f(w)=\frac{1}{2N}\sum_{i=1}^N\bigl(y_i-\sigma(\langle w,x_i\rangle)\bigr)^2,
\end{equation}
and the corresponding population risk is
\begin{equation}
\label{eq:population_risk}
f(w)=\frac{1}{2}\mathbb{E}_{(x,y)\sim D}\Bigl[\bigl(y-\sigma(\langle w,x\rangle)\bigr)^2\Bigr].
\end{equation}
We consider the ReLU activation $\sigma(z)=\max\{0,z\}$.
We recall the structural assumption used in \cite{Pun2024}.

\begin{assumption}\label{ass:relu_density}
For every $t=1,\dots,T$ and every $w\neq w_t^\ast$, let $\mathcal P_{w,w_t^\ast}$ denote the marginal distribution of $x$ on the subspace spanned by $w$ and $w_t^\ast$, viewed as a distribution on $\R^2$. Assume that $\mathcal P_{w,w_t^\ast}$ admits a density $p_{w,w_t^\ast}$ such that
$\inf_{\|x\|\le \varepsilon} p_{w,w_t^\ast}(x)\ge \beta$,
for some $\varepsilon>0$ and $\beta>0$.
\end{assumption}

The following result is the main justification for the benchmark used below.
\begin{proposition}[{\cite[Corollary~3]{Pun2024}}]\label{prop:relu_strong_quasar}
Assume Assumption~\ref{ass:relu_density}. Suppose $\|x\|\le c$ almost surely for $x\sim \mathcal P$, with $c\ge \tfrac12$. For $t=1,\dots,T$, define
\begin{equation}\label{eq:Rt_def}
\mathcal R_t = \left\{w\in\R^n: \|w-w_t^\ast\|^2\le \|w_t^\ast\|^2, \ \|w\|\le 2\|w_t^\ast\|\right\}.
\end{equation}
Then $f_t$ is $(\rho,\mu)$-strongly quasar-convex with respect to $w_t^\ast$ over $\mathcal R_t$, for
$\rho=\frac{\varepsilon^4\bs\sin^3(\pi/8)}{8\sqrt{2}\,c}$ and $\mu=c$.
\end{proposition}
We optimize the population objective
\begin{equation}\label{eq:population_relu_experiment}
f(w)=E_{x\sim \mathcal P}\left[\frac{1}{2}\bigl(\ReLU(\langle w,x\rangle)-\ReLU(\langle w_t^\ast,x\rangle)\bigr)^2\right],
\end{equation}
where model is noiseless, namely $y(x)=\ReLU(\langle w^\ast,x\rangle)$.
The distribution $\mathcal P$ is chosen as the uniform distribution on the Euclidean ball
$\mb(0,\sqrt{c})$,
i.e., $\|x\|^2\le c$ almost surely, and $\mathcal P$ is absolutely continuous. This choice is compatible with the assumptions underlying Proposition~\ref{prop:relu_strong_quasar}. Moreover, after every algorithmic update, we project the iterate onto the region $\mathcal{R}$ in \eqref{eq:Rt_def}, so that all iterations remain in the theorem-relevant set.

In our implementation, we use the default parameters $n=100$ and $c=4$. The ground-truth vector is generated by drawing $\widetilde w^\ast\sim \mathcal N(0,I_n)$ and setting $w^\ast=2\,\frac{\widetilde w^\ast}{\|\widetilde w^\ast\|}$.
Thus $\|w^\ast\|=2$. 
The initial point is chosen as $w^0=w^\ast + \eta \|w^\ast\| d$
with $\eta=0.3$, where $d$ is a random unit vector. The point $w^0$ is then projected onto $\mathcal R$.

Since the objective in \eqref{eq:population_relu_experiment} is an expectation, all function values and gradients are computed by Monte Carlo approximation.
For reporting $f(w)$, we use an independent batch of size $M_{\mathrm{eval}}=50{,}000$ and compute
\begin{equation}\label{eq:objective_mc_eval}
f(w) \approx \frac{1}{2M_{\mathrm{eval}}} \sum_{j=1}^{M_{\mathrm{eval}}}\bigl(\ReLU(\langle w,x_j\rangle)-\ReLU(\langle w^\ast,x_j\rangle)\bigr)^2.
\end{equation}
For a sampled batch $\{x_j\}_{j=1}^M$, we use the approximation
\begin{equation}\label{eq:mc_grad_formula}
\nabla f(w)\approx\frac{1}{M}\sum_{j=1}^M\bigl(\ReLU(\langle w,x_j\rangle)-\ReLU(\langle w^\ast,x_j\rangle)\bigr)\,g_{\ReLU}(\langle w,x_j\rangle)\,x_j,
\end{equation}
where $g_{\ReLU}(z)=1$ for $z>0$ and $g_{\ReLU}(0)=0$.
In the implementation, we use $M_{\mathrm{full}}=12{,}000$ for high-accuracy gradient evaluations and $M_{\mathrm{sgd}}=256$ for stochastic gradient steps.


In our comparison, we used the following methods:
\begin{itemize}
    \item \textbf{PGD}: Projected gradient descent 
    \begin{equation}\label{eq:pgd}
w^{k+1}=\Pi_{\mathcal R}\bigl(w^k-\alpha\,\widehat\nabla f(w^k)\bigr),
\end{equation}
where $\Pi_{\mathcal R}$ denotes projection onto $\mathcal R$, and $\widehat\nabla f(w^k)$ is the Monte Carlo gradient estimator from \eqref{eq:mc_grad_formula} using batch size $M_{\mathrm{full}}$. We use the constant step-size $\alpha_{\mathrm{PGD}}=0.12$.

\item \textbf{PSGD}: Projected stochastic gradient descent
\begin{equation}\label{eq:psgd}
w^{k+1}=\Pi_{\mathcal R}\bigl(w^k-\alpha_k\,\widehat\nabla f_{B_k}(w^k)\bigr),
\end{equation}
where $B_k$ is a mini-batch of size $M_{\mathrm{sgd}}=256$. We use a constant step-size $\alpha_{\mathrm{PSGD}}=0.07$.

\item\textbf{HiPPA}: We test two HiPPA variants, corresponding to $p=2$ and $p=3$. We use $\beta=0.8$.
The subproblem is solved numerically using \texttt{L-BFGS-B}. 
\end{itemize}

All methods are allowed to run for at most $2000$ iterations. They stop earlier whenever the relative error satisfies
\begin{equation}
\label{eq:stopping_relerr}
\RelErr(w^k):=\frac{\|w^k-w^\ast\|}{\|w^\ast\|}<10^{-2}.
\end{equation}

The results are displayed in Figure~\ref{fig:all_numerical_results}.
Subfigure~\ref{fig:all_numerical_results}(a) reports the objective value against elapsed time. This comparison is especially informative because each HiPPA iteration requires the solution of an inner optimization problem, while PGD and PSGD only need a gradient step. Even under this more demanding metric, HiPPA-$p=3$ remains the best method.  HiPPA-$p=2$ is also clearly competitive and outperforms both first-order baselines. PGD is substantially faster than PSGD in terms of practical decrease per unit time.
Subfigure~\ref{fig:all_numerical_results}(b) shows the relative error against elapsed time. The same ranking is observed. HiPPA-$p=3$ is the most efficient method in practice, followed by PSGD, then HiPPA-$p=2$,  and finally PGD. 
\begin{figure}[t]
    \centering

    \begin{subfigure}[t]{0.48\textwidth}
        \centering
        \includegraphics[width=\textwidth]{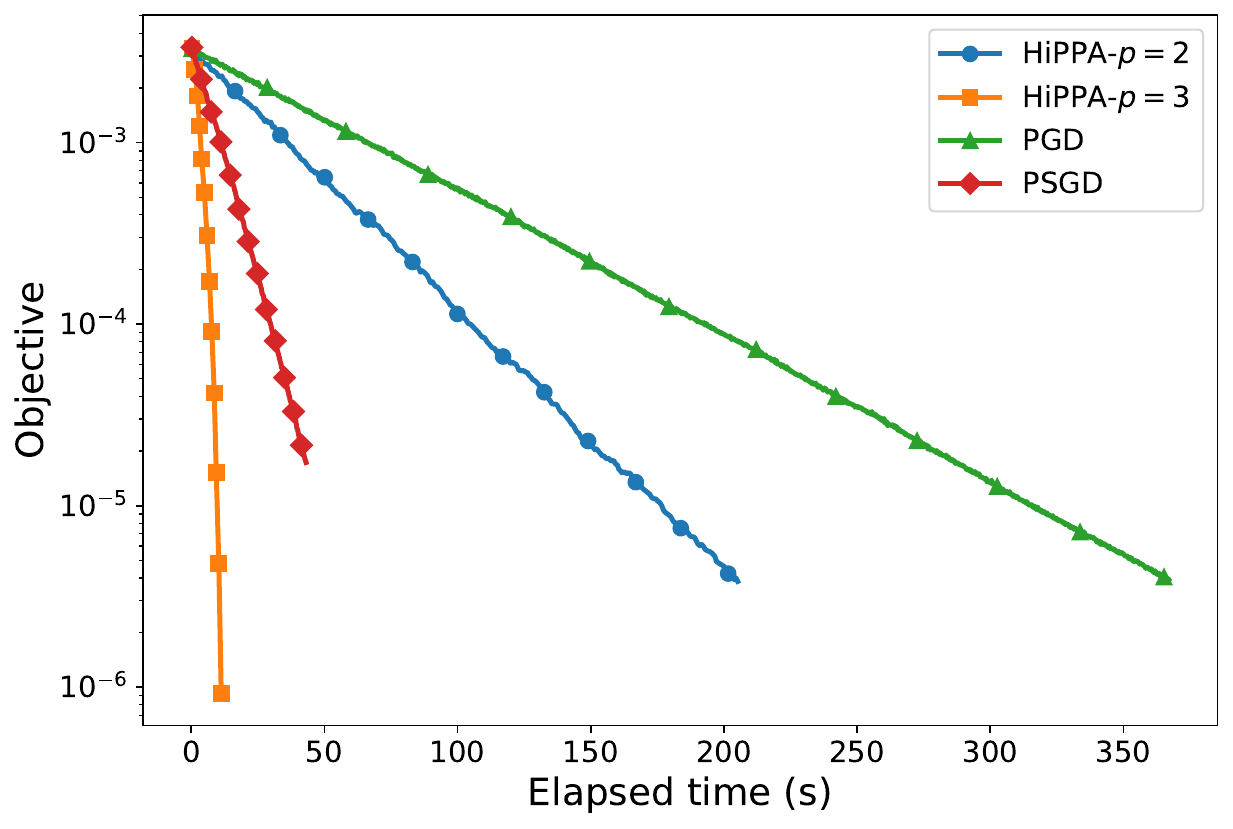}
        \caption{Objective value vs.\ elapsed time.}
        \label{fig:obj_time_combined_sub}
    \end{subfigure}
    \hfill
    \begin{subfigure}[t]{0.48\textwidth}
        \centering
        \includegraphics[width=\textwidth]{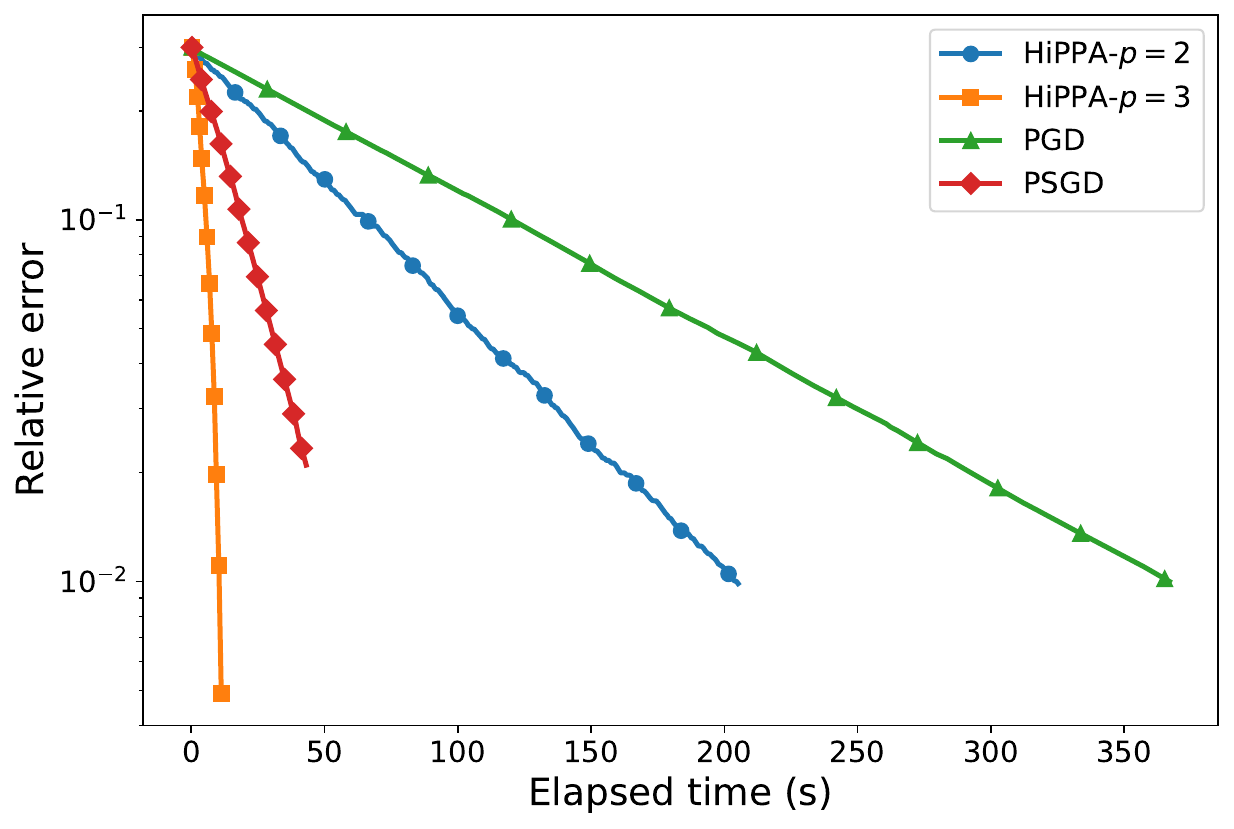}
        \caption{Relative error vs.\ elapsed time.}
        \label{fig:rel_time_combined_sub}
    \end{subfigure}

    \caption{Numerical comparison on the population ReLU-GLM problem.}
    \label{fig:all_numerical_results}
\end{figure}






\subsection{Application to robust multi-task regression}\label{sec:app:reg}
In this section, we consider a nonconvex and nonsmooth robust multi-task regression (RMTR) model and use it as a benchmark for HiPPA. 
We show that, although the model is neither convex nor strongly star-convex, it is strongly quasar-convex on bounded balls around the ground-truth matrix. Then, we describe the numerical implementation used in our experiments.


Let $x_i \in \R^d$ and $y_i \in \R^m$ for $i=1,\dots,N$, and let $W \in \R^{m\times d}$ denote the matrix of regression coefficients. We study the objective
\begin{equation}\label{eq:mtl-model}
\Phi(W):=\left(\frac{1}{N}\sum_{i=1}^N \|W x_i-y_i\|_2\right)^q,
\end{equation}
with $q\in(0,1)$. We assume the realizable model
\begin{equation}\label{eq:mtl-realizable}
y_i=W^\ast x_i,\qquad i=1,\dots,N,
\end{equation}
for some unknown ground-truth matrix $W^\ast \in \R^{m\times d}$. The inner loss
$W \mapsto \frac{1}{N}\sum_{i=1}^N \|W x_i-y_i\|_2$ 
is an $\ell_{2,1}$-type multi-task residual model. It couples the tasks through Euclidean norms of vector residuals and is therefore naturally suited to joint learning settings where several related prediction tasks are learned simultaneously. Such $\ell_{2,1}$-based formulations are widely used in multi-task feature learning and shared-sparsity modeling \cite{LiuJiYe2009Multi-task}. Moreover, robust multi-task regression models are especially relevant when responses may be contaminated by large corruptions or outliers \cite{XuLeng2012}. The additional exponent $q\in(0,1)$ in \eqref{eq:mtl-model} yields a genuinely nonconvex and nonsmooth extension of the classical convex baseline, in the same spirit as nonconvex multi-task feature learning models designed to sharpen convex surrogates \cite{Gong12Multi-Stage}.

Accordingly, \eqref{eq:mtl-model} should be interpreted as a nonconvex robust multi-task regression model. 
Throughout this section, we identify $\R^{m\times d}$ with the Euclidean space $\R^{md}$ endowed with the Frobenius norm
\[
\|W\|_F := \left(\sum_{j=1}^m \sum_{\ell=1}^d W_{j\ell}^2\right)^{1/2}.
\]
We also write $X := [x_1,\dots,x_N] \in \R^{d\times N}$
for the sample matrix and assume that $X$ has full row rank. For $R>0$, define
\begin{equation}\label{eq:KR}
K_R := \{W\in\R^{m\times d} : \|W-W^\ast\|_F \le R\}.
\end{equation}
The next proposition shows that the model \eqref{eq:mtl-model} is strongly quasar-convex on every bounded ball centered at $W^\ast$.
\begin{proposition}\label{prop:mtl-sqc}
Let $q\in(0,1)$, $W^\ast\in\R^{m\times d}$, and let $X=[x_1,\dots,x_N]\in\R^{d\times N}$ have full row rank.
Assume the model~\eqref{eq:mtl-realizable} and define the function $\Phi$ as \eqref{eq:mtl-model}.
Fix $R>0$ and let $K_R$ be given by~\eqref{eq:KR}. Then, for every $\kappa\in(0,q)$, the function $\Phi$ is $(\kappa,\gamma)$-strongly quasar-convex on $K_R$ with respect to $W^\ast$, where
\begin{equation}\label{eq:gamma-mtl}
\gamma=\frac{2(q-\kappa)}{\kappa}\,c_X^q\,R^{q-2},
\end{equation}
with
\[
c_X:=\mathop{\bs\min}\limits_{\|U\|_F=1}\frac{1}{N}\sum_{i=1}^N \|U x_i\|_2>0.
\]
Moreover, $\Phi$ is neither convex nor strongly star-convex.
\end{proposition}
\begin{proof}
See Appendix~\ref{sec:app:a}.
\end{proof}

We now describe the numerical implementation associated with \eqref{eq:mtl-model}. 
Unless stated otherwise, we use the default parameters $d=100$, $m=5$, and $N=400$.
The ground-truth matrix is generated by drawing $\widetilde W^\ast$ with i.i.d.\ standard Gaussian entries and setting $W^\ast=2\,\frac{\widetilde W^\ast}{\|\widetilde W^\ast\|_F}$.
Hence $\|W^\ast\|_F=2$. The covariates $x_i\in\R^d$ are drawn independently from the standard Gaussian distribution. The responses are then generated through the realizable model
\[
y_i=W^\ast x_i,\qquad i=1,\dots,N.
\]
The initial point is chosen as $W^0=\sigma \widetilde W$,
where $\widetilde W$ has i.i.d.\ standard Gaussian entries and $\sigma=20$. Thus, the initialization is genuinely random and is not restricted to a neighborhood of the ground-truth matrix.
We define $\Psi$ as \eqref{eq:eq:Psi-mtl}.


In our comparison, we used the following methods:
\begin{itemize}
    \item \textbf{PSG}: Projected subgradient method with stepsize $\alpha_k=\frac{\alpha_0}{\sqrt{k}}$ with $\alpha_0=0.8$ \cite{beck2017first}.
    \item \textbf{PSSG} Projected stochastic subgradient method with  mini-batches $B_k$ of size $32$. We take $\alpha_k=\frac{\alpha_0}{\sqrt{k}}$ with $\alpha_0=0.8$ \cite{lacoste2012simpler}.
    \item \textbf{HiPPA}: We test two HiPPA variants corresponding to $p=2$ and $p=3$. The subproblem is solved approximately through a smoothed surrogate and numerically with L-BFGS-B. We set $\beta=0.05$.
\end{itemize}


All methods are allowed to run for at most $2000$ iterations. They stop earlier whenever the relative error satisfies
\begin{equation}
\label{eq:stopping_relerr}
\RelErr(W^k):=\frac{\|W^k-W^\ast\|_F}{\|W^\ast\|_F}<10^{-4}.
\end{equation}

Table~\ref{tab:mtl-results} and Figure~\ref{fig:mtl-results} show that HiPPA significantly outperforms the subgradient-based baselines. In particular, both HiPPA variants reach high-accuracy solutions much faster than PSG and PSSG in terms of wall-clock time. 
This behavior is clearly reflected in Figure~\ref{fig:mtl-results}, where the objective value and the relative error for HiPPA decrease rapidly, while PSG exhibits a slower decay and PSSG quickly stagnates. Overall, HiPPA demonstrates a clear advantage over standard first-order nonsmooth methods on this benchmark. Among the two variants, $p=3$ shows the fastest convergence in time.


\begin{table}[ht]
\centering
\caption{Performance comparison on the nonconvex multi-task regression benchmark.}
\label{tab:mtl-results}
\begin{tabular}{lcccc}
\hline
Method & Iteration to stop & Time (s) & Relative error & Objective value \\
\hline
PSG & 296 & 0.0899 & $1.858147\times 10^{-5}$ & $6.671428\times 10^{-3}$ \\
PSSG & 2000 & 0.1228 & $3.557654\times 10^{-2}$ & $2.605372\times 10^{-1}$ \\
HiPPA-\(p=2\) & 14 & 0.0593 & $7.077845\times 10^{-8}$ & $2.657486\times 10^{-4}$ \\
HiPPA-\(p=3\) & 4 & 0.0283 & $3.680355\times 10^{-9}$ & $8.169149\times 10^{-5}$ \\
\hline
\end{tabular}
\end{table}

\begin{figure}[t]
    \centering

    \begin{subfigure}[t]{0.48\textwidth}
        \centering
        \includegraphics[width=\textwidth]{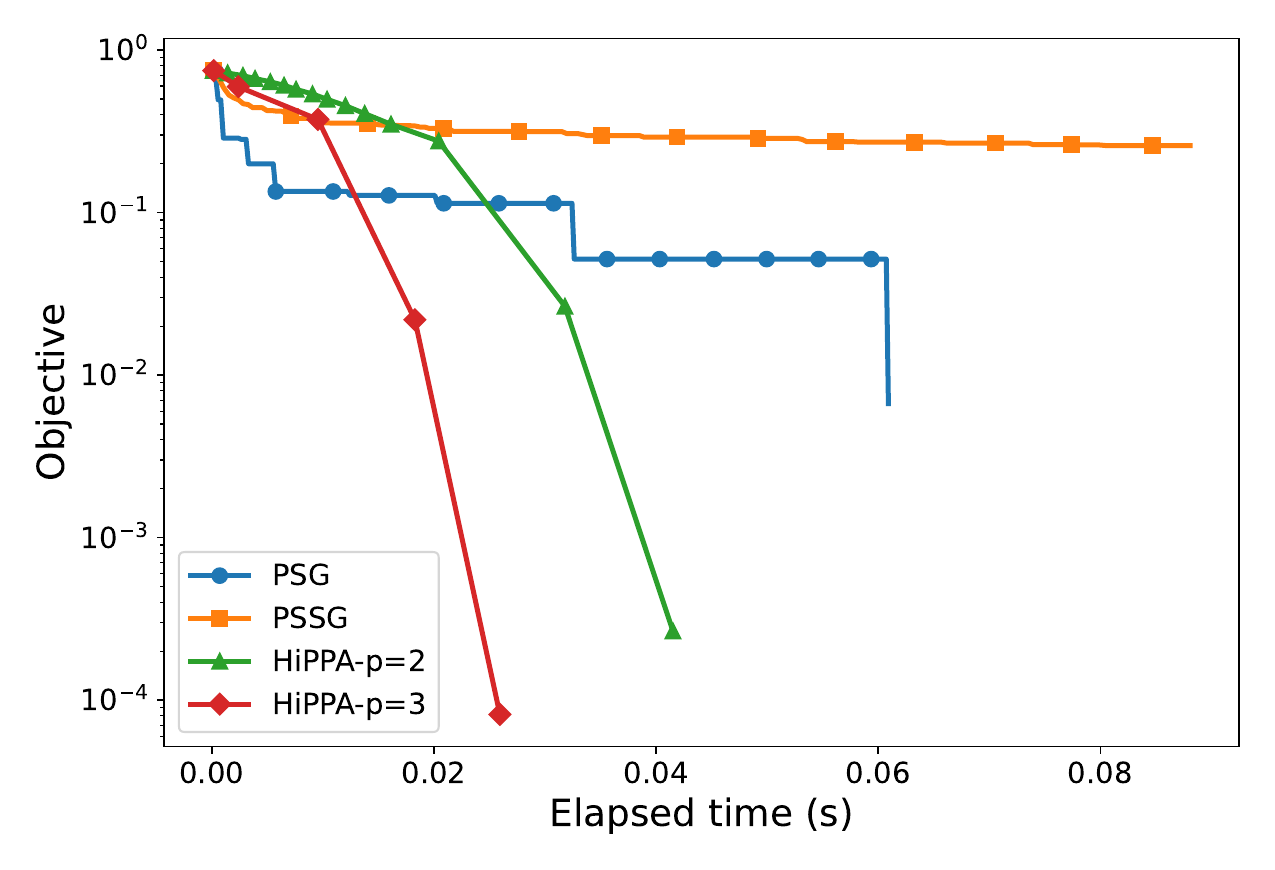}
        \caption{Objective value vs.\ elapsed time.}
    \end{subfigure}
    \hfill
    \begin{subfigure}[t]{0.48\textwidth}
        \centering
        \includegraphics[width=\textwidth]{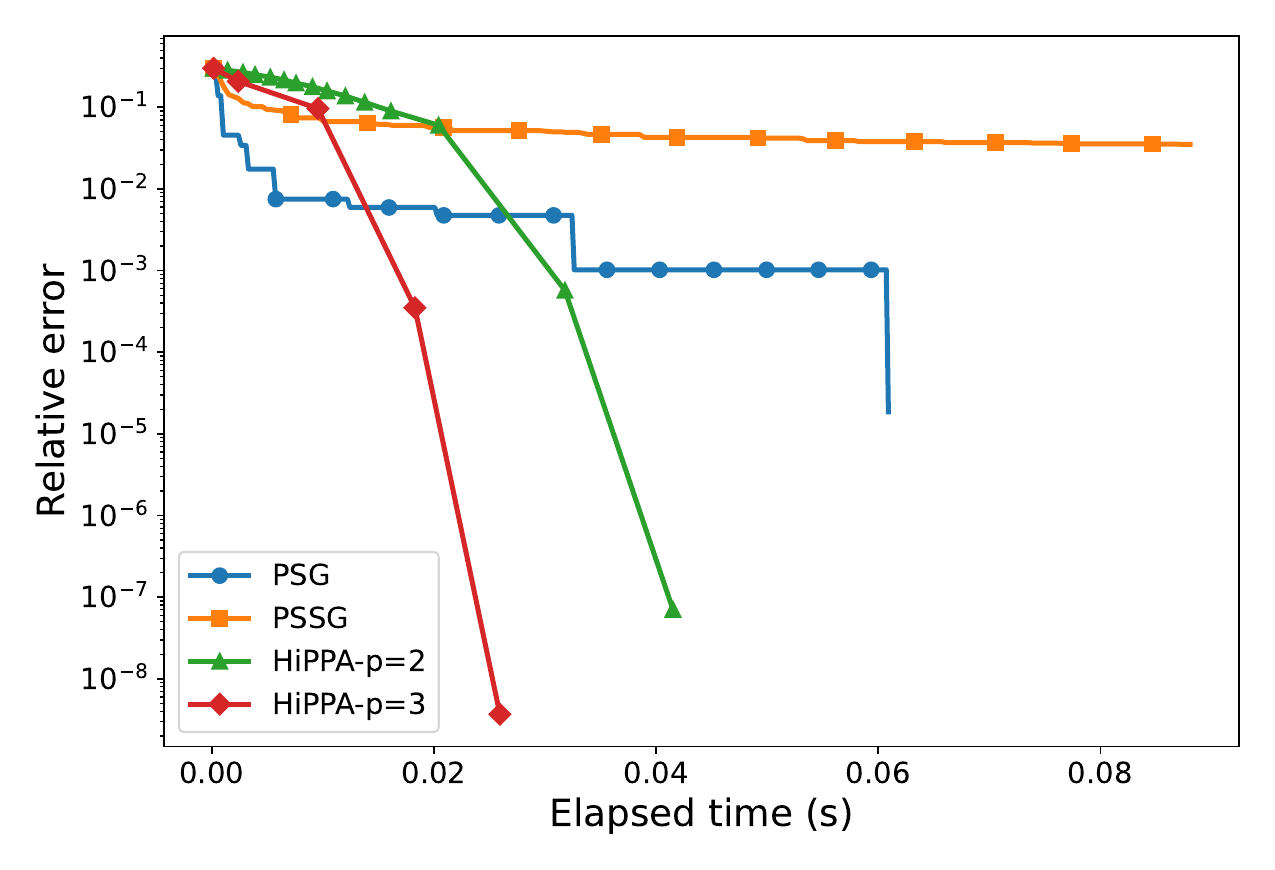}
        \caption{Relative error vs.\ elapsed time.}
    \end{subfigure}

    \caption{Numerical comparison on the nonconvex multi-task regression problem. }
    \label{fig:mtl-results}
\end{figure}





\section{Conclusion}\label{sec:conclusion}
In this paper, we considered the optimization of (strongly) quasar-convex functions. As a first step, we studied the fundamental properties of this class of functions, which pave the way for its applicability across many domains, including machine learning and data science. Next, we developed and analyzed high-order proximal-point methods (HiPPA with $p>1$) for this class of nonconvex optimization and established conditions ensuring convergence to global minimizers. We further derived convergence rates and complexity bounds, showing local linear convergence for $p\in(1,2)$, linear convergence for $p=2$. and superlinear convergence for $p>2$. Preliminary numerical results on a generalized linear model support the efficiency of our approach.

\vspace{4mm}
\noindent \textbf{Acknowledgments}
 MA and AK were partially supported by the Research Foundation Flanders (FWO) research project G081222N, UA BOF DocPRO4 project with ID 46929, and FWO travel grants K153324N and K152524N. FL was partially supported by ANID--Chile under project Fondecyt Regular 1241040. JY was partially supported by NSFC grant 12171087.

\appendix

\section{Proof of Theorem~\ref{thm:char:clarke}}\label{sec:app:a0}
\begin{proof}
    Assume first that $h$ is $(\kappa,\gamma)$-strongly quasar-convex with respect to $\overline x$. Fix 
 $x\in \dom{h}$, and let $L_x>0$ be a Lipschitz constant of $h$ on an open neighborhood $V_x$ of $x$. For $y \in V_x$ and $\lambda>0$ sufficiently small, we have $y+\lambda(\ov{x}-x)\in V_x$ and
\[
y_\lambda := y + \lambda (\overline{x} - y) = \lambda \overline{x} + (1-\lambda)y \in V_x.
\]
By the $(\kappa,\gamma)$-strong quasar-convexity of $h$, we obtain
\begin{equation}\label{eq:sqc-y}
 h(y_\lambda) \le \kappa \lambda h(\overline x) + (1-\kappa \lambda) h(y) - \lambda\Big( 1- \frac{\lambda}{2-\kappa} \Big) \frac{\kappa\gamma}{2} \|y - \overline x\|^2.
\end{equation}
Applying the Lipschitz continuity of $h$ on $V_x$ implies
\begin{equation}\label{lipschitz}
 h(y + \lambda (\overline{x}-x))-h(y_\lambda) \leq L_x \big\| y + \lambda (\overline{x}-x) - y_\lambda \| = L_x \lambda \|x-y\|.
\end{equation}
Adding \eqref{eq:sqc-y} and \eqref{lipschitz} yields
\begin{equation*}
 h(y+\lambda(\overline{x}-x)) \le \kappa\lambda h(\overline x) + (1-\kappa\lambda) h(y) - \lambda \Big( 1- \frac{\lambda}{2-\kappa} \Big) \frac{\kappa \gamma}{2} \|y-\overline x\|^2 + L_x \lambda \|y-x\|.
\end{equation*}
Subtracting $h(y)$ from both sides and dividing by $\lambda>0$, we obtain
\begin{equation}\label{eq:inc-dir-y}
 \frac{h(y+\lambda(\overline x-x))-h(y)}{\lambda} \le \kappa \big( h(\overline x)-h(y) \big) - \Big( 1-\frac{\lambda}{2-\kappa} \Big) \frac{\kappa \gamma}{2} \|y-\overline x\|^2 + L_x\|y-x\|.
\end{equation}
Taking $\limsup_{\lambda \downarrow 0,\, y \to x}$ and using the definition of the Clarke directional derivative, we get
$$ h^{\circ} (x; \overline{x} - x) \le \kappa \big( h(\overline x) - h(x) \big) - \frac{\kappa\gamma}{2} \|x-\overline x\|^2.$$
Together with \eqref{prop:direc}, this yields, for every $v \in \partial^C h(x)$,
$$ \langle v, \overline{x} - x \rangle \le \kappa \big( h(\overline{x}) - h(x) \big) - \frac{\kappa \gamma}{2} \|x-\overline x\|^2,
$$
which is equivalent to \eqref{eq:foC}.

Conversely, assume that \eqref{eq:foC} holds. Fix $x \in \dom{h}$, and define
$x_\lambda := \lambda \overline x+(1-\lambda)x$, and $g(\lambda) := h(x_\lambda) - h(\overline x)$ for
$\lambda \in [0,1]$.
Applying \eqref{eq:foC} at the point $x_\lambda$, we obtain, for every $v\in \partial^C h(x_\lambda)$,
\[
\langle v,\overline x-x_\lambda\rangle \le - \kappa g(\lambda) - \frac{\kappa \gamma}{2} \| x_\lambda - \overline x\|^2,
\]
Since $\overline x-x_\lambda=(1-\lambda)(\overline x-x)$, it follows that 
$$ \langle v, \overline x-x\rangle \le
-\frac{\kappa}{1 - \lambda} g(\lambda)-\frac{\kappa \gamma}{2} (1 - \lambda) \|x - \overline x\|^2.$$
By Lemma \ref{lem:clarke-chain}, for almost every $\lambda \in (0,1)$, there exists $\xi_\lambda \in \partial^C h(x_\lambda)$ such that $ g'(\lambda) = \langle \xi_\lambda, \overline x-x \rangle$. Hence, for almost every $\lambda \in (0,1)$,
\begin{equation} \label{eq:gprime-mod}
 g'(\lambda) \le - \frac{\kappa}{1-\lambda} g(\lambda) - \frac{\kappa \gamma}{2}(1-\lambda) \|x-\overline x\|^2.
\end{equation}
Since $g$ is absolutely continuous  and $\lambda\mapsto (1-\lambda)^{-\kappa}$ is smooth on $[0,1)$,  the function $\lambda\mapsto(1-\lambda)^{-\kappa} g(\lambda)$ is also absolutely continuous over every compact subinterval of $[0,1)$.
Applying \eqref{eq:gprime-mod}, for almost every $\lambda\in(0,1)$, we get
\begin{equation}\label{to:integ}
\frac{d}{d\lambda} \big((1-\lambda)^{-\kappa} g(\lambda)\big)=\kappa(1-\lambda)^{-\kappa-1}g(\lambda)+(1-\lambda)^{-\kappa}g'(\lambda)
\le - \frac{\kappa \gamma}{2} (1-\lambda)^{1-\kappa} \|x - \overline{x} \|^2.
\end{equation}
Integrating \eqref{to:integ} from $\lambda$ to $r\in(\lambda,1)$ yields
\[
(1-r)^{-\kappa}g(r)-(1-\lambda)^{-\kappa}g(\lambda)
\le -\frac{\kappa\gamma}{2}\int_\lambda^r (1-s)^{1-\kappa}\,ds\,\|x-\overline x\|^2,
\]
i.e.,
\[
(1-\lambda)^{-\kappa}g(\lambda)
\ge (1-r)^{-\kappa}g(r)+\frac{\kappa\gamma}{2}\int_\lambda^r (1-s)^{1-\kappa}\,ds\,\|x-\overline x\|^2.
\]
Since $\overline x\in \argmin{\mathbb R^n} h$ and $x_r\in \dom h$, we have $g(r)=h(x_r)-h(\overline x)\ge 0$. Therefore,
\[
(1-\lambda)^{-\kappa} g(\lambda) \ge \frac{\kappa \gamma}{2(2-\kappa)} \Big((1-\lambda)^{2-\kappa} - (1-r)^{2-\kappa} \Big) \|x-\overline{x}\|^2.
\]
Letting $r \uparrow 1$ in the previous equation ensures
\[
g(\lambda) \ge \frac{\kappa \gamma}{2 (2-\kappa)} (1-\lambda)^2 \|x-\overline x\|^2.
\]
Setting $\lambda=0$ leads particularly to
\begin{align}\label{eq:g0-lower}
g(0) \ge \frac{\kappa \gamma}{2(2-\kappa)} \|x-\overline x\|^2. 
\end{align} 
Integrating \eqref{to:integ} from $0$ to $\lambda$ yields
\begin{equation} \label{eq:g-upper}
 g(\lambda) \le (1-\lambda)^\kappa g(0) - \frac{ \kappa \gamma}{2 (2-\kappa)} ( (1-\lambda)^{\kappa} - (1-\lambda)^2) \|x-\overline{x} \|^2,
\end{equation}
i.e.,\[
g(\lambda) \le (1-\lambda)^\kappa \left( g(0)-\frac{\kappa\gamma}{2 (2-\kappa)} \|x-\overline x\|^2 \right) + \frac{\kappa \gamma}{2 (2-\kappa)} (1-\lambda)^2 \|x-\overline x\|^2.
\]
By \eqref{eq:g0-lower}, the term in parentheses is nonnegative. 
By Bernoulli's inequality, $(1-\lambda)^\kappa \le 1-\kappa \lambda$ for every $\kappa \in (0,1]$ and every $\lambda \in [0,1]$, i.e.,
\begin{align}
 g(\lambda) &\le (1-\kappa \lambda) \left( g(0) - \frac{\kappa \gamma}{2 (2-\kappa)} \|x-\overline x\|^2 \right) + \frac{\kappa\gamma}{2(2-\kappa)}(1-\lambda)^2 \|x-\overline x\|^2 \notag \\
 & =  (1-\kappa \lambda)g(0) - \frac{ \kappa \gamma}{2(2-\kappa)} ((1 - \kappa \lambda ) - (1 - \lambda)^2) \|x-\overline x\|^2. \label{to:conclude}
\end{align} 
A direct computation shows that
\[
(1 - \kappa \lambda) - (1-\lambda)^2
= (2-\kappa)\lambda\left(1-\frac{\lambda}{2-\kappa}\right).
\]
Substituting into \eqref{to:conclude} and recalling that
$g(0)=h(x)-h(\overline x)$ leads to
\[
h(\lambda\overline x + (1-\lambda)x) \le \kappa \lambda h(\overline x) + (1 - \kappa \lambda) h(x) - \lambda \left(1-\frac{\lambda}{2-\kappa} \right) \frac{\kappa\gamma}{2}\|x-\overline x\|^2,
\]
thereby establishing the $(\kappa,\gamma)$-strong quasar-convexity of $h$ with respect to $\overline{x}$.
\end{proof}

\section{Proof of Proposition~\ref{prop:mtl-sqc}}
\label{sec:app:a}
\begin{proof}
Let us define
\begin{equation}\label{eq:eq:Psi-mtl}
\Psi(W):=\frac{1}{N}\sum_{i=1}^N \|W x_i-y_i\|_2=\frac{1}{N}\sum_{i=1}^N \|(W-W^\ast)x_i\|_2,
\end{equation}
i.e., $\Phi(W)=\Psi(W)^q$. We first show that $W^\ast$ is the unique minimizer of $\Phi$. Since $y_i=W^\ast x_i$, we have $\Phi(W^\ast)=0$. If $\Phi(W)=0$, then $\Psi(W)=0$, and therefore
\[
(W-W^\ast)x_i=0,\qquad i=1,\ldots,N,
\]
i.e., $(W-W^\ast)X=0$. Since the matrix $X$ has full row rank, it follows that $W=W^\ast$. Hence $\argmin{} \Phi=\{W^\ast\}$.
Next, let $\lambda\in[0,1]$ and define $W_\lambda:=\lambda W^\ast+(1-\lambda)W$. Then, it follows that
\[
\Psi(W_\lambda)=\frac{1}{N}\sum_{i=1}^N \|(W_\lambda-W^\ast)x_i\|_2=(1-\lambda)\Psi(W),
\]
i.e.,
\begin{equation}\label{eq:radial-identity-mtl}
\Phi(W_\lambda)=(1-\lambda)^q\Phi(W).
\end{equation}
Now, we consider the function
\[
U\mapsto \frac{1}{N}\sum_{i=1}^N \|U x_i\|_2,
\]
on the compact unit sphere $\{U\in\R^{m\times d}:\|U\|_F=1\}$, which is continuous, i.e., it attains its minimum. If the minimum is zero, there exists $U\neq 0$ with $\|U\|_F=1$ such that
\[
U x_i=0,\qquad i=1,\dots,N,
\]
or equivalently $UX=0$. Since $X$ has full row rank, its left kernel is trivial, which is impossible, i.e.,
\[
c_X:=\mathop{\bs\min}\limits_{\|U\|_F=1}\frac{1}{N}\sum_{i=1}^N \|U x_i\|_2>0.
\]
Now, let $W\neq W^\ast$ and let us define
\[
U:=\frac{W-W^\ast}{\|W-W^\ast\|_F}.
\]
where $\|U\|_F=1$, and it follows from the definition of $c_X$ that
\[
\Psi(W)=\|W-W^\ast\|_F\,\frac{1}{N}\sum_{i=1}^N \|U x_i\|_2\ge c_X\|W-W^\ast\|_F,
\]
i.e., $\Phi(W)\ge c_X^q\|W-W^\ast\|_F^q$. If $W\in K_R$, then $\|W-W^\ast\|_F\le R$, and since $q-2<0$,
\[
\|W-W^\ast\|_F^q=\|W-W^\ast\|_F^2\|W-W^\ast\|_F^{q-2}\ge R^{q-2}\|W-W^\ast\|_F^2,
\]
leading to
\begin{equation}\label{eq:quad-lb-mtl}
\Phi(W)\ge c_X^q R^{q-2}\|W-W^\ast\|_F^2, \qquad \forall W\in K_R.
\end{equation}
Now, let us fix $\kappa\in(0,q)$ and let $\lambda\in[0,1]$. Since $q\in(0,1)$, the function $t\mapsto t^q$ is concave on $[0,1]$, and thus
$(1-\lambda)^q\le 1-q\lambda$.
Together with \eqref{eq:radial-identity-mtl}, this implies
\[
\Phi(W_\lambda)\le (1-q\lambda)\Phi(W)=(1-\kappa\lambda)\Phi(W)-(q-\kappa)\lambda\Phi(W).
\]
Using \eqref{eq:quad-lb-mtl}, it follows that
\[
\Phi(W_\lambda)\le(1-\kappa\lambda)\Phi(W)-(q-\kappa)\lambda c_X^q R^{q-2}\|W-W^\ast\|_F^2.
\]
Since $1-\frac{\lambda}{2-\kappa}\le 1$, the choice of $\gamma$ in \eqref{eq:gamma-mtl} yields
\[
(q-\kappa)\lambda c_X^q R^{q-2}\|W-W^\ast\|_F^2\ge\lambda\left(1-\frac{\lambda}{2-\kappa}\right)\frac{\kappa\gamma}{2}\|W-W^\ast\|_F^2,
\]
i.e.,
\[
\Phi(\lambda W^\ast+(1-\lambda)W)\le\kappa\lambda\Phi(W^\ast)+(1-\kappa\lambda)\Phi(W)-
\lambda\left(1-\frac{\lambda}{2-\kappa}\right)\frac{\kappa\gamma}{2}\|W-W^\ast\|_F^2.
\]
Since $\Phi(W^\ast)=0$, this is exactly the definition of $(\kappa,\gamma)$-strong quasar-convexity on $K_R$ with respect to $W^\ast$.
Finally, let $W\neq W^\ast$ and $\lambda\in(0,1)$. By \eqref{eq:radial-identity-mtl}, it follows that
\[
\Phi(\lambda W^\ast+(1-\lambda)W)=(1-\lambda)^q\Phi(W).
\]
Since $q\in(0,1)$, we have $(1-\lambda)^q>(1-\lambda)$, i.e.,
\[
\Phi(\lambda W^\ast+(1-\lambda)W)>(1-\lambda)\Phi(W)=\lambda\Phi(W^\ast)+(1-\lambda)\Phi(W).
\]
Therefore, $\Phi$ is not convex. In particular, it is not star-convex with respect to $W^\ast$, and therefore cannot be strongly star-convex.
\end{proof}

\bibliographystyle{siamplain}
\bibliography{references.bib}

\end{document}